\documentclass[a4paper,10pt
]{article}%

\usepackage[dvipsnames]{xcolor}

\usepackage{xr-hyper}
\usepackage[pagebackref, colorlinks, citecolor=PineGreen, linkcolor=PineGreen]{hyperref}
\hypersetup{
  final,
  pdftitle={Homotopy theory of equivariant operads with fixed colors},
  pdfauthor={Bonventre, P. and Pereira, L. A.},
  linktoc=page
}

\usepackage{amsmath, amsthm}

\usepackage[latin1]{inputenc}%
\usepackage{MnSymbol}

\usepackage[normalem]{ulem}
\usepackage{bbm}

\DeclareMathAlphabet\mathbb{U}{msb}{m}{n}
\usepackage{upgreek}
\usepackage{mathrsfs}

\usepackage[normalem]{ulem}

\usepackage[inline,shortlabels]{enumitem}
\setenumerate{label=(\roman*)}

\usepackage[final]{microtype}
\usepackage{relsize}
\usepackage{geometry}

\usepackage{tikz}%
\usetikzlibrary{matrix,arrows,decorations.pathmorphing,
cd,patterns,calc}
\tikzset{%
  treenode/.style = {shape=rectangle, rounded corners,%
                     draw, align=center,%
                     top color=white, bottom color=blue!20},%
  root/.style     = {treenode, font=\Large, bottom color=red!30},%
  env/.style      = {treenode, font=\ttfamily\normalsize},%
  dummy/.style    = {circle,draw,inner sep=0pt,minimum size=2mm}%
}%

\usetikzlibrary[decorations.pathreplacing]

\makeatletter

\def\@testdef #1#2#3{%
  \def\reserved@a{#3}\expandafter \ifx \csname #1@#2\endcsname
  \reserved@a  \else
  \typeout{^^Jlabel #2 changed:^^J%
    \meaning\reserved@a^^J%
    \expandafter\meaning\csname #1@#2\endcsname^^J}%
  \@tempswatrue \fi}

\makeatother

\numberwithin{equation}{section} 
\numberwithin{figure}{section}

\usepackage{mathtools}
\mathtoolsset{showonlyrefs,showmanualtags} 

\newtheorem*{theorem*}{Theorem}%
\newtheorem{lemma}[equation]{Lemma}%
\newtheorem{proposition}[equation]{Proposition}%
\newtheorem{corollary}[equation]{Corollary}%
\newtheorem*{conjecture*}{Conjecture}%
\newtheorem{innercustomgeneric}{\customgenericname}
\providecommand{\customgenericname}{}
\newcommand{\newcustomtheorem}[2]{%
  \newenvironment{#1}[1]
  {%
   \renewcommand\customgenericname{#2}%
   \renewcommand\theinnercustomgeneric{##1}%
   \innercustomgeneric
  }
  {\endinnercustomgeneric}
}

\newcustomtheorem{customthm}{Theorem}
\newcustomtheorem{customcor}{Corollary}

\theoremstyle{definition} 
\newtheorem{definition}[equation]{Definition}%
\newtheorem*{definition*}{Definition}%
\newtheorem{example}[equation]{Example}%
\newtheorem{remark}[equation]{Remark}%
\newtheorem{notation}[equation]{Notation}%
\newcommand{\set}[1]{\left\{#1\right\}}%
\newcommand{\sets}[2]{\left\{ #1 \;|\; #2\right\}}%
\newcommand{\longto}{\longrightarrow}%
\newcommand{\into}{\hookrightarrow}%

\usepackage{harpoon}
\newcommand{\vect}[1]{\text{\overrightharp{\ensuremath{#1}}}}

\newcommand{\Sym}{\ensuremath{\mathsf{Sym}}}%
\newcommand{\Set}{\ensuremath{\mathsf{Set}}}

\newcommand{\sSet}{\ensuremath{\mathsf{sSet}}}%
\newcommand{\Cat}{\mathsf{Cat}}

\newcommand{\Op}{\mathsf{Op}}%
\newcommand{\sOp}{\ensuremath{\mathsf{sOp}}}%

\newcommand{\Fun}{\mathsf{Fun}}

\newcommand{\Alg}{\mathsf{Alg}}

\DeclareMathOperator{\colim}{colim}%
\DeclareMathOperator{\Lan}{Lan}%
\DeclareMathOperator{\Ran}{Ran}%

\DeclareMathOperator{\Aut}{Aut}%

\DeclareMathOperator{\Iso}{Iso}

\newcommand{\F}{\ensuremath{\mathcal F}}
\newcommand{\V}{\ensuremath{\mathcal V}}

\renewcommand{\O}{\ensuremath{\mathcal O}}
\renewcommand{\P}{\ensuremath{\mathcal P}}
\newcommand{\C}{\ensuremath{\mathcal C}}

\newcommand{\G}{\ensuremath{\mathcal G}}

\newcommand{\SC}{\Sigma_{\mathfrak C}}
\newcommand{\OC}{\Omega_{\mathfrak C}}

\title{Homotopy theory of equivariant operads with fixed colors}

\author{Peter Bonventre, Lu\'is A. Pereira}%

\date{\today}

\begin{document}

\maketitle

\begin{abstract}
	We build model structures 
	on the category of equivariant simplicial operads with a fixed set of colors,
	with weak equivalences determined by families of subgroups.
	In particular, by specifying to the family of graph subgroups 
	(or, more generally, one of the indexing systems of Blumberg-Hill),
	we obtain model structures on the category of equivariant simplicial operads with a fixed set of colors,
	with weak equivalences determined by norm map data.
\end{abstract}

\tableofcontents

\section{Introduction}

This paper follows \cite{Per18}, \cite{BP21}, \cite{BP20}
as part of a larger project studying 
\emph{equivariant operads with norm maps}.
Here, norm maps are an extra piece of data 
(not present non-equivariantly)
which must be considered
when studying equivariant operads,
the importance of which was made clear by
Hill-Hopkins-Ravenel 
in their solution to the Kervaire invariant one problem \cite{HHR16}.

For concreteness, let us fix a finite group $G$
and consider the category
$\mathsf{sOp}^{G}_{\**} = \mathsf{Op}_{\**}(\mathsf{sSet}^G)$
or, in words,
the category of single-colored operads
in the category $\mathsf{sSet}^G$
of simplicial sets with a $G$-action.
We note that, for $\O \in \mathsf{sOp}^{G}_{\**}$,
the $n$-th operadic level $\O(n)$ has both a $\Sigma_n$-action and a $G$-action, commuting with each other, or, equivalently, 
a $G \times \Sigma_n$-action.
One of the key upshots
of Blumberg and Hill's work \cite{BH15}
is that the preferred notion of weak equivalence in $\mathsf{sOp}^{G}_{\**}$
is that of \emph{graph equivalence},
i.e. those maps 
$\O \to \mathcal{P}$
such that the fixed point maps
\begin{equation}\label{GRAPHEQ EQ}
\O(n)^{\Gamma} \xrightarrow{\sim} \mathcal{P}(n)^{\Gamma}
\qquad
\text{for }
\Gamma \leq G \times \Sigma_n
\text{ such that }
\Gamma \cap \Sigma_n = \**
\end{equation}
are Kan equivalences in $\mathsf{sSet}$.
Here, the term ``graph'' comes from a neat characterization of the $\Gamma$
as in \eqref{GRAPHEQ EQ}:
such a $\Gamma$ is necessarily the graph of a partial homomorphism
$\phi \colon H \to \Sigma_n$ for some subgroup $H \leq G$,
i.e.
$\Gamma = \sets{(h,\phi(h))}{h \in H}$. 
Note that one hence has a canonical isomorphism $\Gamma \simeq H$.
Briefly, the need to consider such \emph{graph subgroups} $\Gamma$ comes from the study of algebras.
Suppose $X \in \mathsf{sSet}^G$ is an algebra over
$\O$,
so that one has algebra multiplication maps
as on the left below
\begin{equation}\label{ALGNORM EQ}
\O(n) \times X^n \to X
\qquad \qquad
\O(n)^{\Gamma} \times N_{\Gamma}X \to X
\end{equation}
which are required to be 
$G \times \Sigma_n$-equivariant
(where the target $X$ is given the trivial $\Sigma_n$-action).
One then has induced 
$H$-equivariant maps 
on the right in \eqref{ALGNORM EQ},
where the \emph{norm object} $N_{\Gamma} X$
denotes $X^{n}$ with the $H$-action determined by $\Gamma \simeq H$.
In particular, each point
$\rho \in \O(n)^{\Gamma}$
determines a $H$-equivariant \emph{norm map}
$\rho \colon N_{\Gamma} X \to X$,  
and such maps turn out to be a key piece of data
for algebras in the equivariant context.
The reason to prefer the graph equivalences in \eqref{GRAPHEQ EQ}
is then to ensure that weakly equivalent operads
have equivalent ``spaces of norm maps'' $\O(n)^{\Gamma}$.

The existence of a model structure 
on $\mathsf{sOp}^G_{\**}$
with weak equivalences given by the graph equivalences
in \eqref{GRAPHEQ EQ}
was established independently as a particular case of either
\cite[Thm. I]{BP21} or \cite[Thm. 3.1]{GW18}.
It seems worth noting that these results are somewhat non formal.
On the one hand, using the description 
$\mathsf{sOp}_{\**}^G \simeq \mathsf{Op}_{\**}(\mathsf{sSet}^G)$
as single-colored operads enriched in 
$\mathsf{sSet}^G$,
one could obtain a model structure
on $\mathsf{sOp}_{\**}^G$
by applying \cite[Thm. 3.2]{BM03}
to the genuine/strong model structure on $\mathsf{sSet}^G$.
Alternatively, one also has an identification
$\mathsf{sOp}_{\**}^G \simeq 
\left(\mathsf{Op}_{\**}(\mathsf{sSet})\right)^G$
as $G$-objects on the category
$\mathsf{Op}_{\**}(\mathsf{sSet})$
of single-colored simplicial operads,
so one could likewise build a model structure on 
$\mathsf{sOp}_{\**}^G$
by applying \cite[Prop. 2.6]{Ste16}
to the category
$\mathsf{Op}_{\**}(\mathsf{sSet})$
with its usual model structure.
However, neither of these approaches
recovers the graph equivalences in \eqref{GRAPHEQ EQ},
instead leading to a weaker notion of equivalence for which the fixed points in \eqref{GRAPHEQ EQ}
only need to be weak equivalences for
$\Gamma \leq G \leq G \times \Sigma_n$.

The main result in this paper, 
Theorem \ref{THMI}, 
extends the graph equivalence model structures of
\cite[Thm. I]{BP21}, \cite[Thm. 3.1]{GW18}
from the context of single-colored operads 
to the context of $G$-operads
$\sOp^G_{\mathfrak C}$
with any fixed $G$-set of colors $\mathfrak{C}$.
In the direct sequel \cite{BP_ACOP},
we will then use the model structures in 
Theorem \ref{THMI} to obtain a Dwyer-Kan style model structure
\cite[Thm. \ref{AC-THMA}]{BP_ACOP}
on the larger category
$\sOp^G_\bullet = \left(\Op_\bullet(\sSet)\right)^G$
of $G$-objects on the category $\Op_\bullet(\sSet)$
of colored operads with varying sets of colors
(here, and throughout, we use $\bullet$ to indicate that colors are allowed to change;
note that in $\sOp^G_\bullet$ the color sets
thus have $G$-actions).
More precisely, one has inclusions
$\sOp^G_{\mathfrak C} \subset \sOp^G_\bullet$,
so that the model structures in Theorem \ref{THMI}
are almost (but not quite)
the restrictions of the model structure in 
\cite[Thm. \ref{AC-THMA}]{BP_ACOP},
cf. \cite[Prop. \ref{AC-FIBERGLMOD PROP}]{BP_ACOP},
extending the analogue stories for categories
$\mathsf{sCat}_\bullet$
\cite{Ber07b,BM13}
and operads $\mathsf{sOp}_\bullet$
\cite{Rob,CM13b,Cav}.

The motivation for 
\cite[Thm. \ref{AC-THMA}]{BP_ACOP},
and thus also for Theorem \ref{THMI},
is that it allows us to formulate 
a Quillen equivalence \cite{BP_TAS}
\begin{equation}\label{CMMAINTHM EQ}
\mathsf{dSet}^G \rightleftarrows \mathsf{sOp}_{\bullet}^G
\end{equation}
(where $\mathsf{dSet}^G$
is the category of equivariant dendroidal sets
of \cite{Per18}), thereby completing the 
generalization of the Cisinski-Moerdijk project of \cite{CM11,CM13a,CM13b} to the equivariant setting.
Further discussion on \eqref{CMMAINTHM EQ} can be found in the introduction to \cite{BP_TAS}.

We end this introduction by noting that, 
as in the single-colored case,
building the desired model structures on 
$\mathsf{sOp}^G_{\mathfrak C}$, $\mathsf{sOp}^G_{\bullet}$
is again non formal, 
and the colored case brings extra nuances.

Focusing first on
$\mathsf{sOp}^G_{\mathfrak C}$,
it is tempting to try to identify this as
either the category
$\mathsf{Op}_{\mathfrak C}(\mathsf{sSet}^G)$
of $\mathfrak{C}$-colored operads on $\mathsf{sSet}^G$,
or as the $G$-object category
$\left(\mathsf{Op}_{\mathfrak C}(\mathsf{sSet})\right)^G$.
However, both of these alternatives forget the 
$G$-action on $\mathfrak C$,
so that the identifications only hold if 
$G$-acts trivially on $\mathfrak{C}$.
As such, Theorem \ref{THMI}
can not be derived from neither 
\cite[Thm. 3.2]{BM03} nor \cite[Prop. 2.6]{Ste16}.

As for the larger category
$\mathsf{sOp}^G_{\bullet}$,
formal approaches again run into issues, albeit different ones.
First, since one only has an inclusion
$\mathsf{Op}_{\bullet}(\mathsf{sSet}^G)
\subsetneq \mathsf{sOp}_{\bullet}^G$,
rather than an equivalence,
the model structure on $\mathsf{sOp}_{\bullet}^G$
can not be built using \cite{Cav}.
Second, since, by definition,
$\mathsf{sOp}^G_{\bullet}$
is a category of $G$-objects
$\left(\mathsf{sOp}_{\bullet}\right)^G$,
one could try to build a model structure
on $\mathsf{sOp}^G_{\bullet}$
by applying \cite[Prop. 2.6]{Ste16}
to the model structure on $\mathsf{sOp}_\bullet$ from \cite{Cav}.
However, and as in the single-colored case,
this approach would not produce
the desired notion of equivalences suggested by graph subgroups.

\subsection{Main Results}

As noted in \eqref{GRAPHEQ EQ},
our preferred notion of equivalence of equivariant operads is determined by the graph subgroups.
However, throughout this paper we will find it technically 
no harder to work with a largely arbitrary collection of subgroups,
defined as follows.

\begin{definition}\label{FAM1ST DEF}
	A \emph{$(G,\Sigma)$-family} is a
	a collection
	$\mathcal{F} = \{\mathcal{F}_n\}_{n \leq 0}$,
	where each $\mathcal{F}_n$
	is a family of subgroups of $G \times \Sigma_n^{op}$.
\end{definition}

The use of $\Sigma_n^{op}$ rather than $\Sigma_n$
in Definition \ref{FAM1ST DEF} 
(and throughout the paper) 
is motivated by regarding $\Sigma$
as the category of corollas (trees with a single node; 
see \eqref{OPSSYMS EQ},\eqref{CSYM EQ1},\eqref{CSYM EQ2}),
and the fact that the dendroidal nerve \cite[\S 1]{MW07} of an operad is contravariant on the category of trees.


Recall that a colored operad $\O$
with color set $\mathfrak{C}$ has levels 
$
\O(\vect{C})=
\O(\mathfrak{c}_1,\cdots,\mathfrak{c}_n;\mathfrak{c}_0)$
indexed by tuples
$\vect{C} = (\mathfrak{c}_1,\cdots,\mathfrak{c}_n;\mathfrak{c}_0) = (\mathfrak c_i)_{0 \leq i \leq n}$
of elements in $\mathfrak{C}$, called \emph{$\mathfrak{C}$-profiles}.
Our operads will always be symmetric, i.e. equipped with associative and unital isomorphisms
$
\O(\mathfrak{c}_1,\cdots,\mathfrak{c}_n;\mathfrak{c}_0) \to 
\O(\mathfrak{c}_{\sigma(1)},\cdots,\mathfrak{c}_{\sigma(n)};\mathfrak{c}_0)
$
for each permutation $\sigma \in \Sigma_n$.
Moreover, if 
$\O \in \mathsf{sOp}^G_{\mathfrak{C}}$
is a $G$-equivariant operad, 
the color set $\mathfrak{C}$ is itself a $G$-set,
and one has additional associative and unital isomorphisms
$
\O(\mathfrak{c}_1,\cdots,\mathfrak{c}_n;\mathfrak{c}_0) \to 
\O(g\mathfrak{c}_{1},\cdots,g\mathfrak{c}_{n};g\mathfrak{c}_0)
$ for $g \in G$.
All together, one thus has isomorphisms
\begin{equation}\label{OPSSYMS EQ}
\O(\mathfrak{c}_1,\cdots,\mathfrak{c}_n;\mathfrak{c}_0)
\to 
\O(g \mathfrak{c}_{\sigma(1)},\cdots,g \mathfrak{c}_{\sigma(n)};g\mathfrak{c}_0)
\end{equation}
for $(g,\sigma) \in G \times \Sigma_n^{op}$.
Note that these isomorphisms 
are associated with an action of 
$G \times \Sigma_n^{op}$
on the set $\mathfrak{C}^{n+1}$ of $n$-ary profiles via
$(g,\sigma) (\mathfrak{c}_i)_{0\leq i \leq n}
= (g \mathfrak{c}_{\sigma(i)})_{0\leq i \leq n}$,
where we implicitly write $\sigma(0)=0$.
As such, we say that a subgroup 
$\Lambda \leq G \times \Sigma_n^{op}$
\emph{stabilizes} a profile $\vect{C}=(\mathfrak{c}_i)_{0 \leq i \leq n}$ if,
for any $(g,\sigma) \in \Lambda$,
it is 
$\mathfrak{c}_i = g \mathfrak{c}_{\sigma(i)}$ for all $0 \leq i \leq n$.
Note that,
for $\O \in \mathsf{sOp}^G_{\mathfrak{C}}$,
the level $\O(\vect{C})$ has a $\Lambda$-action.

\begin{customthm}{I}\label{THMI}
Let $G$ be a finite group. Fix a $G$-set of colors $\mathfrak{C}$
and a $(G,\Sigma)$-family $\F = \set{\F_n}_{n \geq 0}$.

Then there exists a model structure on
$\mathsf{sOp}^G_{\mathfrak{C}} =
\mathsf{Op}^G_{\mathfrak{C}}(\mathsf{sSet})$,
which we call the \emph{$\mathcal{F}$-model structure},
such that a map
$\mathcal{O} \to \mathcal{P}$
is a weak equivalence (resp. fibration) if the maps
\begin{equation}\label{THMI_EQ}
	\O(\vect{C})^{\Lambda} \to \mathcal{P}(\vect{C})^{\Lambda}
\end{equation}
are Kan equivalences (Kan fibrations)
in $\mathsf{sSet}$
for all $\mathfrak{C}$-profiles $\vect{C}$
and $\Lambda \in \F$ which stabilize $\vect{C}$.

More generally, a model structure on 
$\Op^G_{\mathfrak C}(\V)$
with weak equivalences and fibrations determined as in 
\eqref{THMI_EQ}
exists provided that:
\begin{enumerate}[label = (\roman*)]
	\item $\V$ is a cofibrantly generated model category
	such that the domains of the generating (trivial) cofibrations are small;
	\item for any finite group $G$, the $G$-object category $\V^G$ admits the genuine model structure (Definition \ref{GENMOD DEF});
	\item $(\V, \otimes)$ is a closed symmetric monoidal model category with cofibrant unit;
	\item $(\V, \otimes)$ satisfies the global monoid axiom (Definition \ref{GLOBMONAX_DEF});
	\item $(\V, \otimes)$ has cofibrant symmetric pushout powers (Definition \ref{CSPP_DEF}).
\end{enumerate}
\end{customthm}

The proof of Theorem \ref{THMI} is given in \S \ref{OPC_MS_SEC}.
As usual, the $\F$-model structure on 
$\Op^G_{\mathfrak C}(\V)$ is lifted from 
a similar $\F$-model structure on the simpler category
$\mathsf{Sym}^G_{\mathfrak C}(\V)$
of $\mathfrak{C}$-colored symmetric sequences
(Definition \ref{CSSYM DEF}),
which are objects with the isomorphism data as in 
\eqref{OPSSYMS EQ}, 
but lacking the operadic composition maps.

For certain special $(G,\Sigma)$-families
(motivated by the indexing systems of Blumberg-Hill \cite{BH15})
we will also show that cofibrant objects in 
$\Op^G_{\mathfrak C}(\V)$
forget to cofibrant objects in 
$\mathsf{Sym}^G_{\mathfrak C}(\V)$.

\begin{customthm}{II}\label{THMII}
	Suppose $\V$ satisfies the hypotheses in Theorem \ref{THMI},
	and let $\F$ be a pseudo indexing system (Definition \ref{PIS_DEF}).
	Then, if $\O \to \mathcal{P}$ in $\Op^G_{\mathfrak{C}}(\V)$ is a cofibration between cofibrant objects for the $\F$-model structure,
	so is the underlying
	map of symmetric sequences in $\Sym^G_{\mathfrak{C}}(\V)$.
\end{customthm}

Theorem \ref{THMII} is proven in \S \ref{INDSYS_SEC}
as a particular case of Proposition \ref{SIGMAG_COF PROP}.
The notion of pseudo indexing systems
extends that of weak indexing systems \cite[Def. 4.58]{BP21}
(or, equivalently, realizable sequences \cite[Def. 4.6]{GW18}),
which themselves extend the original 
indexing systems of Blumberg-Hill \cite{BH15}.
Notably, the graph subgroups in \eqref{GRAPHEQ EQ} form a (pseudo) indexing system.

\begin{remark}\label{RESTTOCATS REM}
	There are identifications
	$\Cat_\bullet^G(\V) \simeq \Op_\bullet^G(\V) \downarrow \**$ and
	$\Cat_{\mathfrak C}^G(\V) \simeq \Op_{\mathfrak C}^G(\V) \downarrow \**_{\mathfrak C}$,
	where $\**$ (resp. $\**_{\mathfrak C}$) denotes the terminal $\V$-category (with color set $\mathfrak C$),
	so the $\F$-model structures on $\Op_{\mathfrak C}^G(\V)$, $\Op_\bullet^G(\V)$
	induce model structures on $\Cat_{\mathfrak C}^G(\V)$, $\Cat_\bullet^G(\V)$.
	Since categories contain only unary operations,
	these latter model structures depend only on $\F_1$,
	which is identified with a family of subgroups of $G$ itself.

	Moreover, we note that the analogues 
	for $\Cat_{\mathfrak C}^G(\V)$
	of both Theorems \ref{THMI} and \ref{THMII}
	follow from our proofs without using
	the cofibrant pushout power condition (v)
	in Theorem \ref{THMI},
	and without additional restrictions on $\F_1$
	(i.e. no analogue of the pseudo indexing system condition  is needed).
	For details, see Remark \ref{CSPNTHI REM}.
	For a similar discussion concerning
	$\Cat_\bullet^G(\V)$ and
	\cite[Thm. \ref{AC-THMA}]{BP_ACOP},
	see \cite[Rem. \ref{AC-RESTTOCATS REM}]{BP_ACOP}.
\end{remark}

\begin{remark}\label{SEMI_REM}
	When working with operads, some authors (e.g. \cite{Spi,Whi17,WY18})
	discuss \emph{semi-model structures}.
	Briefly, these are a weakening of Quillen's original definition,
	where those factorization and lifting axioms
	that involve trivial cofibrations
	are only required to hold if the trivial cofibration 
	has cofibrant source \cite[\S 2.2]{WY18}.
	We note that, in particular, semi-model structures suffice for 
	performing 
	bifibrant replacements.
        
	The semi-model structure analogues of 
	Theorems \ref{THMI} and \ref{THMII}
	can be obtained by slight variants of our proofs
	without using the global monoid axiom (iv).
	For details, see Remark \ref{THMISM REM}.
\end{remark}

\begin{remark}\label{GTRIV REM}
	It may be tempting to think that if the group $G=\**$ is trivial
	one can omit the existence of genuine model structures assumption in (ii) of Theorem \ref{THMI}.
	However, that is not the case since, 
	even when $\F$ is the $(G,\Sigma)$-family of trivial subgroups
	(i.e. the projective case usually discussed in the literature),
	our arguments in the proofs of Theorems \ref{THMI},\ref{THMII}
	are still using the rather strong
	cofibrant pushout powers assumption (v).
	However, in this specific case 
	there are less restrictive sufficient conditions available in the literature, such as those in \cite[Thm. 1.1]{PS18}.
%
 \end{remark}

\subsection{Examples}\label{EXAMPLES SEC}

The examples of 
model categories satisfying 
all of conditions (i) through (v)
in Theorem \ref{THMI}
are fairly limited, 
mostly due to 
the cofibrant pushout powers axiom (v),
which is rather restrictive.
We further discuss the role of this condition 
in Remarks 
\ref{CPPWHY REM} and
\ref{SPNONEX REM} below.

Below we list all known examples of categories satisfying all the above conditions.
\begin{enumerate}[label = (\alph*)]
\item $(\mathsf{sSet},\times)$ or $(\mathsf{sSet}_{\**},\wedge)$
with the Kan model structure.
\item $(\mathsf{Top},\times)$ or $(\mathsf{Top}_{\**},\wedge)$
with the usual Serre model structure.
\item $(\mathsf{Set},\times)$ the category of sets with its canonical model structure,
where weak equivalences are the bijections and all maps are both cofibrations and fibrations.
\item $(\Cat,\times)$ the category of usual categories
with the ``folk'' or canonical model structure (e.g. \cite{Rez})
where weak equivalences are the equivalences of categories,
cofibrations are the functors which are injective on objects,
and fibrations are the isofibrations.
\end{enumerate}
In all these cases, conditions (i) and (iii) 
are well known.
Moreover,
the existence of the genuine model structures on $\V^G$ for condition (ii)
and the global monoid axiom for condition (iv)
all follow from Proposition \ref{WEAKCELL PROP},
since these $(\V,\otimes)$ satisfy the usual monoid axiom and
are readily seen to satisfy the conditions in Definition \ref{WEAKCELL DEF}.

Lastly, we discuss the cofibrant symmetric pushout powers condition in (v).
In case (a) this was shown in 
\cite[Ex. 6.19]{BP21}.
Case (b) is a consequence of case (a), 
since the generating (trivial) cofibrations of (b) are geometric realizations of those in (a),
and by \cite[Rem. 6.17]{BP21}
it suffices to verify the cofibrant pushout power condition
on generating sets.
Case (c) is straightforward and left as an exercise 
(the ``hardest'' step is the observation that fixed points of isomorphisms are isomorphisms).
For case (d), the only non obvious claim is that
if $u$ is a trivial cofibration in $\mathsf{Cat}$
then the pushout product map $u^{\square n}$
is a $\Sigma_n$-genuine trivial cofibration in $\mathsf{Cat}^{\Sigma_n}$.
By \cite[Rem. 6.17]{BP21}
we can assume that $u = \left(\{0\} \to (0\rightleftarrows 1)\right)$,
i.e. that u is the map from a singleton to the walking isomorphism category,
as this is the only generating trivial cofibration of 
$\mathsf{Cat}$.
One can then either repeat the argument in \cite[Ex. 6.19]{BP21}
or compute $u^{\square n}$ directly.
The target of $u^{\square n}$
is the contractible groupoid on the set $\{0,1\}^{\times n}$,
so that $u^{\square n}$
is the inclusion of the full subcategory of 
$\{0,1\}^{\times n}$
with $(1,1,\cdots,1)$ removed.
It then follows that $u^{\square n}$
is a pushout of the map
$\Sigma_n/\Sigma_n \cdot (0 \rightleftarrows 1)$.

\begin{remark}
        These examples also satisfy the axioms necessary for \cite[Thm. \ref{AC-THMA}]{BP_ACOP},
        yielding a Dwyer-Kan model structure on $\Op_\bullet^G(\V)$, cf \cite[\S \ref{AC-EXAMPLES SEC}]{BP_ACOP}.
\end{remark}

\begin{remark}\label{CPPWHY REM}
	As noted above, the cofibrant pushout powers
	condition (v) is the most restrictive 
	out of all the conditions in 
	Theorem \ref{THMI}. 
	Nonetheless, we chose to use this property 
	since it has two very convenient bootstrapping properties.
	Firstly, as noted in the previous discussion, \cite[Rem. 6.17]{BP21}
	says that the cofibrant pushout powers condition needs only be checked on generating (trivial cofibrations).
	Secondly, the cofibrant pushout powers condition 
	can be used to deduce more complex
	``$\G$-cofibrant pushout powers'' analogue conditions for any groupoid $\G$,
	as formulated in Proposition \ref{SIGMAWRGF PROP}.
	
	In our view, the main obstacle to generalizing our results is the identification of 
	a substitute for the cofibrant pushout powers condition (v)
	which has similar bootstrapping properties.
\end{remark}

We end this section by discussing a noteworthy 
example for which our results do not apply.

\begin{remark}\label{SPNONEX REM}
	The category $(\mathsf{Sp}^{\Sigma}(\mathsf{sSet}),\wedge)$
	of symmetric spectra (on simplicial sets),
	with the positive $S$ model structure,
	satisfies most of the axioms in Theorem
	\ref{THMI} (and \cite[Thm. \ref{AC-THMA}]{BP_ACOP}), 
	with the exceptions being 
	the cofibrant unit requirement in (iii)
	and the cofibrant pushout powers axiom in (v).
	However, we believe that with some care 
	both of these problems could be sidestepped
	if necessary. 
	On the one hand, \cite{GV12}
	showed that the non-cofibrancy of the unit 
	does not prevent the existence of fixed color model structures
	and, on the other hand, 
	the second author showed in \cite{Pe16}
	that pushout powers $u^{\square n}$ of positive $S$ cofibrations $u$
	in $(\mathsf{Sp}^{\Sigma}(\mathsf{sSet}),\wedge)$
	satisfy a ``lax-$\Sigma_n$-cofibrancy'' condition,
	which enjoys analogues of the bootstrapping conditions
	in Remark \ref{CPPWHY REM}
	(in fact, our key result concerning cofibrant pushout powers,
	Proposition \ref{SIGMAWRGF PROP}, and its direct precursor
	\cite[Prop. 6.25]{BP21}
	were originally inspired by \cite[Thm. 1.2]{Pe16}).

	As such, we believe that symmetric spectra
	$\mathsf{Sp}^{\Sigma}(\mathsf{sSet})$
	likely satisfy a close analogue of Theorems 
	\ref{THMI}, \ref{THMII} (and \cite[Thm. \ref{AC-THMA}]{BP_ACOP}).
	However, such results 
	would be fundamentally unsatisfying, 
	since the resulting notion of weak equivalence on
	$G$-symmetric spectra
	$\left(\mathsf{Sp}^{\Sigma}(\mathsf{sSet})\right)^G$
	does not match the correct notion of 
	\emph{genuine equivalences of $G$-spectra}.
	More precisely, the latter equivalences are a localization of the former, so that the direct analogues of 
	Theorems \ref{THMI}, \ref{THMII} (and \cite[Thm. \ref{AC-THMA}]{BP_ACOP})
	would at best represent only an intermediate step
	towards the ``genuine'' results
	for equivariant colored spectral operads.

	More generally,
	for an arbitrary model category $\V$,
	the initial choice of the $(G,\Sigma)$-family $\F$ in  
	Definition \ref{FAM1ST DEF}
	should be replaced with a choice of model structures on
	$\V^{G \times \Sigma_n^{op}}$
	for each $n \geq 0$.	 
\end{remark}

\subsection{Outline}

We start in \S \ref{PRE SEC} by discussing some preliminary notions that will be needed throughout.
As noted in the introduction to 
\S \ref{PRE SEC}, our treatment will be simplified by 
using an ``all colors'' approach,
working with the category $\Op_\bullet(\V)$ of all colored operads,
regarded as suitably ``fibered'' over the category $\mathsf{Set}$
of sets.
To that end, 
\S \ref{GROTFIB SEC} recalls the necessary
notion of Grothendieck fibration, 
while \S \ref{FIBCAT_SEC}
discusses how the notions of adjunction and monad 
interact with such fibrations.

\S \ref{ECO_SEC} then applies the abstract setup in \S \ref{PRE SEC}
to discuss equivariant colored symmetric sequences and operads. 
\S \ref{EQCOSYMSEQ SEC}
explores the category
$\mathsf{Sym}_{\bullet}(\V)$
of all colored symmetric sequences,
with a highlight being
Proposition \ref{EQUIVFNCON PROP},
which shows that the category
$\mathsf{Sym}^G_{\mathfrak{C}}(\V)$
of $G$-symmetric sequences with a fixed $G$-set of objects
$\mathfrak{C}$ can be described as a presheaf category.
In \S \ref{REPFUN_SEC} we prove
Proposition \ref{REPALTDESC PROP},
which provides a convenient description of the representable functors in 
$\mathsf{Sym}^G_{\mathfrak{C}} = \mathsf{Sym}^G_{\mathfrak{C}}(\mathsf{Set})$.
Lastly, \S \ref{EQCOSYMOP SEC}
briefly describes the category 
$\mathsf{Op}_{\bullet}(\V)$ 
of all colored operads as the algebras over a ``fibered monad''
on  
$\mathsf{Sym}_{\bullet}(\V)$,
and unpacks the abstract discussion in \S \ref{FIBCAT_SEC}
so as to likewise describe the category 
$\mathsf{Op}_{\bullet}^G(\V)$ 
of all equivariant colored operads
as algebras on 
$\mathsf{Sym}^G_{\bullet}(\V)$.

\S \ref{EHT_SEC} develops the equivariant homotopy theory
needed to prove of Theorems \ref{THMI} and \ref{THMII}.
First, \S \ref{GMA_SEC} introduces the global monoid axiom featured in Theorem \ref{THMI}.
Then, in \S \ref{FGPP_SEC}, 
we extend 
the work 
about equivariant model structures determined by families of subgroups in \cite[\S 6]{BP21}
from the context of groups to that of groupoids.
This culminates in 
Proposition \ref{SIGMAWRGF PROP},
which concerns the properties of pushout powers $f^{\square n}$,
and is one of the key technical results in the paper. 

Lastly, \S \ref{FIXCOL SEC}
is dedicated to proving our two main results, 
Theorems \ref{THMI} and \ref{THMII}.
\S \ref{SYMC_MS_SEC} first specifies the theory in
\S \ref{FGPP_SEC}
to obtain model structures on the categories
$\mathsf{Sym}^G_{\mathfrak{C}}(\V)$ of symmetric sequences with fixed colors.
These are then used in \S \ref{OPC_MS_SEC}
to obtain transferred model structures on the 
categories
$\mathsf{Op}^G_{\mathfrak{C}}(\V)$
of operads with fixed colors, establishing Theorem \ref{THMI}.
To finish, \S \ref{INDSYS_SEC} proves Theorem \ref{THMII}
via a more careful analysis of the argument in the proof of 
Theorem \ref{THMI}.


In Appendix \ref{MONAD_APDX},
we fill in some technical work that was postponed in \S \ref{ECO_SEC} and \S \ref{FIXCOL SEC},
namely the full description of the ``free operad monad''
from \S \ref{EQCOSYMSEQ SEC},
and the proof of Lemma \ref{OURE LEM},
which provides the key filtrations used in the proofs of
Theorems \ref{THMI} and \ref{THMII}.

\section{Preliminaries}\label{PRE SEC}


Much as in the non-equivariant case,
the first step towards the
construction of a model structure on the category $\Op_\bullet^G(\V)$ of operads on all sets of colors
is to build model structures on each 
fixed color subcategory $\Op_{\mathfrak C}^G(\V)$.
However, the equivariant setting presents some technical challenges that will 
require us to somewhat repackage the non-equivariant narrative of
\cite{CM13b},\cite{Cav}.

To see why, recall that \cite{CM13b},\cite{Cav}
follow a ``work color by color and then assemble'' strategy.
More precisely, first the model structures on fixed color operads 
$\mathsf{Op}_{\mathfrak{C}}(\V)$
are built by identifying these as algebras over a monad
$\mathbb{F}_{\mathfrak{C}}$
on the category $\mathsf{Sym}_{\mathfrak{C}}(\V)$
of fixed color symmetric sequences.
As such, maps of operads that do not fix colors
only appear afterward when assembling the model structure on the full category $\mathsf{Op}_{\bullet}(\V)$.

However, when working equivariantly, 
while the maps in the fixed $G$-set of objects categories 
$\mathsf{Op}^G_{\mathfrak{C}}(\V)$ do fix colors, 
the $G$-action on \emph{objects}
$\O \in \mathsf{Op}^G_{\mathfrak{C}}(\V)$
involves maps of operads $\O \xrightarrow{g} \O$
that need not fix colors (unless $\mathfrak{C}$ is a trivial $G$-set).
In other words, when discussing equivariant operads,
even the fixed color categories 
$\mathsf{Op}^G_{\mathfrak{C}}(\V)$
require color change data.
Due to this issue, 
our approach will be that the transition from the non-equivariant to the equivariant case is easier to describe 
if we regard the ``all colors'' framework as the primary framework,
and then restrict to color fixed operads only when needed.

More explicitly, the basis of our approach is to combine the 
fixed color symmetric sequence categories
$\mathsf{Sym}_{\mathfrak{C}}(\V)$ for all color sets 
$\mathfrak{C} \in \mathsf{Set}$
(and change of color data between them)
into a single category $\mathsf{Sym}_{\bullet}(\V)$
(Definition \ref{CSSYM DEF}).
There is then a Grothendieck fibration
$\mathsf{Sym}_{\bullet}(\V) \to \mathsf{Set}$
which records the underlying set of colors,
and whose fiber over $\mathfrak{C} \in \mathsf{Set}$ is 
$\mathsf{Sym}_{\mathfrak{C}}(\V)$.
Similarly, the monads $\mathbb{F}_{\mathfrak{C}}$
on $\mathsf{Sym}_{\mathfrak{C}}(\mathcal{V})$
(and change of color data between them) 
assemble into a single monad $\mathbb{F}$
on $\mathsf{Sym}_{\bullet}(\V)$
(Definition \ref{FREEOP DEF})
which is suitably compatible with 
the Grothendieck fibration 
$\mathsf{Sym}_{\bullet}(\V) \to \mathsf{Set}$,
and, by considering (a suitable subcategory of) algebras over $\mathbb{F}$,
one obtains the category $\mathsf{Op}_{\bullet}(\V)$
of colored operads for all colors.
These fit together as on the left below.
\begin{equation}\label{FIBADJMON EQ}
\begin{tikzcd}[row sep = small, column sep = tiny]
\Sym_\bullet(\V) \arrow[rr, "\mathbb F"] \arrow[dr]
&&
\Op_\bullet(\V) \arrow[dl]
&&& 
\Sym^G_\bullet(\V) \arrow[rr, "\mathbb F^G"] \arrow[dr]
&&
\Op^G_\bullet(\V) \arrow[dl]
\\
&
\mathsf{Set}
&
&&& 
&
\mathsf{Set}^G
\end{tikzcd}
\end{equation}
Within this ``all colors'' framework, passing to the equivariant case is simply a matter of applying $G$-objects throughout as on the right above, 
so that one has a Grothendieck fibration
$\mathsf{Sym}_{\bullet}^G \to \mathsf{Set}^G$
with a compatible monad $\mathbb{F}^G$
from which one obtains the category 
$\mathsf{Op}_{\bullet}(\V)$
of $G$-equivariant colored operads for all colors.

The plan for this preliminary section is then as follows.
\S \ref{GROTFIB SEC} recalls the notion of Grothendieck fibration and
introduces some related constructions that are used throughout.
\S \ref{FIBCAT_SEC} discusses how the notions of adjunction and monad interact with Grothendieck fibrations,
which will allow us in \S \ref{ECO_SEC}
to regard the monads
(and associated adjunctions)
in \eqref{FIBADJMON EQ}
as being suitably fibered over $\mathsf{Set}$ and $\mathsf{Set}^G$.

\begin{remark}
        As an aside,
        we note that our discussion of fibered category theory will
        also streamline our work in the sequel \cite{BP_TAS}.
        Therein, we will need to establish a Quillen equivalence
        $\mathsf{PreOp}^G \rightleftarrows \mathsf{sOp}^G$
        between the category of (simplicial) preoperads and
        the category of simplicial operads.
        In that work, the claim that the adjunction is Quillen 
        will be greatly simplified by noting that in both categories 
        the generating (trivial) cofibrations
        can be described using a ``fibered simplicial cotensoring''.
\end{remark}

\subsection{Grothendieck fibrations}\label{GROTFIB SEC}

Recall that a functor 
$\pi \colon \mathcal{C} \to \mathcal{B}$
is called a \emph{Grothendieck fibration} if,
for all arrows
$\varphi \colon b' \to b$ in $\mathcal{B}$
and $c \in \mathcal{C}$ such that $\pi(c) = b$,
there exists a \emph{cartesian arrow}
$\varphi^{\**}c \to c$
lifting $\varphi$,
meaning that for any choice of solid arrows
\begin{equation}\label{CARTARDEF EQ}
\begin{tikzcd}
c'' \ar{rr} \ar[dashed]{rd}[swap]{\exists!} 
&&
c
&
b'' \ar{rr} \ar{rd} 
&&
b
\\
& \varphi^{\**} c \ar{ru}
&
&
& b'\ar{ru}[swap]{\varphi}
&
\end{tikzcd}
\end{equation}
such that the right diagram commutes and 
$c'' \to c$ lifts $b'' \to b$,
there exists a unique dashed arrow
$c'' \to \varphi^{\**} c$ lifting $b'' \to b'$
and making the left diagram commute.

A \emph{cleavage} of $\pi$ is a fixed choice of cartesian arrows
$\varphi^{\**} c \to c$
for each $\varphi \colon b' \to b$ and $c$ with $\pi(c)=b$.
Note that, writing $\mathcal{C}_b$ for the fiber over $b \in \mathcal{B}$, a cleavage determines functors
$\varphi^{\**} \colon \mathcal{C}_b \to \mathcal{C}_{b'}$
for each $\varphi \colon b' \to b$.

Dually, if $\pi^{op} \colon \mathcal{C}^{op} \to \mathcal{B}^{op}$
is a Grothendieck fibration,
we say that $\pi$ is a \emph{Grothendieck opfibration}.
More explicitly, this means that, for any arrow
$\varphi \colon b \to b'$ in $\mathcal{B}$
and $c \in \mathcal{C}$ such that $\pi(c) = b$,
there exists a \emph{cocartesian arrow}
$c \to \varphi_! c$ lifting $\varphi$
and satisfying the dual of the universal property in 
\eqref{CARTARDEF EQ}.
A cleavage of an opfibration is similarly defined as a 
choice of cocartesian arrows $c \to \varphi_!c$.

\begin{notation}\label{MAPSDEC NOT}
	Given any functor of categories
	$\pi \colon \mathcal{C} \to \mathcal{B}$,
	one has a natural decomposition of mapping sets
	\begin{equation}
	\mathcal{C}(c',c) = 
	\coprod_{f \in \mathcal{B}(\pi(c'),\pi(c))}
	\mathcal{C}_{\varphi}\left(c',c \right)
	\end{equation}
	where $\mathcal{C}_{\varphi}\left(c',c \right)$ consists of the arrows projecting to $\varphi$.
\end{notation}

\begin{remark}\label{CARTCHAR REM}
	Specifying to the case $b'' = b'$ in 
	\eqref{CARTARDEF EQ},
	one has that,
	when $\pi$ is a Grothendieck fibration, the contravariant functors
	\begin{equation}\label{FIXPIREP EQ}
	\mathcal{C}_{\varphi}(-,c)
	\colon
	\mathcal{C}_{b'} 
	\to
	\mathsf{Set}
	\end{equation}
	are represented by (some choice of) $\varphi^{\**}c$.
	Moreover, note that under the representing isomorphism
	$\mathcal{C}_{b'}(\varphi^{\**}c,\varphi^{\**}c)
	\simeq \mathcal{C}_{\varphi}(\varphi^{\**}c,c)$
	the identity
	$\varphi^{\**}c \xrightarrow{=} \varphi^{\**}c$
	yields the canonical map
	$\varphi^{\**}c \to c$ over $\varphi$.
\end{remark}

\begin{remark}\label{COMPSEMI REM}
	The condition in \eqref{CARTARDEF EQ}
	is stronger than the representability of 
	\eqref{FIXPIREP EQ}.
	More precisely, let us say an arrow 
	$\varphi^{\**}c \to c$
	is weakly cartesian if it represents \eqref{FIXPIREP EQ},
	and that  
	$\pi \colon \mathcal{C} \to \mathcal{B}$
	is a weak Grothendieck fibration if it admits all weakly cartesian arrows.
	Then $\pi$ is in fact a Grothendieck fibration,
	i.e. the weakly cartesian arrows 
	$\varphi^{\**}c \to c$
	satisfy the stronger condition in \eqref{CARTARDEF EQ},
	iff the composites of weakly cartesian arrows are again weakly cartesian, 
	i.e. if 
	for any composable arrows
	$b'' \xrightarrow{\psi} b' \xrightarrow{\varphi} b$
	and $c$ with $c \in \mathcal{C}_b$ it is
	$\psi^{\**} \varphi^{\**} c \simeq 
	\left( \varphi \psi \right)^{\**} c$,
	where the isomorphism is in $\mathcal{C}_b$.
\end{remark}

\begin{remark}\label{ALSOOPADJ REM}
	As previously noted, a cleavage of a Grothendieck fibration $\pi \colon \mathcal{C} \to \mathcal{B}$ determines functors 
	$\varphi^{\**} \colon \mathcal{C}_b \to \mathcal{C}_{b'}$
	for each arrow $\varphi \colon b' \to b$.
	In addition, 
	the claim that $\pi$ is also an opfibration is then equivalent to the existence of left adjoints
	\[
	\varphi_! \colon
	\mathcal{C}_{b'}
	\rightleftarrows
	\mathcal{C}_{b}
	\colon \varphi^{\**}
	\]
	for all arrows $\varphi \colon b' \to b$ in $\mathcal{B}$.
	Note that the required condition that the functors
	$\varphi_! \psi_!$ and
	$(\varphi \psi)_!$
	are naturally isomorphic (cf. Remark \ref{COMPSEMI REM})
	is automatic, 
	as these are left adjoints to 
	$\psi^{\**} \varphi^{\**} \simeq 
	\left( \varphi \psi \right)^{\**}$.
\end{remark}

\begin{remark}\label{FUNISGROTH REM}
	Let $\pi \colon \mathcal{C} \to \mathcal{B}$
	be a Grothendieck fibration and
	$I$ a small category.
	Writing $\mathcal{C}^{I}$, $\mathcal{B}^{I}$ for the categories of functors 
	$I \to \mathcal{C}$, $I \to \mathcal{B}$,
	the functor
	$\pi^I \colon \mathcal{C}^I \to \mathcal{B}^I$
	is again a Grothendieck fibration, with cartesian arrows in $\mathcal{C}^I$ the natural transformations built of cartesian arrows in $\mathcal{C}$.
\end{remark}

\begin{notation}\label{GROTHCONS NOT}
	Given a functor $\mathcal{B} \to \mathsf{Cat}$
	let us write
	$\mathcal{C}_b \in \mathsf{Cat}$
	for the image of $b \in \mathcal{B}$ and 
	$\varphi_! \colon \mathcal{C}_{b} \to \mathcal{C}_{b'}$
	for the functor induced by
	$\varphi \colon b \to b'$.
	We then write 
	$\mathcal{B} \ltimes \mathcal{C}_{\bullet}$
	for the associated (covariant) Grothendieck construction.
	
	More explicitly, $\mathcal{B} \ltimes \mathcal{C}_{\bullet}$ is
	the category whose objects are pairs 
	$(b,c)$ with $b \in \mathcal{B}$ and $c \in \mathcal{C}_b$,
	and with an arrow
	$(b,c) \to (b',c')$
	given by an arrow 
	$\varphi \colon b \to b'$ in $\mathcal{B}$
	together with an arrow
	$f \colon \varphi_! c \to c'$ in $\mathcal{C}_{b'}$.
	Note that the composite of 
	$f \colon \varphi_! c \to c'$ and
	$f' \colon \psi_! c' \to c''$ is given by
	\[
	\psi_! \varphi_! c \xrightarrow{\psi_!f}
	\psi_! c' \xrightarrow{f'}
	c''.
	\]
	Lastly, note that the natural projection
	$\pi \colon \mathcal{B} \ltimes \mathcal{C}_{\bullet} \to \mathcal{B}$
	is naturally a Grothendieck opfibration.
\end{notation}

\begin{remark}\label{GROTHUNPCK REM}
	It can be helpful to simplify notation 
	and write elements $(b,c)$ of $\mathcal{B} \ltimes \mathcal{C}_{\bullet}$ simply as $c$.
	Under this convention, we depict arrows in 
	$\mathcal{B} \ltimes \mathcal{C}_{\bullet}$
	as composites $c \rightsquigarrow \varphi_! c \xrightarrow{f} c'$
	where 
	$c \rightsquigarrow \varphi_! c$
	denotes the cocartesian arrow from $c$ to $\varphi_! c$
	and $f$ is the fiber arrow.
	Composites of either two cocartesian or two fiber arrows
	work as obvious:
	$c \rightsquigarrow
	\varphi_! c 
	\rightsquigarrow
	\psi_! \varphi_! c$
	equals 
	$c \rightsquigarrow
	(\psi \varphi_!) c$
	while
	$c \xrightarrow{f} c' \xrightarrow{f'} c''$
	equals
	$c \xrightarrow{f'f} c''$.
	The only non-obvious composites are then those of the form
	$c \xrightarrow{f} c' \rightsquigarrow \varphi_! c'$,
	which are determined by the commutativity of the square
	\begin{equation}\label{GROTHUNPCK EQ}
	\begin{tikzcd}
	c \ar{d}[swap]{f} 
	\arrow[rightsquigarrow]{r}
	&
	\varphi_! c \arrow{d}{\varphi_! f}
	\\
	c' \arrow[rightsquigarrow]{r} &
	\varphi_! c'
	\end{tikzcd}
	\end{equation}
\end{remark}

\begin{remark}\label{SPLITOPFIB REM}
	If $\pi \colon \mathcal{C} \to \mathcal{B}$
	is an opfibration then, by (the dual of) Remark \ref{COMPSEMI REM},
	for any cleavage one must have associativity isomorphisms
	$\varphi_! \psi_! \simeq \left(\varphi \psi\right)_!$,
	but these need not be equalities.
	Should a cleavage be strictly associative, i.e. 
	$\varphi_! \psi_! = \left(\varphi \psi\right)_!$
	for all composable $\varphi,\psi$
	and unital, i.e.
	$(id_b)_! = id_{\mathcal{C}_b}$,
	then the Grothendieck opfibration is called \emph{split}.
	One can show that an opfibration is split iff it is isomorphic to a Grothendieck construction in the sense of Notation \ref{GROTHCONS NOT}.
\end{remark}

The following are the two main instances of 
Notation \ref{GROTHCONS NOT} we will make use of.

\begin{example}\label{GLTIMES EQ}
	When $G$ is a group regarded as a category with a single object,
	a functor
	$G \to \mathsf{Cat}$
	consists of a single category $\C$ with a $G$-action.
	In this case, $G \ltimes \mathcal{C}$
	can be thought of as the category obtained from $\mathcal{C}$
	by formally adding ``action arrows''
	$c \xrightarrow{g} gc$ for each $g\in G,c\in \mathcal{C}$.
	
	More explicitly, an arrow from $c$ to $c'$
	in $G \ltimes \mathcal{C}$
	is uniquely described as an arrow
	$\varphi \colon gc \to c'$ in $\mathcal{C}$ for some $g \in G$,
	with the composite of 
	$\varphi \colon gc \to c'$
	and
	$\bar{\varphi} \colon \bar{g}c' \to c''$
	given by
	the composite
	$ \bar{g}g c \xrightarrow{\bar{g} \varphi} \bar{g}c' \xrightarrow{\bar{\varphi}} c''$ in $\C$.
\end{example}

\begin{remark}\label{INVLTIMES REM}
	Any group $G$ admits an inversion isomorphism
	$G \simeq G^{op}$ given by $g \mapsto g^{-1}$.
	If $\mathcal{C}$ is a $G$-category, 
	so is $\mathcal{C}^{op}$,
	and the inversion isomorphism lifts
	to an isomorphism
	$G \ltimes \mathcal{C}^{op} \simeq \left(G \ltimes \mathcal{C}\right)^{op}$ which is the identity on objects.
	Indeed, an arrow from $c$ to $\bar{c}$
	in $G \ltimes \mathcal{C}^{op}$
	is an arrow $f \colon \bar{c} \to gc$ in $\mathcal{C}$
	while an arrow in $\left(G \ltimes \mathcal{C}\right)^{op}$
	is an arrow
	$f' \colon g \bar{c} \to c$ in $\mathcal{C}$.
	The isomorphism then identifies 
	the arrow of $G \ltimes \mathcal{C}^{op}$
	represented by
	$f \colon \bar{c} \to gc$ 
	with the arrow of $\left(G \ltimes \mathcal{C}\right)^{op}$
	represented by
	$g^{-1}f \colon g^{-1}\bar{c} \to c$.
\end{remark}

\begin{notation}\label{SIGWR NOT}
	Let $\Sigma = \coprod_{n \geq 0} \Sigma_n$
	be the \emph{symmetric category},
	whose objects are the non-negative integers $n\geq 0$,
	and whose arrows are all automorphisms,
	with $n$ having automorphism group $\Sigma_n$.
	
	For any category $\mathcal{C}$ there is then a functor
	\begin{equation}
	\Sigma^{op} \longto \Cat,
	\qquad
	n \mapsto \mathcal C^{\times n}
	\end{equation}
	and we will abbreviate
	$\Sigma \wr \C = 
	\left(\Sigma^{op} \ltimes \mathcal C^{\times (-)}\right)^{op}$.
	Unpacking notation, 
	the elements of $\Sigma \wr \mathcal{C}$
	are tuples
	$(c_i)_{1 \leq i \leq n}$
	of elements $c_i \in \mathcal{C}$ for some $n \geq 0$,
	and a map of tuples $(c_i) \to (d_i)$, 
	necessarily of the same size $n$,
	consists of a permutation 
	$\sigma \in \Sigma_n$
	together with maps
	$c_i \to d_{\sigma(i)}$
	for $1 \leq i \leq n$.
\end{notation}

\begin{remark}\label{FWR REM}
	Writing $\mathsf{F}$ for the skeleton of 
	the category of finite sets consisting of the sets 
	$\underline{n} = \{1,\cdots,n\}$ for $n \geq 0$,
	we can regard
	$\Sigma \subset \mathsf{F}$
	as the maximal subgroupoid.
	
	The tuple category $\Sigma \wr \mathcal{C}$
	in Notation \ref{SIGWR NOT}
	is then a subcategory of an analogous category
	$\mathsf{F} \wr \mathcal{C}$.
\end{remark}

\begin{remark}
	It is clear from the description
	of $\Sigma \wr (-)$
	as a category of tuples that, 
	should $(\V,\otimes)$ be a \emph{symmetric} monoidal category,
	then $\otimes$
	induces a functor
	$\Sigma \wr \V \to \mathcal{V}$
	via $(v_i) \mapsto \bigotimes_i v_i$.
	
	In fact, one can reverse this process: a symmetric monoidal structure on $\V$ can be described as a functor
	$\Sigma \wr \V \to \V$
	satisfying suitable associativity and unitality conditions, and we will make use of this alternative description when defining operads.
	
	However, there is a slight caveat. 
	We will in fact prefer to describe a symmetric monoidal structure
	as a functor
	$\left(\Sigma \wr \V^{op}\right)^{op} \to \V$
	or, equivalently, 
	$\Sigma \wr \V^{op} \to \V^{op}$.
	The equivalence between this description and the one above
	follows since a (symmetric) monoidal structure $\otimes$ on $\V$ is also a monoidal structure on $\V^{op}$
	(or, using the isomorphism
	$\Sigma \simeq \Sigma^{op}$ given by $\sigma \mapsto \sigma^{-1}$,
	since $\Sigma \wr \V^{op} \simeq \left(\Sigma \wr \V \right)^{op}$).
	The reason for us to prefer this seemingly more cumbersome setup is because it actually seems to be more natural in practice.
	For example, if $\otimes = \Pi$ is the cartesian product then,
	in addition to symmetry isomorphisms, 
	$\Pi$ also admits projection maps and diagonals,
	and to encode these one must use a map
	$\left(\mathsf{F} \wr \mathcal{V}^{op}\right)^{op} \to \mathcal{V}$
	rather than
	$\mathsf{F} \wr \mathcal{V} \to \mathcal{V}$.
        For further discussion, see \cite[\S 2.2]{BP21}.
\end{remark}

\begin{remark}\label{WRDIAG REM}
	If $\mathcal{C}$ is a category with a $G$-action,
	the constructions in 
	Example \ref{GLTIMES EQ} and Notation \ref{SIGWR NOT}	
	can be applied in either order to obtain categories
	$G \ltimes (\Sigma \wr \mathcal{C})$
	and
	$\Sigma \wr (G \ltimes \mathcal{C})$.
	
	In either of these categories, the objects are the tuples
	$(c_i)$ with $c_i \in \mathcal{C}$.
	As for the arrows,
	in $G \ltimes (\Sigma \wr \mathcal{C})$
	an arrow $(c_i) \to (d_i)$
	between tuples of size $n$ is encoded by arrows 
	$gc_i \to d_{\sigma(i)}$
	in $\C$ for some $g\in G$, $\sigma \in \Sigma_n$
	while in $\Sigma \wr (G \ltimes \mathcal{C})$
	such an arrow is encoded by arrows
	$g_i c_i \to d_{\sigma(i)}$
	in $\C$ for some $(g_i) \in G^{\times n}$, $\sigma \in \Sigma_n$.
	Hence, we see that there is an inclusion of categories
	\begin{equation} 
	G \ltimes (\Sigma \wr \mathcal C) \to \Sigma \wr (G \ltimes \mathcal C).
	\end{equation}
	Informally, $G \ltimes (\Sigma \wr \mathcal C)$
	is the subcategory where the only
	$G$-action arrows are diagonal action arrows
	(i.e. those corresponding to constant tuples $(g)_{1 \leq i \leq n} \in G^{\times n}$).
\end{remark}

\begin{remark}\label{LIMINFIBSUP REM}
	Let $\pi \colon \mathcal{C} \to \mathcal{B}$ be a Grothendieck fibration.
	Then, if the base $\mathcal{B}$ and fibers 
	$\mathcal{C}_b$
	are all complete, so is the total category $\mathcal{C}$.
	Indeed, given a diagram $I \xrightarrow{c_{\bullet}} \mathcal{C}$,\
	and writing
	$b = \lim_{i \in I} \pi(c_i)$
	and 
	$\varphi_i \colon b \to \pi(c_i)$
	for the canonical maps,
	one has
	\begin{equation}\label{LIMINFIB EQ}
	\lim_{i \in I} c_i = 
	\lim_{i \in I} \varphi_{i}^{\**}c_i
	\end{equation}
	where the second limit formula is computed in $\mathcal{C}_b$.
	
	Moreover, keeping the setup above, let $\bar{\varphi}_i \colon \bar{b} \to \pi(c_i)$
	be an arbitrary cone in $\mathcal{B}$
	and $\bar{\varphi} \colon \bar{b} \to b$
	the induced map.
	Then \eqref{LIMINFIB EQ} implies that one further has
	\begin{equation}\label{LIMINFIBSUP EQ}
	\bar{\varphi}^{\**}\left(\lim_{i \in I} c_i\right) 
	= 
	\bar{\varphi}^{\**}\left(\lim_{i \in I} \varphi_{i}^{\**}c_i\right)
	=
	\lim_{i \in I} \bar{\varphi}^{\**} \varphi_{i}^{\**}c_i
	=
	\lim_{i \in I} \bar{\varphi}_{i}^{\**}c_i.
	\end{equation}
\end{remark}

\subsection{Fibered adjunctions and fibered monads}
\label{FIBCAT_SEC}

\begin{definition}\label{FIBADJ DEF}
	Let 
	$\pi \colon \mathcal{C} \to \mathcal{B}$,
	$\pi \colon \mathcal{D} \to \mathcal{B}$
	be functors with a common target.
	A \emph{fibered adjunction} is an adjunction
	\[
	L \colon \mathcal{C} \rightleftarrows \mathcal{D} \colon R
	\]
	where the functors $L,R$, 
	unit $\eta \colon id_{\mathcal{C}} \Rightarrow RL$ and 
	counit $\epsilon \colon LR \Rightarrow id_{\mathcal{D}}$
	are all fibered, i.e.
	\[
	\pi L=\pi, \qquad
	\pi R = \pi, \qquad
	\pi \eta = id_{\pi}, \qquad 
	\pi\epsilon = id_{\pi}.
	\]
\end{definition}

\begin{remark}
	A fibered adjunction induces natural isomorphisms
	(cf. Notation \ref{MAPSDEC NOT})
	\[
	\mathcal{D}_{\varphi}\left(Lc,d\right)
	\simeq
	\mathcal{C}_{\varphi}\left(c,Rd\right)
	\]
	for each $c\in \mathcal{C}$, $d \in \mathcal{D}$, 
	$\varphi \colon \pi(c)\to \pi(d)$. 
\end{remark}

\begin{proposition}\label{FIBADJCAR PROP}
	Let $L \colon \mathcal{C} \rightleftarrows \mathcal{D} \colon R$
	be an adjunction between Grothendieck fibrations
	over $\mathcal{B}$.
	
	If the adjunction is fibered then 
	$R$ is 
	a fibered functor which preserves cartesian arrows.
	
	Conversely, if the right adjoint $R$ is 
	a fibered functor which preserves cartesian arrows, 
	then one can modify the left adjoint (and unit, counit)
	so that the adjunction becomes a fibered adjunction.
\end{proposition}

\begin{proof}
	For the first claim, 
	letting $\Phi \colon \bar{d} \to d$ be a cartesian arrow
	in $\mathcal{D}$, 
	the fact that $R(\Phi)$ is again cartesian follows from
	Remark \ref{CARTCHAR REM} applied to the composite isomorphism
	\[
	\mathcal{C}_{\pi(\bar{d})}
	\left(-,R\bar{d}\right)
	\simeq 
	\mathcal{D}_{\pi(\bar{d})}
	\left(L(-),\bar{d}\right)
	\xrightarrow{\simeq}
	\mathcal{D}_{\pi(\Phi)}\left(L(-),d\right)
	\simeq
	\mathcal{C}_{\pi(\Phi)}\left(-,Rd\right).
	\]
	For the ``conversely'' claim,
	noting that by assumption it is $\pi RL = \pi L$,
	one can choose a cartesian natural transformation 
	$\bar{L} \xrightarrow{\gamma} L$
	(i.e. a cartesian arrow in $\mathcal{D}^{\mathcal{C}}$)
	over the projection of the adjunction unit
	$\pi \xrightarrow{\pi \eta} \pi RL$
	(which is an arrow in $\mathcal{E}^{\mathcal{C}}$).
	Moreover, noting that by assumption
	$R\bar{L} \xrightarrow{R \gamma} RL$ is again cartesian,
	we write
	$id_{\mathcal{C}} \xrightarrow{\bar{\eta}} 
	R \bar{L} \xrightarrow{R \gamma} RL$
	for the natural factorization
	with fibered $\bar{\eta}$,
	as well as $\bar{\epsilon}$ for the composite
	$\bar{L}R \xrightarrow{\gamma R} 
	LR \xrightarrow{\epsilon} id_{\mathcal{D}}$.
	We claim that $\bar{L},R,\bar{\eta},\bar{\epsilon}$
	now form a fibered adjunction, with the non obvious claim being that this is in fact still an adjunction.
	That the composite
	$R\xrightarrow{\bar{\eta}R} R\bar{L}R \xrightarrow{R\bar{\epsilon}} R$
	is the identity follows since
        we can rewrite this composite as
	$R \xrightarrow{\bar{\eta} R} R\bar{L}R 
	\xrightarrow{R \gamma R} RLR \xrightarrow{R \epsilon} R$ 
	and thus also as
        $R \xrightarrow{\eta R} RLR \xrightarrow{R \epsilon} R$.
	The remaining claim is that the top horizontal composite in the diagram below is the identity, 
	\begin{equation}
	\begin{tikzcd}
	\bar{L} \ar{d}[swap]{\gamma} \arrow{r}{\bar{L}\bar{\eta}}
	&
	\bar{L} R \bar{L} \arrow[d] \ar{r} \ar{d}
	\arrow[bend left]{rr}{\bar{\epsilon}\bar{L}}
	&
	L R \bar{L} \ar{r}{\epsilon \bar{L}} \ar{d}
	&
	\bar{L} \ar{d}{\gamma}
	\\
	L \arrow{r}{L\bar{\eta}}
	\arrow[bend right]{rr}[swap]{L \eta}
	&
	LR\bar{L} \ar{r}
	&
	LRL \ar{r}{\epsilon L}
	&
	L
	\end{tikzcd}
	\end{equation}
	and, since 
	$\gamma \colon \bar{L} \to L$ is cartesian, 
	it in fact suffices to show that the overall composite 
	$\bar{L} \to L$ in the diagram above
	is the map $\gamma$, which is clear. 
\end{proof}

\begin{remark}
	Let $\pi \colon \mathcal{C} \to \mathcal{B}$
	be a functor that preserves coproducts and $G$ be a group.
	Then the free-forget adjunctions 
	\begin{equation}
	\begin{tikzcd}[column sep =50pt]
	\mathcal{C}
	\arrow[shift left]{r}{G \cdot (-)}
	\arrow[d, "\pi"']
	&
	\mathcal{C}^G 
	\arrow[shift left]{l}{\mathsf{fgt}}
	\arrow[d, "\pi"]
	\\
	\mathcal{B} 
	\arrow[shift left]{r}{G \cdot (-)}
	&
	\mathcal{B}^G
	\arrow[shift left]{l}{\mathsf{fgt}}
	\end{tikzcd}
	\end{equation}
	are compatible in the sense that both the left and right adjoints commute with the projections $\pi$.
	In addition, given $b \in \mathcal{B}$, $b' \in \mathcal{B}^G$
	and $\varphi \colon b \to b'$ a map in $\mathcal{B}$, write
	$G\cdot \varphi \colon G\cdot b \to b'$
	for the adjoint map.
	Then, for $c \in \mathcal{C}_b$ and $c' \in \mathcal{C}^G_{b'}$,
	one has that the left isomorphism below decomposes
	into the isomorphisms on the right
	\begin{equation}\label{ADJOVADJ EQ}
	\C(c,c') 
	\simeq 
	\C^G(G \cdot c,c')\qquad
	\C_{\varphi}(c,c') 
	\simeq 
	\C^G_{G\cdot\varphi}(G \cdot c,c')
	\end{equation}
\end{remark}

\begin{definition}\label{FIBMON DEF}
	Given a functor $\pi \colon \mathcal{C} \to \mathcal{B}$,
	a \textit{fibered monad} is a monad $T \colon \mathcal{C} \to \mathcal{C}$ such that 
	the functor $T$,
	multiplication 
	$\mu \colon TT \Rightarrow T$
	and unit $\eta \colon I \Rightarrow T$
	are all fibered, i.e.
	\[
	\pi T = \pi,\qquad
	\pi\mu=\pi\eta=id_{\pi}.
	\]
	%
	Moreover, a \textit{fiber algebra} is a $T$-algebra $c \in \mathcal{C}$
	such that the multiplication map
	$Tc \xrightarrow{m} c$ satisfies 
	$\pi(m)=id_{\pi(c)}$.
	Lastly, we write $\mathsf{Alg}^{\pi}_T(\mathcal{C}) \subseteq \mathsf{Alg}_T(\mathcal{C})$ for the full subcategory of fiber algebras.
\end{definition}

\begin{remark}
        \label{MONFIB_REM}
	For each $b\in \mathcal{B}$, a fibered monad $T$ restricts to a monad on each fiber $\mathcal{C}_b$, and we write $T_b$ to denote this restricted monad.
\end{remark}

\begin{remark}
	If $T$ is a fibered monad then any free algebra $Tc$ is automatically a fiber algebra, so that the free $T$-algebra functor factors
	as 
	$\mathcal{C} \to \mathsf{Alg}^{\pi}_T(\mathcal{C}) \subseteq \mathsf{Alg}_T(\mathcal{C})$.
\end{remark}

\begin{proposition}\label{FIBALGGR PROP}
	Given a fibered monad on a Grothendieck fibration $\pi \colon \mathcal{C} \to \mathcal{B}$,
	the projection $\mathsf{Alg}^{\pi}_T(\mathcal{C}) \to \mathcal{B}$
	is again a Grothendieck fibration, with fibers
	$\left(\mathsf{Alg}^{\pi}_T(\mathcal C)\right)_b = \mathsf{Alg}_{T_b}(\mathcal C_b)$.
	
	Moreover, the free-algebra and forgetful functors then form a fibered adjunction
	$\mathcal{C} \rightleftarrows \mathsf{Alg}^{\pi}_T(\mathcal{C})$.
\end{proposition}

The key to this proof is that fiber algebra structures can be ``pulled back'' along cartesian arrows
(which, by Proposition \ref{FIBADJCAR PROP}, must be the case if $\mathcal{C} \rightleftarrows \mathsf{Alg}^{\pi}_T(\mathcal{C})$ is to be a fibered adjunction).

\begin{proof}
        The identification of the fibers is straightforward.
	Given a cartesian arrow $\Phi \colon \bar{c} \to c$ on $\mathcal{C}$ and a fiber algebra structure on $c$, we claim there is a unique fiber algebra structure on $\bar{c}$ making $\Phi$ into an algebra map. Indeed, the properties of cartesian arrows imply that there is a unique way to choose a dashed fiber arrow in the diagram
	\[
	\begin{tikzcd}
	T \bar{c} \ar{r}{T \Phi} \ar[dashed]{d} & T c \ar{d}
	\\
	\bar{c} \ar{r}[swap]{\Phi} & c
	\end{tikzcd}
	\]
	which provides the multiplication on $\bar{c}$.
	The claims that $T\bar{c} \to \bar{c}$ is an algebra structure and that 
	$\Phi$ is again cartesian when viewed as an algebra map likewise follow from 
	$\Phi$ being cartesian in $\mathcal{C}$.
	
	The ``moreover'' claim that 
	$\mathcal{C} \rightleftarrows \mathsf{Alg}^{\pi}_T(\mathcal{C})$
	is a fibered adjunction
	follows by noting that the adjunction unit is the unit
	$I \Rightarrow T$ of the monad $T$, 
	which is fibered by assumption,
	while the counit, evaluated on a fiber algebra $c$, is the multiplication $Tc \to c$, and hence fibered by the definition of fiber algebras.
\end{proof}

\begin{remark}\label{FIBMONCL REM}
	Suppose a cleavage of a Grothendieck fibration
	$\pi \colon \mathcal{C} \to \mathcal{B}$ 
	has been chosen.
	
	A fibered monad $T$ is then equivalent to the data of the fiber monads 
	$T_b$ on the fibers $\mathcal{C}_b$ together with,
	for each arrow $\varphi \colon b' \to b$ in $\mathcal{B}$,
	natural transformations
	$T_{b'} \varphi^{\**} \Rightarrow \varphi^{\**} T_{b}$,
	such that
	\begin{itemize}
		\item[(a)]
		for composites $b'' \xrightarrow{\psi} b' \xrightarrow{\varphi} b$
		and identities $b \xrightarrow{id_b} b$,
		the induced diagrams below commute
		\begin{equation}\label{GROTHASS EQ}
		\begin{tikzcd}
		T_{b''} \psi^{\**} \varphi^{\**} 
		\ar[Rightarrow]{r} \ar[Leftrightarrow]{d}[swap]{\simeq} &
		\psi^{\**} T_{b'} \varphi^{\**} \ar[Rightarrow]{r} &
		\psi^{\**} \varphi^{\**}  T_{b} \ar[Leftrightarrow]{d}{\simeq} &
		T_b \ar[equal]{r} \ar[Leftrightarrow]{d}[swap]{\simeq} &
		T_b \ar[Leftrightarrow]{d}{\simeq}
		\\
		T_{b''} (\varphi \psi)^{\**} \ar[Rightarrow]{rr}{} &&
		(\varphi \psi)^{\**} T_{b} &
		T_b id_b^{\**} \ar[Rightarrow]{r}{} &
		id_b^{\**} T_{b}
		\end{tikzcd}
		\end{equation}
		\item[(b)] the natural squares below,
		with vertical maps induced by the monads $T_{b'},T_b$, commute 
		\begin{equation}\label{GROTHCART EQ}
		\begin{tikzcd}
		T_{b'} T_{b'} \varphi^{\**} \ar[Rightarrow]{r} \ar[Rightarrow]{d}[swap]{} &
		T_{b'} \varphi^{\**} T_{b} \ar[Rightarrow]{r} &
		\varphi^{\**} T_{b} T_{b} \ar[Rightarrow]{d}{} &
		\varphi^{\**} \ar[equal]{r} \ar[Rightarrow]{d}&
		\varphi^{\**} \ar[Rightarrow]{d}
		\\
		T_{b'} \varphi^{\**} \ar[Rightarrow]{rr}{} &&
		\varphi^{\**} T_{b} &
		T_{b'} \varphi^{\**} \ar[Rightarrow]{r}{} &
		\varphi^{\**} T_{b}
		\end{tikzcd}
		\end{equation}
	\end{itemize}
\end{remark}

\begin{remark}\label{ABSPUSH REM}
	Suppose the Grothendieck fibration 
	$\pi \colon \mathcal{C} \to \mathcal{B}$
	in Remark \ref{FIBMONCL REM}
	is also an opfibration so that,
	by Remark \ref{ALSOOPADJ REM},
	the cleavage functors $\varphi^{\**}$ for 
	$\varphi \colon b' \to b$
	admit left adjoints $\varphi_!$.
	
	Then $\varphi^{\**} T_b \varphi_!$ has a monad structure
	(which combines the multiplication and unit of $T_b$ with the unit and counit of the $(\varphi_!,\varphi^{\**})$ adjunction)
	and commutativity of the diagrams in $\eqref{GROTHCART EQ}$
	is equivalent to the claim that the induced natural transformation
	$T_{b'} \Rightarrow \varphi^{\**}T_{b}\varphi_{!}$
	is a map of monads.
\end{remark}

\begin{remark}\label{ALGPUSHLL REM}
	Suppose that both $\mathcal{C}$ and $\mathsf{Alg}_T^{\pi}(\mathcal{C})$
	admit adjunctions as in 
	Remark \ref{ABSPUSH REM} for each map $\varphi \colon b' \to b$.
	We denote these two adjunctions by
	\[
	\varphi_! \colon \mathcal{C}_{b'} 
	\rightleftarrows
	\mathcal{C}_{b}\colon \varphi^{\**}
	\qquad
	\check{\varphi}_! \colon \mathsf{Alg}_{T_{b'}}(\mathcal{C}_{b'}) 
	\rightleftarrows 
	\mathsf{Alg}_{T_{b}}(\mathcal{C}_{b})\colon \varphi^{\**}
	\]
	to emphasize the fact that the algebraic cleavage functor $\varphi^{\**}$ lifts the underlying $\varphi^{\**}$ (cf. the proof of Proposition \ref{FIBALGGR PROP}).

	On the other hand, the algebraic $\check{\varphi}_!$ functor is not a lift of
	the underlying $\varphi_!$.
	Rather, by the dual of Proposition \ref{FIBADJCAR PROP},
	one has that $T \colon \mathcal{C} \to \mathsf{Alg}_T^{\pi}(\mathcal{C})$ preserves cocartesian arrows,
	so that the map $T_{b'} \Rightarrow \check{\varphi}_! T_{b}$
	consisting of algebraic cocartesian arrows
	is identified with 
	$T_{b'} \Rightarrow T_{b} \varphi_!$,
	i.e. the image under $T$
	of the natural map
	$id_{\mathcal{C}_{b'}} \Rightarrow \varphi_!$
	consisting of underlying cocartesian arrows in $\mathcal{C}$.

	Thus, since any 
	$c \in \mathsf{Alg}_{T_{b'}}(\mathcal{C}_{b'})$
	is given by a coequalizer of free algebras
	$c \simeq \mathrm{coeq}(T_{b'}T_{b'} c \rightrightarrows T_{b'} c)$,
	one has the formula
	$\check{\varphi}_! c \simeq 
	\mathrm{coeq}(T_{b}\varphi_!T_{b'} c \rightrightarrows T_{b}\varphi_! c)$.
\end{remark}

\begin{proposition}\label{DIAGRAMFM_PROP}
	Let $I$ be a fixed small category, and $T$ a fibered monad with respect to a Grothendieck fibration 
	$\pi \colon \mathcal{C} \to \mathcal{B}$. Then:
	\begin{itemize}
		\item[(i)] $T^I$ is a fibered monad with respect to $\pi^I\colon \mathcal{C}^I \to \mathcal{B}^I$;
		\item[(ii)] there is a natural identification 
		$\mathsf{Alg}_{T^I}^{\pi^I}(\mathcal{C}^I)
		\simeq
		\left(\mathsf{Alg}_T^{\pi}(\mathcal{C})\right)^I$.
	\end{itemize}
\end{proposition}

\begin{proof}
	Both parts follow readily from the definitions.
\end{proof}

\section{Equivariant colored operads}
\label{ECO_SEC}

This section discusses the categories of equivariant colored symmetric sequences and operads.
First, \S \ref{EQCOSYMSEQ SEC} applies the abstract theory from \S \ref{PRE SEC}
to the categories
$\mathsf{Sym}_{\bullet}(\V)$,
$\mathsf{Sym}^G_{\bullet}(\V)$
of symmetric sequences with varying sets of colors 
(Definition \ref{CSSYM DEF}), 
with most of the work being devoted to providing 
explicit descriptions of the equivariant fibers
$\mathsf{Sym}^G_{\mathfrak{C}}(\V)$
for $\mathfrak{C} \in \mathsf{Set}^G$
as a presheaf category,
given in Proposition \ref{EQUIVFNCON PROP}. 
Then, \S \ref{REPFUN_SEC} provides a convenient description of the representable functors in these fibers in Proposition \ref{REPALTDESC PROP}.
Lastly, in \S \ref{EQCOSYMOP SEC}, 
we describe the fibered monad $\mathbb F$ on $\Sym_\bullet(\V)$
which defines operads $\mathsf{Op}_\bullet(\V)$,
cf. \eqref{FIBADJMON EQ},
as given in Definition \ref{FREEOP DEF}
(and elaborated on in Appendix \ref{MONAD_APDX}). Moreover, 
in \eqref{FGC_EQ}
we determine the fibers of the equivariant monad $\mathbb F^G$ on $\Sym_\bullet^G(\V)$.

\subsection{Equivariant colored symmetric sequences}
\label{EQCOSYMSEQ SEC}

We now discuss the categories
$\mathsf{Sym}_{\bullet}(\V)$,
$\mathsf{Sym}^G_{\bullet}(\V)$
of symmetric sequences appearing in \eqref{FIBADJMON EQ}.

\begin{definition}\label{CSYM DEF}
	Let $\mathfrak {C} \in \mathsf{Set}$ be a fixed set of colors (or objects).
	A tuple
	$\vect C = (\mathfrak c_1, \dots, \mathfrak c_n; \mathfrak c_0) \in \mathfrak{C}^{\times n+1}$
	is called a \textit{$\mathfrak {C}$-profile} of \textit{arity} $n$.
	The \textit{$\mathfrak C$-symmetric category} $\Sigma_{\mathfrak C}$ is the category whose objects are the $\mathfrak{C}$-profiles and whose morphisms are action maps
	\begin{equation}\label{CSYM EQ1}
	\vect{C} =
	(\mathfrak c_1, \dots, \mathfrak c_n; \mathfrak c_0) \xrightarrow{\sigma} (\mathfrak c_{\sigma^{-1}(1)}, \dots, \mathfrak c_{\sigma^{-1}(n)}; \mathfrak c_0)
	= \vect{C} \sigma^{-1}
	\end{equation}
	for each permutation $\sigma \in \Sigma_n$, with the natural notion of composition.
	
	Alternatively, we will find it useful to visualize profiles as corollas (i.e. trees with a single node)
	with edges decorated by colors in $\mathfrak{C}$, as depicted below, so that the map labeled $\sigma$
	is the unique map of trees indicated such that the coloring of an edge equals the coloring of its image.
	\begin{equation}\label{CSYM EQ2}
	\begin{tikzpicture}
	[grow=up,auto,level distance=2.3em,every node/.style = {font=\footnotesize},dummy/.style={circle,draw,inner sep=0pt,minimum size=1.75mm}]
	
	\node at (0,0) [font=\normalsize]{$\vect{C}$}
	child{node [dummy] {}
		child{
			edge from parent node [swap,near end] {$\mathfrak c_n$} node [name=Kn] {}}
		child{
			edge from parent node [near end] {$\mathfrak c_1$}
			node [name=Kone,swap] {}}
		edge from parent node [swap] {$\mathfrak c_0$}
	};
	\draw [dotted,thick] (Kone) -- (Kn) ;
	\node at (5,0) [font=\normalsize] {$\vect{C} \sigma^{-1}
		$}
	child{node [dummy] {}
		child{
			edge from parent node [swap,near end] {$\mathfrak c_{\sigma^{-1}(n)}$} node [name=Kn] {}}
		child{
			edge from parent node [near end] {$\mathfrak c_{\sigma^{-1}(1)}$}
			node [name=Kone,swap] {}}
		edge from parent node [swap] {$\mathfrak c_0$}
	};
	\draw [dotted,thick] (Kone) -- (Kn) ;
	
	\draw[->] (1.5,0.8) -- node{$\sigma$} (3,0.8);
	\end{tikzpicture}
	\end{equation}
	Given any map of color sets $\varphi \colon \mathfrak{C} \to \mathfrak{D}$,
	there is a functor (abusively written)
	$\varphi \colon \Sigma_{\mathfrak{C}} \to \Sigma_{\mathfrak{D}}$,
	given by 
	$\varphi (\mathfrak c_1, \dots, \mathfrak c_n; \mathfrak c_0) = (\varphi(\mathfrak c_1),\cdots,\varphi(\mathfrak c_n);\varphi(\mathfrak c_0))$. 
\end{definition}

\begin{remark}\label{GLOBSIG REM}
	The notation $\vect{C} \sigma^{-1}$
	in \eqref{CSYM EQ1},\eqref{CSYM EQ2}
	reflects the fact that $\Sigma_n$
	acts on the right on $\mathfrak{C}$-profiles of arity $n$
	via 
	$\vect{C} \sigma = (\mathfrak{c}_i)\sigma = 
	(\mathfrak{c}_{\sigma(i)})$,
	where we make the convention that $\sigma(0)=0$.
\end{remark}

\begin{remark}
	If one regards $\mathfrak{C} \in \mathsf{Set}$
	as a discrete category, 
	following Notation \ref{SIGWR NOT} one has an identification 
	of groupoids
	$\Sigma_{\mathfrak{C}} = (\Sigma \wr \mathfrak{C}) \times \mathfrak{C}$,
	where the 
	$\Sigma \wr \mathfrak{C}$ factor accounts for the sources
	$\mathfrak{c}_1,\cdots,\mathfrak{c}_n$ of a profile
	and the $\mathfrak{C}$ factor accounts for the target 
	$\mathfrak{c}_0$.
\end{remark}

\begin{notation}
	We will (slightly abusively) write
	$\Sigma_{\bullet} \to \mathsf{Set}$
	for the Grothendieck construction (see Notation \ref{GROTHCONS NOT})
	of the functor
	$\mathsf{Set} \to \mathsf{Cat}$ given by
	$\mathfrak{C} \mapsto \Sigma_{\mathfrak{C}}$.
	
	More explicitly, 
	the objects of $\Sigma_{\bullet}$
	are the $\vect{C} \in \Sigma_{\mathfrak{C}}$ 
	for some set of colors $\mathfrak{C}$
	and an arrow from
	$\vect{C} \in \Sigma_{\mathfrak{C}}$ to
	$\vect{D} \in \Sigma_{\mathfrak{D}}$
	over $\varphi \colon \mathfrak{C} \to \mathfrak{D}$
	is an arrow
	$\varphi \vect{C} \to \vect{D}$ in $\Sigma_{\mathfrak{D}}$.
\end{notation}

\begin{definition}\label{CSSYM DEF}
	Let $\mathcal{V}$ be a category.
	The category $\mathsf{Sym}_\bullet(\mathcal{V})$ of
	\textit{symmetric sequences on $\mathcal{V}$} 
	(on all colors) is the category with:
	\begin{itemize}
		\item objects given by pairs $(\mathfrak C, X)$ with
		$\mathfrak{C} \in \mathsf{Set}$ a set of colors and
		$X \colon \Sigma_{\mathfrak{C}}^{op} \to \mathcal{V}$ a functor;
		\item arrows $(\mathfrak C, X) \to (\mathfrak D, Y)$ given by a map 
		$\varphi \colon \mathfrak{C} \to \mathfrak{D}$ of colors and a natural transformation $X \Rightarrow Y \varphi$ as below.
		\begin{equation}
		\begin{tikzcd}[row sep = tiny, column sep = 45pt]
		\Sigma_{\mathfrak{C}}^{op} \arrow[dr, "X"{name=U}] 
		\arrow{dd}[swap]{\varphi}
		\\
		& \mathcal{V}
		\\
		|[alias=V]| \Sigma_{\mathfrak{D}}^{op} \arrow[ur, "Y"']
		\arrow[Leftarrow, from=V, to=U,shorten >=0.25cm,shorten <=0.25cm]
		\end{tikzcd}
		\end{equation}
	\end{itemize}
\end{definition}

\begin{remark}\label{SUBCATDOWNL REM}
	Given a category $\V$, let $\Cat \downarrow^l \V$ denote the category with
	objects given by functors $\mathcal{C} \to \V$,
	and arrows from $\mathcal{C} \to \V$ to $\mathcal{D} \to \V$
	given by pairs 
	$(\varphi,\phi)$ with 
	$\varphi \colon \mathcal{C} \to \mathcal{D}$ a functor and
	$\phi$ a natural transformation
	\begin{equation}\label{SUBCATDOWNL EQ}
	\begin{tikzcd}[row sep = tiny, column sep = 45pt]
	\mathcal C \arrow{dr}[name=U]{} \arrow{dd}[swap]{\varphi}
	\\
	& \mathcal{V}
	\\
	|[alias=V]| \mathcal D \arrow{ur}[swap]{}
	\arrow[Leftarrow, from=V, to=U,shorten >=0.25cm,shorten <=0.25cm
	,swap,"\phi"]
	\end{tikzcd}
	\end{equation}
	Then $\mathsf{Sym}_{\bullet}(\mathcal{V})$ is naturally a (neither wide nor full) subcategory of  $\mathsf{Cat}\downarrow^l \mathcal{V}$.
\end{remark}

\begin{remark}\label{COLCHADJ REM}
	We caution that 
	$\mathsf{Sym}_{\bullet}(\V)$
	is quite different from the presheaf category 
	$\mathsf{Fun}(\Sigma_{\bullet}^{op},\V)$.

	Instead, one should think of 
	$\mathsf{Sym}_{\bullet}(\V)$
	as a form of ``fibered presheaves''.
	More precisely,
        the obvious functor
        $\mathsf{Sym}_{\bullet}(\V) \to \mathsf{Set}$
        is a (split) Grothendieck fibration,
	with fiber over 
	$\mathfrak{C} \in \mathsf{Set}$
	the presheaf category
	$\Sym_{\mathfrak C}(\V)=
	\mathsf{Fun}(\Sigma_{\mathfrak{C}}^{op},\mathcal{V})$.

	Note that,
	for any map $\varphi \colon \mathfrak{C} \to \mathfrak{D}$,
	one has adjunctions
	\begin{equation}\label{COLCHADJ EQ}
	\varphi_! \colon \Sym_{\mathfrak C}(\V) 
	\rightleftarrows 
	\Sym_{\mathfrak D}(\V) \colon \varphi^{\**}
	\end{equation}
	where $\varphi^{\**}$
	(resp. $\varphi_!$)
	is precomposition with (resp. left Kan extension along)
	$\varphi\colon 
	\Sigma^{op}_{\mathfrak{C}}
	\to 
	\Sigma^{op}_{\mathfrak{D}}$,
	so that the Grothendieck fibration
	$\mathsf{Sym}_{\bullet}(\V) \to \mathsf{Set}$
	is also an opfibration. 
\end{remark}

\begin{remark}\label{SUBCOCART REM}
	The forgetful functor $\mathsf{Cat} \downarrow^l \mathcal{V} \to \mathsf{Cat}$
	remembering only the source category is likewise both a fibration and an opfibration, 
	with a diagram \eqref{SUBCATDOWNL EQ}
	being a cartesian arrow if  $\phi$ an isomorphism
	and cocartesian if it is a left Kan extension. 
\end{remark}

Building on Remark \ref{COLCHADJ REM},
one can define a fibered Yoneda embedding, as in the following,
where we abbreviate
$\mathsf{Sym}_{\bullet} = \mathsf{Sym}_{\bullet}(\mathsf{Set})$.

\begin{notation}\label{FIBYON NOT}
	Let $\mathfrak{C} \in \mathsf{Set}$, $\vect{C} \in \Sigma_{\mathfrak{C}}$.
	We write
	$\Sigma_{\mathfrak{C}}[\vect{C}] 
	\in \mathsf{Sym}_{\mathfrak{C}} = \mathsf{Set}^{\Sigma_{\mathfrak{C}}^{op}}$ for the representable presheaf
	\[\Sigma_{\mathfrak{C}}[\vect{C}](-)
	= \Sigma_{\mathfrak{C}}(-,\vect{C}).\]
	Moreover, we define the \emph{fibered Yoneda embedding}
	\begin{equation}\label{FIBYON EQ}
	\Sigma_{\bullet} \xrightarrow{\Sigma_{\bullet}[-]} \mathsf{Sym}_{\bullet}
	\end{equation}
	to be
        $\Sigma_{\mathfrak{C}}[\vect{C}]$
	when evaluated on an object
	$\vect{C} \in \Sigma_{\mathfrak{C}} \subset \Sigma_{\bullet}$,
	and on an arrow 
	$\varphi \vect{C} \to \vect{D}$
	over $\varphi \colon \mathfrak{C} \to \mathfrak{D}$
	to be
        the natural transformation
	$\Sigma_{\mathfrak{C}}[\vect{C}]
	\Rightarrow
	\varphi^{\**}
	\Sigma_{\mathfrak{D}}[\vect{D}]
	$ given by the natural composites
	\[\Sigma_{\mathfrak{C}}[\vect{C}](-)
	= \Sigma_{\mathfrak{C}}(-,\vect{C})
	\to 
	\Sigma_{\mathfrak{D}}(\varphi(-),\varphi\vect{C})
	\to
	\Sigma_{\mathfrak{D}}(\varphi(-),\vect{D})
	=
	\varphi^{\**} \Sigma_{\mathfrak{D}}[\vect{D}](-).
	\]
\end{notation}

\begin{proposition}\label{FIBYONPUSH PROP}
	Let $\vect{C} \in \Sigma_{\mathfrak{C}}$,
	$\varphi \colon \mathfrak{C} \to \mathfrak{D}$
	be a map of colors and consider the adjunction \eqref{COLCHADJ EQ}.
	
	Then the adjoint of the canonical map
	$\Sigma_{\mathfrak{C}}[\vect{C}]
	\to
	\varphi^{\**}\Sigma_{\mathfrak{D}}[\varphi\vect{C}]$
	in $\mathsf{Sym}_{\mathfrak{C}}$
	is an isomorphism
	$\varphi_{!} \Sigma_{\mathfrak{C}}[\vect{C}]
	\xrightarrow{\simeq}
	\Sigma_{\mathfrak{D}}[\varphi\vect{C}]$
	in $\mathsf{Sym}_{\mathfrak{D}}$.
\end{proposition}

Alternatively, this result states that 
the fibered Yoneda $\Sigma_{\bullet}[-]$
preserves cocartesian arrows.

\begin{proof}
	For $Y \in \mathsf{Sym}_{\mathfrak{D}}$, by the (usual) Yoneda lemma one has isomorphisms
	\begin{equation}\label{TWOYON EQ}
	\mathsf{Sym}_{\mathfrak{D}}
	(\Sigma_{\mathfrak{D}}[\varphi \vect{C}],Y)
	\simeq
	Y(\varphi \vect{C})
	\simeq
	\mathsf{Sym}_{\mathfrak{C}}(\Sigma_{\mathfrak{C}}[\vect{C}],\varphi^{\**}Y)
	\end{equation}
	proving 
	$\Sigma_{\mathfrak{D}}[\varphi \vect{C}] \simeq \varphi_! \Sigma_{\mathfrak{C}}[\vect{C}]$.
	That this identification is adjoint to the canonical map
	$\Sigma_{\mathfrak{C}}[\vect{C}]
	\to
	\varphi^{\**}\Sigma_{\mathfrak{C}}[\varphi\vect{C}]$
	in $\mathsf{Sym}_{\mathfrak{C}}$
	follows since this sends
	$id_{\vect{C}}$ to $id_{\varphi\vect{C}}$
	(and since these identities determine the isomorphisms in \eqref{TWOYON EQ}).
\end{proof}

Now let $G$ be a group. Writing 
$\mathsf{Sym}^G_\bullet(\mathcal{V})
=
\left(\mathsf{Sym}_\bullet(\mathcal{V})\right)^G$
for the category of $G$-objects
on $\mathsf{Sym}_{\bullet}(\V)$,
the argument in Remark \ref{FUNISGROTH REM}
yields that
$\mathsf{Sym}^G_\bullet(\mathcal{V}) \to \mathsf{Set}^G$
is again a Grothendieck fibration.
For each $\mathfrak{C} \in \mathsf{Set}^G$,
we will then write
$\mathsf{Sym}^G_{\mathfrak{C}}(\V)$
to denote the fiber of
$\mathsf{Sym}^G_\bullet(\mathcal{V}) \to \mathsf{Set}^G$
over $\mathfrak{C}$.
We caution that 
$\mathsf{Sym}^G_{\mathfrak{C}}(\V)$
\emph{is not} the category 
$\left(\mathsf{Sym}_{\mathfrak{C}}(\V)\right)^G$
of $G$-objects on $\mathsf{Sym}_{\mathfrak{C}}(\V)$
unless the $G$-action on $\mathfrak{C}$
happens to be trivial.

We will thus find it convenient to have a more explicit description of
$\mathsf{Sym}^G_{\mathfrak{C}}(\V)$,
provided by the following result, 
which adapts
\cite[Lemma A.6]{BP21}.
In fact, we prove a more general description
for the categories $\mathsf{Cat} \downarrow^l \mathcal{V}$
in Remark \ref{SUBCATDOWNL REM},
for which the natural ``source category'' functor
$\mathsf{Cat} \downarrow^l \mathcal{V} \to \mathsf{Cat}$
is likewise a split Grothendieck fibration,
with fiber over $\mathcal{C} \in \mathsf{Cat}$
given by $\mathsf{Fun}(\mathcal{C},\mathcal{V})$.

\begin{proposition}\label{EQUIVFNCON PROP}
	Let $G$ be a group and $\mathfrak{C} \in \mathsf{Set}^G$
	be a $G$-set. There is then a natural identification
	\[
	\mathsf{Sym}^G_{\mathfrak{C}}(\V)
	\simeq
	\mathsf{Fun}(G \ltimes \Sigma^{op}_{\mathfrak{C}},\V).
	\]
	More generally, for a category $\mathcal{B}$,
	the fiber of
	$\left(\mathsf{Cat} \downarrow^l \V \right)^{\mathcal{B}}
	\to \mathsf{Cat}^{\mathcal{B}}$
	over a functor
	$\mathcal{B} \xrightarrow{b \mapsto \mathcal{C}_b} \mathsf{Cat}$
	is given by
	\begin{equation}\label{FUNBLTICV EQ}
	\mathsf{Fun}(\mathcal{B} \ltimes \mathcal{C}_{\bullet},\V).
	\end{equation}
\end{proposition}

\begin{proof}
	It is immediate from the definitions that
	$\mathsf{Sym}^G_{\mathfrak{C}}(\C)$
	matches the fiber of 
	$\left(\mathsf{Cat} \downarrow^l \V \right)^{G}
	\to \mathsf{Cat}^{G}$
	over $\Sigma_{\mathfrak{C}}^{op} \in \mathsf{Cat}^G$,
	so we need only address the general case.
	
	As in Notation \ref{GROTHCONS NOT},
	let us write $\varphi_! \colon \mathcal{C}_b \to \mathcal{C}_{b'}$
	for the functor induced by the arrow 
	$\varphi \colon b \to b'$
	in $\mathcal{B}$.
	Unpacking definitions, an object of 
	$\left(\mathsf{Cat} \downarrow^l \V \right)^{\mathcal{B}}$
	over 
	$\mathcal{C}_{\bullet} \colon \mathcal{B} \to \mathsf{Cat}$
	corresponds to the data of functors
	$\gamma_b \colon \mathcal{C}_b \to \mathcal{V}$ for each $b \in \mathcal{B}$
	and natural transformations
	$\phi_{\varphi} \colon \gamma_b \Rightarrow \gamma_{b'} \varphi_!$
	for each $\varphi \colon b \to b'$ in $\mathcal{B}$
	\begin{equation}\label{EQUIVFUNEX EQ}
	\begin{tikzcd}[row sep = tiny, column sep = 50pt]
	\mathcal{C}_{b} \arrow{dr}[name=U]{\gamma_b} \arrow{dd}[swap]{\varphi_!}
	\\
	& \mathcal{V}
	\\
	|[alias=V]| \mathcal{C}_{b'} \arrow{ur}[swap]{\gamma_{b'}}
	\arrow[Leftarrow, from=V, to=U,shorten >=0.25cm,shorten <=0.25cm
	,swap, near end, "\phi_{\varphi}"
	]
	\end{tikzcd}
	\end{equation}
	subject to the requirements that
	$\gamma_b \overset{\phi_{\varphi}}{\Rightarrow}
	\gamma_{b'} \varphi_!
	\overset{\phi_{\psi}\varphi_!}{\Rightarrow}
	\gamma_{b'} \varphi_! \psi_!
	$
	equals 
	$\gamma_b \overset{\phi_{\psi\varphi}}{\Rightarrow}
	\gamma_{b'} \varphi_! \psi_!
	$
	for composable $b \xrightarrow{\varphi} b' \xrightarrow{\psi} b''$
	and that
	$\phi_{id_b} = id_{\gamma_b}$.
	
	We now claim that the data above is exactly the data of a functor
	$F \colon \mathcal{B} \ltimes \mathcal{C}_{\bullet} \to \mathcal{V}$.
	To ease notation, we follow
	Remark \ref{GROTHUNPCK REM}
	and describe arrows in 
	$\mathcal{B} \ltimes \mathcal{C}_{\bullet}$
	as composites
	$c \rightsquigarrow \varphi_! c \xrightarrow{f} c'$.
	On fiber arrows 
	$f\colon c \to c'$ in $\mathcal{C}_b$
	set $F(f) = \gamma_b(f)$,
	and on cocartesian arrows set
	$F(c \rightsquigarrow \varphi_! c) =
	(\phi_{\varphi})_c$.
	Then:
	\begin{enumerate*}[label=(\roman*)]
		\item
		associativity and unitality of $F$ with respect to fiber arrows is equivalent to associativity and unitality of the $\gamma_b$;
		\item associativity and unitality of $F$ with respect to the cocartesian arrows is equivalent to the conditions following 
		\eqref{EQUIVFUNEX EQ};
		\item $F$ respecting the commutative squares \eqref{GROTHUNPCK EQ}
		is equivalent to naturality of the $\phi_{\varphi}$.
	\end{enumerate*}
	
	This finishes the proof that \eqref{FUNBLTICV EQ} has the correct set of objects. The claim that it also has the correct arrows is similar.
\end{proof}

Following the result above, 
we will represent $G$-equivariant symmetric sequences
$X \in \mathsf{Sym}^G_{\mathfrak{C}}(\mathcal{V})$
for some $\mathfrak{C} \in \mathsf{Set}^G$
as functors
$X \colon G \ltimes \Sigma_{\mathfrak{C}}^{op} \to \mathcal{V}$,
and objects of
$\left( \mathsf{Cat} \downarrow^{l} \mathcal{V} \right)^G$
as functors $G \ltimes \mathcal{C} \to \V$.

\begin{remark}\label{RHOPURP REM}
	Recall that, if $\mathcal{C}\in \mathsf{Cat}^G$ is a category with a $G$-action,
	then 
	$\Sigma \wr \mathcal{C}$
	has a $G$-action where $G$ acts diagonally on tuples $(c_i)$ with $c_i \in \mathcal{C}$.
	Therefore, by applying the 
	$\Sigma \wr (-)$
	construction to the categories
	$\mathsf{Cat} \downarrow^l \V$
	in Remark \ref{SUBCATDOWNL REM}
	(note that that one can apply $\Sigma \wr (-)$
	to the entirety of diagram \eqref{SUBCATDOWNL REM}),
	one obtains a description of the functor
	\begin{equation}\label{RHOPURP EQ}
	\begin{tikzcd}[column sep =40,row sep =0]
	\left( \mathsf{Cat} \downarrow^{l} \mathcal{V} \right)^G
	\ar{r}{\Sigma \wr (-) } &
	\left( \mathsf{Cat} \downarrow^{l} \Sigma \wr \mathcal{V} \right)^G
	\\
	G \ltimes \mathcal{C} \to \mathcal{V} \ar[mapsto]{r} &
	\left(G \ltimes (\Sigma \wr \mathcal{C}) \to 
	\Sigma \wr (G \ltimes  \mathcal{C}) \to \Sigma \wr \mathcal{V}\right)
	\end{tikzcd}
	\end{equation}
	where the functor
	$G \ltimes (\Sigma \wr \mathcal{C}) \to 
	\Sigma \wr (G \ltimes  \mathcal{C})$
	is the natural inclusion described in Remark \ref{WRDIAG REM}.
\end{remark}

\subsection{Representable functors}
\label{REPFUN_SEC}

Since  
$\mathsf{Sym}^G_{\mathfrak{C}}(\mathcal{V}) \simeq 
\V^{G \ltimes \Sigma^{op}_{\mathfrak{C}}}$ is again a functor category, 
we will find it useful to discuss its representable functors.
In the following discussion 
we set $\V = \mathsf{Set}$ and abbreviate
$\mathsf{Sym}^G_{\mathfrak{C}} = \mathsf{Sym}^G_{\mathfrak{C}}(\mathsf{Set})$.

As the objects of 
$G \ltimes \Sigma^{op}_{\mathfrak{C}}$
are simply the profiles
$\vect{C} \in \Sigma_{\mathfrak{C}}$,
each such profile induces a representable functor
$\left(G \ltimes \Sigma^{op}_{\mathfrak{C}}\right)(\vect{C},-)$
in 
$\mathsf{Sym}^G_{\mathfrak{C}}$.
However, caution is needed.
If one forgets the $G$-action on $\mathfrak{C}$,
then $\vect{C} \in \Sigma_{\mathfrak{C}}$
likewise induces a representable functor 
$\Sigma_{\mathfrak{C}}[\vect{C}]=\Sigma^{op}_{\mathfrak{C}}(\vect{C},-)$
in $\mathsf{Sym}_{\mathfrak{C}}$,
as discussed in Notation \ref{FIBYON NOT}.
As such, our next task is to understand the relation between 
$\left(G \ltimes \Sigma^{op}_{\mathfrak{C}}\right)(\vect{C},-)$
and 
$\Sigma_{\mathfrak{C}}[\vect{C}]$
in such a way that we have analogues of the 
fibered Yoneda \eqref{FIBYON EQ}
and Proposition \ref{FIBYONPUSH PROP}.

As in our previous work in 
\cite[Not. 5.56]{Per18}, \cite[\S 2.3]{BP20},
the key to achieve this will be to extend the
$\Sigma_{\mathfrak{C}}[\vect{C}]$
notation to be defined not only for corollas $\vect{C}$
but also for ``colored forests of corollas''.
In fact, we will do a little more. 
In anticipation of our discussion of operads, 
we will extend this notation by defining 
$\Sigma_{\mathfrak{C}}[\vect{F}]$
for $\vect{F}$ a general colored forest (of trees).
In the following, $\Phi$ denotes the category of forests
(i.e. formal coproducts of trees; see \cite[\S 5.1]{Per18}).

\begin{definition}\label{COLFOR DEF}
	Let $\mathfrak{C}$ be a set of colors.
	The category $\Phi_{\mathfrak{C}}$ of $\mathfrak{C}$-colored forests has
	\begin{itemize}
		\item objects pairs
		$\vect{F} = (F,\mathfrak{c})$
		where 
		$F\in \Phi$ is a forest
		and 
		$\mathfrak{c}\colon \boldsymbol{E}(F) \to \mathfrak{C}$ 
		is a coloring of its edges;
		\item a map
		$\vect{F}=(F,\mathfrak{c}) \to 
		(F',\mathfrak{c}') = \vect{F'}$
		is a map $\rho \colon F \to F'$ in $\Phi$
		such that
		$\mathfrak{c} = \mathfrak{c}' \rho$.
	\end{itemize}
	If $\varphi\colon \mathfrak{C} \to \mathfrak{D}$ is a map of colors,
	we write
	$\varphi \colon \Phi_{\mathfrak{C}} \to \Phi_{\mathfrak{D}}$
	for the functor sending 
	$\vect{F} = (F,\mathfrak{c})$
	to 
	$\varphi\vect{F} = (F,\varphi\mathfrak{c})$.
	Note that this defines a Grothendieck opfibration
	\begin{equation}\label{PHIBGRO EQ}
	\Phi_{\bullet} \to \mathsf{Set}
	\end{equation}
	whose objects are the $\vect{F} \in \Phi_{\mathfrak{C}}$ for some
	$\mathfrak{C} \in \mathsf{Set}$,
	and with an arrow
	from $\vect{F}$ to $\vect{F'}$
	over
	$\varphi \colon \mathfrak{C} \to \mathfrak{D}$
	given by an arrow
	$\varphi \vect{F} \to \vect{F'}$ in $\Phi_{\mathfrak{D}}$.
\end{definition}

For each vertex $v \in \boldsymbol{V}(F)$ in a forest,
we write $F_v$ for the associated corolla.
Note that, given a $\mathfrak{C}$-coloring $\vect{F}$ on $F$,
one one likewise obtains colorings $\vect{F}_v$ on $F_v$.

\begin{notation}
	Given $\vect{F} \in \Phi_{\mathfrak{C}}$
	we define
	\begin{equation}\label{GENSIGC EQ}
	\Sigma_{\mathfrak{C}}[\vect{F}]=
	\coprod\nolimits^{\mathfrak{C}}_{v \in \boldsymbol{V}(F)} 
	\Sigma_{\mathfrak{C}}[\vect{F}_v]
	\end{equation}
	where we highlight that the coproduct $\amalg^{\mathfrak{C}}$ is fibered, i.e. it takes place in $\mathsf{Sym}_{\mathfrak{C}}$
	rather than $\mathsf{Sym}_{\bullet}$.
\end{notation}

\begin{example}\label{COLFORES EX}
	Let 
	$\mathfrak{C} = \{ \mathfrak{a}, \mathfrak{b}, \mathfrak{c} \}$.
	On the left below we depict a $\mathfrak{C}$-colored forest 
	$\vect{F} = \vect{T} \amalg \vect{S}$
	with tree components $\vect{T},\vect{S}$.
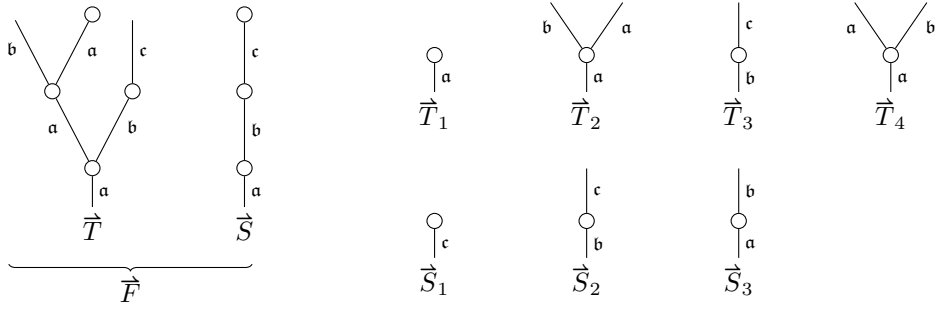
	\begin{equation}
	\begin{tikzpicture}[auto,grow=up, level distance = 2.2em,
	every node/.style={font=\scriptsize,inner sep = 2pt}]%
	\tikzstyle{level 2}=[sibling distance=3em]%
	\node at (0,0) [font = \normalsize] {$\vect{T}$}%
	child{node [dummy] {}%
		child[level distance = 2.9em]{node [dummy] {}%
			child{node {}%
				edge from parent node [swap] {$\mathfrak{c}$}}%
			edge from parent node [swap,near end] {$\mathfrak{b}$}}%
		child[level distance = 2.9em]{node [dummy] {}%
			child[level distance = 2.9em]{node [dummy] {}%
				edge from parent node [swap,	near end] {$\mathfrak{a}\phantom{\mathfrak{b}}$}}%
			child[level distance = 2.9em]{node {}%
				edge from parent node [near end] {$\mathfrak{b}$}}%
			edge from parent node [near end] {$\phantom{\mathfrak{b}}\mathfrak{a}$}}%
		edge from parent node [swap] {$\mathfrak{a}$}};%
	\node at (2,0) [font = \normalsize] {$\vect{S}$}%
	child{node [dummy] {}%
		child[level distance = 2.9em]{node [dummy] {}%
			child{node [dummy] {}%
				edge from parent node [swap] {$\mathfrak{c}$}}%
			edge from parent node [swap] {$\mathfrak{b}$}}%
		edge from parent node [swap] {$\mathfrak{a}$}};%
	\node at (4.5,1.5) [font = \normalsize] {$\vect{T}_1$}%
	child{node [dummy] {}%
		edge from parent node [swap] {$\mathfrak{a}$}};%
	\node at (6.5,1.5) [font = \normalsize] {$\vect{T}_2$}%
	child{node [dummy] {}%
		child{node {}%
			edge from parent node [swap, near end] {$\mathfrak{a}\phantom{\mathfrak{b}}$}}%
		child{node {}%
			edge from parent node [near end] {$\mathfrak{b}$}}%
		edge from parent node [swap] {$\mathfrak{a}$}};%
	\node at (8.5,1.5) [font = \normalsize] {$\vect{T}_3$}%
	child{node [dummy] {}%
		child{node {}%
			edge from parent node [swap] {$\mathfrak{c}$}}%
		edge from parent node [swap] {$\mathfrak{b}$}};%
	\node at (10.5,1.5) [font = \normalsize] {$\vect{T}_4$}%
	child{node [dummy] {}%
		child{node {}%
			edge from parent node [swap, near end] {$\mathfrak{b}$}}%
		child{node {}%
			edge from parent node [near end] {$\phantom{\mathfrak{b}}\mathfrak{a}$}}%
		edge from parent node [swap] {$\mathfrak{a}$}};%
	\node at (4.5,-0.7) [font = \normalsize] {$\vect{S}_1$}%
	child{node [dummy] {}%
		edge from parent node [swap] {$\mathfrak{c}$}};%
	\node at (6.5,-0.7) [font = \normalsize] {$\vect{S}_2$}%
	child{node [dummy] {}%
		child{node {}%
			edge from parent node [swap] {$\mathfrak{c}$}}%
		edge from parent node [swap] {$\mathfrak{b}$}};%
	\node at (8.5,-0.7) [font = \normalsize] {$\vect{S}_3$}%
	child{node [dummy] {}%
		child{node {}%
			edge from parent node [swap] {$\mathfrak{b}$}}%
		edge from parent node [swap] {$\mathfrak{a}$}};%
	\draw[decorate,decoration={brace,amplitude=2.5pt}] (2.1,-0.5) -- (-1.1,-0.5) 
	node[midway,inner sep=4pt,font=\normalsize]{$\vect{F}$}; %
	\end{tikzpicture}%
	\end{equation}%
	Moreover, on the right we depict the $\mathfrak{C}$-profiles/corollas
	$\vect{T}_i$ and $\vect{S}_j$
	corresponding to the vertices of $T,S$, so that
	\[
	\Sigma_{\mathfrak{C}}[\vect{F}]
	=
	\Sigma_{\mathfrak{C}}[\vect{T}] 
	\amalg^{\mathfrak{C}}
	\Sigma_{\mathfrak{C}}[\vect{S}]
	=
	\left(
	\coprod_{1\leq i \leq 4}^{\mathfrak{C}}
	\Sigma_{\mathfrak{C}}[\vect{T}_i] 
	\right)
	\amalg^{\mathfrak{C}}
	\left(
	\coprod_{1\leq j \leq 3}^{\mathfrak{C}}
	\Sigma_{\mathfrak{C}}[\vect{S}_j]
	\right)
	\]
\end{example}

\begin{remark}
	The representables $\Sigma_{\mathfrak{C}}[-]$ in \eqref{GENSIGC EQ}
	do not quite define a functor on the entire category  $\Phi_{\bullet}$,
	due to the fact that the only maps of forests 
	sending vertices to vertices are the outer maps \cite[\S 3.2]{BP21}.
	Writing $\Phi^o_{\bullet} \hookrightarrow \Phi_{\bullet}$
	for the wide subcategory of those arrows whose underlying maps of uncolored forests are outer (on each tree component),
	\eqref{FIBYON EQ}
	extends to give generalized fibered Yoneda embeddings
	(where the right functor is obtained from the left functor by taking $G$-objects)
	\[
	\Phi_{\bullet}^o 
	\xrightarrow{\Sigma_{\bullet}[-]}
	\mathsf{Sym}_{\bullet},
	\qquad
	\Phi_{\bullet}^{o,G}
	\xrightarrow{\Sigma_{\bullet}[-]}
	\mathsf{Sym}^G_{\bullet}
	\]
	which are fibered over $\mathsf{Set}$ and $\mathsf{Set}^G$, respectively.
	Note that, thanks to formula \eqref{GENSIGC EQ},
	Proposition \ref{FIBYONPUSH PROP} automatically generalizes, i.e. one has natural identifications
	\begin{equation}\label{PUSHEQAG EQ}
	\varphi_! \Sigma_{\mathfrak{C}}[\vect{F}] \simeq 
	\Sigma_{\mathfrak{D}}[\varphi \vect{F}]
	\end{equation}
	for each map of colors
	$\varphi \colon \mathfrak{C} \to \mathfrak{D}$
	(which is an equivariant map in the equivariant case).
\end{remark}

\begin{notation}
	We write $(-)^{\tau} \colon \Phi \to \Phi_{\bullet}$
	for the \emph{tautological coloring} functor
	which sends $F \in \Phi$ to 
	$F^{\tau} \in \Phi_{\boldsymbol{E}(T)}$
	where
	$F^{\tau} = (F,\mathfrak{t})$ is the underlying forest $F$
	together with the identity coloring
	$\mathfrak{t} \colon \boldsymbol{E}(T) \xrightarrow{=} \boldsymbol{E}(T)$.
	Moreover, we then abbreviate 
	$\Sigma_{\tau}[F] = \Sigma_{\boldsymbol{E}(F)}[F^{\tau}]$.
\end{notation}

\begin{remark}
	For any colored forest $\vect{F}=(F,\mathfrak{c})$,
	regarding $\mathfrak{c} \colon \boldsymbol{E}(T) \to \mathfrak{C}$
	as a change of color map, 
	one has $\vect{F} = \mathfrak{c} F^{\tau}$,
	so that \eqref{PUSHEQAG EQ} then yields 
	\begin{equation}\label{CANPUSH EQ}
	\Sigma_{\mathfrak{C}}[\vect{F}] = 
	\mathfrak{c}_! \Sigma_{\tau}[F]
	\end{equation}
\end{remark}

\begin{definition}
	Let $G$ be a group,
	$\mathfrak{C} \in \mathsf{Set}^G$
	be a $G$-set of colors, 
	and $\vect{C} \in \Sigma_{\mathfrak{C}}$ be a $\mathfrak{C}$-profile/corolla.
	Write
	$\vect{C} = (C,\mathfrak{c})$
	with $C\in \Sigma$ the underlying corolla
	and 
	$\mathfrak{c} \colon \boldsymbol{E}(T) \to \mathfrak{C}$.
	
	Writing 
	$G \cdot \mathfrak{c} \colon G \cdot \boldsymbol{E}(T) \to \mathfrak{C}$
	for the adjoint map, 
	$G \cdot C \in \Phi^G$
	for the $G$-free forest determined by $C$,
	and noting that
	$\boldsymbol{E}(G \cdot C) \simeq 
	G \cdot \boldsymbol{E}(C)$,
	we define $G \cdot_{\mathfrak{C}} \vect{T} \in \Phi^G_{\mathfrak{C}}$ by
	\begin{equation}\label{GCDTCC EQ}
	G \cdot_{\mathfrak{C}} \vect{C} = 
	(G \cdot \mathfrak{c})(G \cdot C)^{\tau}.
	\end{equation}
\end{definition}

\begin{remark}
	Writing $g \colon \mathfrak{C} \to \mathfrak{C}$
	for the $G$-action maps,
	one has the more explicit formula 
	(see Example \ref{GDOTCC EX})
	\[
	G \cdot_{\mathfrak{C}} \vect{C}
	= 
	\coprod_{g \in G}
	g \vect{C}
	\]
	However, in practice we will prefer to use 
	\eqref{GCDTCC EQ} for technical purposes.
\end{remark}

\begin{remark}
	One can further extend \eqref{GCDTCC EQ}
	to define a functor
	$G \cdot_{\mathfrak{C}} (-) \colon \Phi_{\mathfrak{C}}
	\to \Phi_{\mathfrak{C}}^G$,
	which is the left adjoint to the forgetful functor
	$ \Phi_{\mathfrak{C}}^G
	\to \Phi_{\mathfrak{C}}$.
\end{remark}

\begin{example}\label{GDOTCC EX}
	Let $G = \{1,i,-1,-i\} \simeq \mathbb{Z}_{/4}$ 
	be the group of quartic roots of unit and
	$\mathfrak{C} = \{\mathfrak{a}, -\mathfrak{a}, i\mathfrak{a}, -i\mathfrak{a}, \mathfrak{b}, i \mathfrak{b} \}$ where we implicitly have
	$-\mathfrak{b} = \mathfrak{b}$.
	The following depicts the forest (of corollas)
	$G \cdot_{\mathfrak{C}} \vect{C}$
	in $\Phi_{\mathfrak{C}}^G$
	for $\vect{C}$ in $\Sigma_{\mathfrak{C}}$ the leftmost corolla.
	
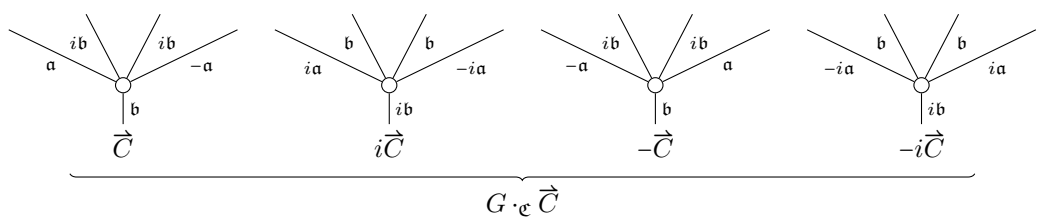
\begin{equation}
	\begin{tikzpicture}[auto,grow=up, level distance = 2.2em,
	every node/.style={font=\scriptsize,inner sep = 2pt}]%
	\tikzstyle{level 2}=[sibling distance=3em]%
	\node at (0,0) [font = \normalsize] {$\vect{C}$}%
	child{node [dummy] {}%
		child{node {}%
			edge from parent node [swap] {$-\mathfrak{a}$}}%
		child[level distance = 2.9em]{node {}%
			edge from parent node [swap,	near end] {$i\mathfrak{b}$}}%
		child[level distance = 2.9em]{node {}%
			edge from parent node [near end] {$i\mathfrak{b}$}}%
		child{node {}%
			edge from parent node  {$\mathfrak{a}$}}%
		edge from parent node [swap] {$\mathfrak{b}$}};%
	\node at (3.5,0) [font = \normalsize] {$i\vect{C}$}%
	child{node [dummy] {}%
		child{node {}%
			edge from parent node [swap] {$-i\mathfrak{a}$}}%
		child[level distance = 2.9em]{node {}%
			edge from parent node [swap,	near end] {$\mathfrak{b}$}}%
		child[level distance = 2.9em]{node {}%
			edge from parent node [near end] {$\mathfrak{b}$}}%
		child{node {}%
			edge from parent node  {$i\mathfrak{a}$}}%
		edge from parent node [swap] {$i\mathfrak{b}$}};%
	\node at (7,0) [font = \normalsize] {$-\vect{C}$}%
	child{node [dummy] {}%
		child{node {}%
			edge from parent node [swap] {$\mathfrak{a}$}}%
		child[level distance = 2.9em]{node {}%
			edge from parent node [swap,	near end] {$i\mathfrak{b}$}}%
		child[level distance = 2.9em]{node {}%
			edge from parent node [near end] {$i\mathfrak{b}$}}%
		child{node {}%
			edge from parent node  {$-\mathfrak{a}$}}%
		edge from parent node [swap] {$\mathfrak{b}$}};%
	\node at (10.5,0) [font = \normalsize] {$-i\vect{C}$}%
	child{node [dummy] {}%
		child{node {}%
			edge from parent node [swap] {$i\mathfrak{a}$}}%
		child[level distance = 2.9em]{node {}%
			edge from parent node [swap,	near end] {$\mathfrak{b}$}}%
		child[level distance = 2.9em]{node {}%
			edge from parent node [near end] {$\mathfrak{b}$}}%
		child{node {}%
			edge from parent node  {$-i\mathfrak{a}$}}%
		edge from parent node [swap] {$i\mathfrak{b}$}};%
	\draw[decorate,decoration={brace,amplitude=2.5pt}] (11.2,-0.4) -- (-0.7,-0.4) 
	node[midway,inner sep=4pt,font=\normalsize]{$G \cdot_{\mathfrak{C}} \vect{C}$}; %
	\end{tikzpicture}%
	\end{equation}%
	Note that the pairs $\vect{C},-\vect{C}$
	and $i\vect{C},-i\vect{C}$ are isomorphic in $\Sigma_{\mathfrak{C}}$
	while any other pair,
	such as $\vect{C},i\vect{C}$, is not.
	In general, it is moreover possible for two or more tree components of
	$G \cdot_{\mathfrak{C}} \vect{C}$ to be equal.
\end{example}

\begin{proposition}\label{REPALTDESC PROP}
	For any $G$-set of colors $\mathfrak C$
	and $\mathfrak{C}$-profile
	$\vect{C} \in \Sigma_{\mathfrak{C}}$
	one has a natural identification
	\[
	(G \ltimes \Sigma^{op}_{\mathfrak{C}})(\vect{C},-)
	\simeq 
	\Sigma_{\mathfrak{C}} [G \cdot_{\mathfrak{C}} \vect{C}].
	\]
\end{proposition}

\begin{proof}
	Recalling the $C_{\varphi}(-,-)$ notation (cf. Notation \ref{MAPSDEC NOT})
	for maps over $\varphi\colon b \to b'$ we likewise write 
	$\mathsf{Sym}^G_{\varphi}(-,-),\mathsf{Sym}_{\varphi}(-,-)$
	for maps over the map of colors $\varphi$.
	The result now follows from the string of isomorphisms
	(which show that $\vect{C}$ represents
	$\Sigma_{\mathfrak{C}} [G \cdot_{\mathfrak{C}} \vect{C}] \in \mathsf{Set}^{G \ltimes \Sigma^{op}_{\mathfrak{C}}}$)
	\[
	\mathsf{Sym}^G_{\mathfrak{C}}
	(\Sigma_{\mathfrak{C}} [G \cdot_{\mathfrak{C}} \vect{C}],X)
	\simeq
	\mathsf{Sym}^G_{G \cdot \mathfrak{c}}
	(\Sigma_{\tau} [G \cdot C],X)
	\simeq
	\mathsf{Sym}_{\mathfrak{c}}
	(\Sigma_{\tau} [C],X)
	\simeq
	\mathsf{Sym}_{\mathfrak{C}}
	(\Sigma_{\mathfrak{C}}[\vect{C}],X)
	=
	X(\vect{C})
	\]
	where:
        the first and third steps 
        use the canonical pushforwards \eqref{CANPUSH EQ}
        and (the dual of) Remark \ref{CARTCHAR REM},
        the second step uses \eqref{ADJOVADJ EQ}
        and the observation that
        $\Sigma_{\tau}[G \cdot C] \simeq G \cdot \Sigma_{\tau}[C]$, and
        the last step is the Yoneda lemma in $\mathsf{Sym}_{\mathfrak{C}}$.
\end{proof}

We end this section by discussing convenient notation  
for subgroups 
$\Lambda \leq \mathsf{Aut}_{G \ltimes \Sigma_{\mathfrak{C}}^{op}}(\vect{C})$,
where the $\mathfrak{C}$-profile $\vect{C}$
is regarded as an object of 
$G \ltimes \Sigma_{\mathfrak{C}}^{op}$.

\begin{remark}\label{SIGACT REM}
	Extending Remark \ref{GLOBSIG REM}, 
	the group $G \times \Sigma_n^{op}$
	acts on the set $\Sigma_{\mathfrak{C},n}$
	of $n$-ary $\mathfrak{C}$-profiles
	$\vect{C} = (\mathfrak{c}_1,\cdots,\mathfrak{c}_n;\mathfrak{c}_0)$
	via the assignment (where $g \in G$ acts on the left and $\sigma \in \Sigma_n$ on the right)
	\begin{equation}\label{SIGACT_EQ}
	g \vect{C} \sigma =
	g (\mathfrak{c}_1,\cdots,\mathfrak{c}_n;\mathfrak{c}_0) \sigma
	=
	(g\mathfrak{c}_{\sigma(1)},\cdots,g\mathfrak{c}_{\sigma(n)};g\mathfrak{c}_{\sigma(0)})
	\end{equation}
	or, more compactly, $g (\mathfrak{c}_i) \sigma = (g \mathfrak{c}_{\sigma(i)})$.
	
	Moreover, the natural functor 
	$G \ltimes \Sigma^{op}_{\mathfrak{C}}
	\to G \times \Sigma^{op}$
	which forgets colors is faithful,
	with an arrow
	$\vect{C} \to \vect{C'}$
	between arity $n$ profiles
	in $G \ltimes \Sigma_{\mathfrak{C}}^{op}$
	given by 
	$(g,\sigma) \in G \times \Sigma_n^{op}$
	such that
	$g \vect{C} \sigma = \vect{C'}$.
\end{remark}

\begin{definition}\label{STABS DEF}
	If a subgroup $\Lambda \leq G \times \Sigma_n^{op}$
	fixes a profile 
	$\vect{C} = (\mathfrak{c}_1,\cdots,\mathfrak{c}_n;\mathfrak{c}_0)$,
	i.e. if
	$g\mathfrak{c}_{\sigma(i)} = \mathfrak{c}_i$
	for all $(g, \sigma) \in \Lambda, 0 \leq i \leq n$,
	we say that \textit{$\Lambda$ stabilizes $\vect C$}. 
\end{definition}

\subsection{Equivariant colored symmetric operads}
\label{EQCOSYMOP SEC}

Our goal in this section is to describe the category 
$\mathsf{Op}^{G}_{\bullet}(\V)$
of equivariant colored operads in a way suitable for our 
proof of the main theorems in \S \ref{FIXCOL SEC}. 

We start by considering the non-equivariant case
$G = \**$,
for which we will describe
$\mathsf{Op}_{\bullet}(\V)$
as the category of fiber algebras over a certain fibered monad $\mathbb{F}$ (Def. \ref{FREEOP DEF}) on 
$\mathsf{Sym}_\bullet(\mathcal{V})$.
The restrictions $\mathbb{F}$ to each fiber
$\mathsf{Sym}_{\mathfrak{C}}(\V)$,
i.e. the fiber monads $\mathbb{F}_{\mathfrak{C}}$,
are simply the ``free operad with set of objects $\mathfrak{C}$'' monads (Rem. \ref{CONVER REM}),
which are well known to be defined using 
$\mathfrak{C}$-colored trees.

As such, cf. Definition \ref{COLFOR DEF},
we first write 
$\Omega_{\mathfrak{C}} \subset \Phi_{\mathfrak{C}}$
for the subcategory of $\mathfrak{C}$-colored forests which happen to be trees,
as well as 
$\Omega^0_{\mathfrak{C}} \subseteq \Omega_{\mathfrak{C}}$
for the wide subcategory whose arrows are the isomorphisms.
Next, following \cite[Not. 3.38]{BP21},
there is a ``colored arity functor''
\begin{equation}\label{LRDEF EQ}
\Omega_{\mathfrak C}^0 \xrightarrow{\mathsf{lr}_{\mathfrak C}} \Sigma_{\mathfrak{C}}
\end{equation}
which we call the \emph{leaf-root functor}, described as follows.
Given $\vect{T} \in \Omega^0_{\mathfrak{C}}$, 
its leaf-root
$\mathsf{lr}(\vect{T}) \in \Sigma_{\mathfrak{C}}$
is the only $\mathfrak{C}$-corolla
admitting a planar tall map
$\mathsf{lr}(\vect{T}) \to \vect{T}$
(see \cite[Defs. 3.21 and 3.35]{BP21}),
where by tall map we mean a map
which sends leaves to leaves and the root to the root.

\begin{example}
	For $\vect{T},\vect{S}$ the trees in Example \ref{COLFORES EX},
	we depict $\mathsf{lr}(\vect{T})$, $\mathsf{lr}(\vect{S})$
	below,
	which, informally, are obtained by keeping only the leaves and roots of the trees $\vect{T}, \vect{S}$.
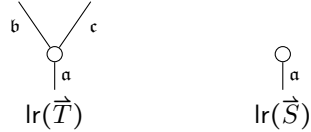
	\begin{equation}
	\begin{tikzpicture}[auto,grow=up, level distance = 2.2em,
	every node/.style={font=\scriptsize,inner sep = 2pt}]%
	\tikzstyle{level 2}=[sibling distance=3em]%
	\node at (0,0) [font = \normalsize] {$\mathsf{lr}(\vect{T})$}%
	child{node [dummy] {}%
		child{node {}%
			edge from parent node [swap, near end] {$\mathfrak{c}\phantom{\mathfrak{b}}$}}%
		child{node {}%
			edge from parent node [near end] {$\mathfrak{b}$}}%
		edge from parent node [swap] {$\mathfrak{a}$}};%
	\node at (3,0) [font = \normalsize] {$\mathsf{lr}(\vect{S})$}%
	child{node [dummy] {}%
		edge from parent node [swap] {$\mathfrak{a}$}};%
	\end{tikzpicture}%
	\end{equation}%
	Alternatively, in profile notation we have
	$\mathsf{lr}(\vect{T}) = (\mathfrak{b},\mathfrak{c};\mathfrak{a})$
	and 
	$\mathsf{lr}(\vect{S}) = (;\mathfrak{a})$.
\end{example}

\begin{remark}\label{ETACNOT REM}
	Note that for a stick tree $\eta_{\mathfrak{c}}$,
	consisting of a single edge decorated by the color 
	$\mathfrak{c} \in \mathfrak{C}$,
	it is 
	$\mathsf{lr}(\eta_{\mathfrak{c}}) = (\mathfrak{c};\mathfrak{c})$,
	which is the corolla with two edges (a leaf and a root),
	both labeled by $\mathfrak{c}$.
\end{remark}

\begin{remark}
	The colored leaf-root functor is in fact fibered over $\mathsf{Set}^G$.
	More explicitly, for any map of colors 
	$\varphi \colon \mathfrak C \to \mathfrak D$,
	we have the following commuting diagram.
	\begin{equation}
	\label{COLORLR_EQ}
	\begin{tikzcd}
	\Omega_{\mathfrak C}^0 \arrow[r, "\varphi"] \arrow[d, "\mathsf{lr}_{\mathfrak C}"']
	&
	\Omega_{\mathfrak D}^0 \arrow[d, "\mathsf{lr}_{\mathfrak D}"]
	\\
	\Sigma_{\mathfrak C} \arrow[r, "\varphi"]
	&
	\Sigma_{\mathfrak D}
	\end{tikzcd}
	\end{equation}
\end{remark}

For each $\mathfrak{C}$-profile $\vect{C}$,
we write $\vect{C} \downarrow \Omega_{\mathfrak{C}}^0$
for the undercategory with respect to $\mathsf{lr}$, 
whose objects consist of a tree $\vect{T}\in \Omega_{\mathfrak{C}}^0$
together with a choice of isomorphism 
$\vect{C} \to \mathsf{lr}(\vect{T})$.
Morally, $\vect{C} \downarrow \Omega_{\mathfrak{C}}^0$
is the ``groupoid of trees with arity $\vect{C}$. 

We can now provide the ``usual'' formula for the ``free operad monad'' 
(see \cite[page 816]{BM07} for the non-colored case).
Letting $X \in \mathsf{Sym}_{\mathfrak{C}}(\V)$,
then for each $\mathfrak{C}$-profile $\vect{C}$ we have
\begin{equation}
        \label{FROPEXP EQ}
        \mathbb{F}_{\mathfrak{C}} X (\vect{C})
        =
        \coprod_{[\vect{T}] \in 
          \mathsf{Iso}(\vect{C} \downarrow \Omega^0_{\mathfrak{C}})}
        \left(
                \left(
                        \bigotimes_{v \in \boldsymbol{V}(T)} X(\vect{T}_v)
                \right)
                \cdot_{\mathsf{Aut}_{\Omega_{\mathfrak{C}}}(\vect{T})}
                \mathsf{Aut}_{\Sigma_{\mathfrak{C}}}(\vect{C})
        \right)
\end{equation}
where $\mathsf{Iso}(-)$ denotes isomorphism classes of objects.

However, one drawback of the formula  
\eqref{FROPEXP EQ}
is that it is not immediately clear how it should be modified 
in the equivariant case,
where $\mathfrak{C}$ is a $G$-set
and $X \in \mathsf{Sym}^G_{\mathfrak{C}}$
is a functor 
$X \colon G \ltimes \Sigma_{\mathfrak{C}}^{op} \to \V$.
To address this, we first repackage \eqref{FROPEXP EQ}, 
following our approach in \cite[\S 4]{BP21}.
We first need to define another functor, which we call the \emph{vertex functor}.
As motivation, we note that
in \eqref{FROPEXP EQ}
the $\mathsf{Aut}_{\Omega_{\mathfrak{C}}}(\vect{T})$-action
on the term
$\bigotimes_{v \in \boldsymbol{V}(T)} X(\vect{T}_v)$
depends on both permutations 
of the set $\boldsymbol{V}(T)$
and on automorphisms of the corollas $\vect{T}_v$.
As such, rather than regard the vertices of $\vect{T}$ as merely a set
we define 
\begin{equation}\label{VFUNDEF EQ}
\Omega_{\mathfrak{C}}^0 \xrightarrow{\boldsymbol{V}} \Sigma \wr \Sigma_{\mathfrak{C}}
\qquad 
\vect{T} \mapsto 
\boldsymbol{V}(\vect{T})=(\vect{T}_v )_{v \in \boldsymbol{V}(T)}
\end{equation}
In words, $\boldsymbol{V}(\vect{T})$
is the tuple of corollas indexed by the vertices of $T$. 
Note that,
by regarding $\boldsymbol{V}(\vect{T})$ as an object in 
$\Sigma \wr \Sigma_{\mathfrak{C}}$ rather than just a set,
we keep track of extra automorphism data.

Noting that both the leaf root and vertex functors are naturally compatible with change of colors 
$\varphi \colon \mathfrak{C} \to \mathfrak{D}$, 
we can now provide the following alternative 
description of \eqref{FROPEXP EQ}.

\begin{definition}\label{FREEOP DEF}
	Let $\mathcal{V}$ be a closed symmetric monoidal category.
	
	The \textit{fibered free operad monad} $\mathbb{F}$ on $\mathsf{Sym}_\bullet(\mathcal{V})$ 
	assigns to 
	$\Sigma_{\mathfrak{C}}^{op} \xrightarrow{X} \mathcal{V}$
	the left Kan extension
	\begin{equation}\label{FREEOP_EQ}
                \begin{tikzcd}[column sep = 50pt]
                        \Omega^{0,op}_{\mathfrak{C}}
                        \arrow[d, "\mathsf{lr}^{op}"']
                        \arrow[r, "\boldsymbol{V}^{op}"]
                        &
                        (\Sigma \wr \Sigma_{\mathfrak{C}})^{op} \arrow[r, "(\Sigma \wr X^{op})^{op}"]
                        \arrow[dl, Rightarrow]
                        &
                        (\Sigma \wr \V^{op})^{op} \arrow[r, "\otimes"]
                        &
                        \V
                        \\
                        \Sigma^{op}_{\mathfrak{C}}
                        \arrow[urrr, "\Lan = \mathbb F_{\mathfrak{C}} X = \mathbb F X"']
                \end{tikzcd}
        \end{equation}
        Then $\Op_\bullet(\V) = \Alg^\pi_{\mathbb F}(\Sym_\bullet(\V))$,
        with fibers $\Op_{\mathfrak C}(\V) = \Alg_{\mathbb F_{\mathfrak C}}(\Sym_{\mathfrak C}(\V))$ over each $\mathfrak C \in \Set$. 
\end{definition}
The complete discussion of the monad structure
$\mathbb{F}\mathbb{F} \Rightarrow \mathbb{F}$,
$id \Rightarrow \mathbb{F}$ is postponed to Appendix \ref{MONAD_APDX},
culminating in Definition \ref{COLORMON_DEF}.

\begin{remark}\label{CONVER REM}
	To relate \eqref{FROPEXP EQ} and \eqref{FREEOP_EQ}, recall that, 
	for any span 
	$\bar{\mathcal{G}} \overset{k}{\leftarrow} \mathcal{G} \xrightarrow{X} \mathcal{V}$
	with $\mathcal{G},\bar{\mathcal{G}}$ groupoids,
	one has the formula
	\[\Lan X (\bar{G}) \simeq 
	\colim_{(k(g) \to \bar{g})\in (\mathcal{G} \downarrow \bar{g})} X(g) \simeq
	\coprod_{[k(g) \to \bar{g}] 
		\in \mathsf{Iso}(\mathcal{G} \downarrow \bar{g})}
	\mathsf{Aut}_{\bar{\mathcal{G}}}(\bar{g})
	\cdot_{\mathsf{Aut}_{\mathcal{G}}(g)}
	X(g)
	\]
	where the first identification is
	(the dual of) \cite[Thm. 1.3.5]{Ri14}
	and the second identification uses the observation that
	$\mathcal{G} \downarrow \bar{g}$ is also a groupoid.
\end{remark}

By the categorical argument in 
Proposition \ref{DIAGRAMFM_PROP},
we have that, by taking $G$-objects, 
$\mathbb{F}$ also induces a fibered monad 
$\mathbb{F}^G$ on $\mathsf{Sym}^G_{\bullet}(\V)$.

To describe $\mathbb{F}^G$,
note that \eqref{FREEOP_EQ}
can be regarded as an arrow in 
the category $\mathsf{Cat} \downarrow^l \V$
from Remark \ref{SUBCATDOWNL REM}
which, being a left Kan extension, is cocartesian over $\mathsf{Cat}$ (cf. Remark \ref{SUBCOCART REM}).
Hence, if $X \in \mathsf{Sym}_{\bullet}^G(\V)$
is $G$-equivariant, 
\eqref{FREEOP_EQ} is then a cocartesian arrow in 
$\left(\mathsf{Cat} \downarrow^l \V\right)^G$
over $\mathsf{Cat}^G$.
By Proposition \ref{EQUIVFNCON PROP},
we can hence rewrite 
such a $G$-equivariant \eqref{FREEOP_EQ}
as the left Kan extension for a span
$G \ltimes \Sigma^{op}_{\mathfrak{C}}
\leftarrow 
G \ltimes \Omega^{0,op}_{\mathfrak{C}}
\to 
\V$.
To fully describe this span,
we need to understand how equivariance 
affects the top composite in \eqref{FREEOP_EQ},
with the non-obvious issue being that of understanding
what happens to the middle map
therein, which can be described using
Remark \ref{RHOPURP REM}.
Putting all of this together
(and using the isomorphisms
$G \ltimes \mathcal{C}^{op} \simeq (G \ltimes \mathcal{C})^{op}$
from Remark \ref{INVLTIMES REM}),
we get the following.

\begin{proposition}\label{FGC PROP}
	The monad $\mathbb{F}^G$ on $\mathsf{Sym}^G_{\bullet}(\V)$
	assigns to 
	$G \ltimes \Sigma^{op}_{\mathfrak{C}} \xrightarrow{X} \V$
	the left Kan extension below.
	\begin{equation}\label{FGC_EQ}
                \begin{tikzcd}[column sep = 28pt]
                        G \ltimes \Omega^{0,op}_{\mathfrak{C}}
                        \arrow[d, "G \ltimes \mathsf{lr}_{\mathfrak C}^{op}"']
                        \arrow[r, "G \ltimes \boldsymbol{V}^{op}"]
                        &[6pt]
                        G \ltimes (\Sigma \wr \Sigma_{\mathfrak{C}})^{op} \arrow{r}
                        \arrow[dl, Rightarrow, shorten >=0.15cm,shorten <=0.15cm]
                        &
                        \left( \Sigma \wr \left( G^{op} \ltimes \Sigma_{\mathfrak{C}} \right) \right)^{op}
                        \arrow[r, "(\Sigma \wr X^{op})^{op}"]
                        &[10pt]
                        \left(\Sigma \wr \V^{op}\right)^{op} \arrow[r, "\otimes"]
                        &
                        \V
                        \\
                        G \ltimes \Sigma^{op}_{\mathfrak{C}}
                        \arrow[urrrr, "\Lan = \mathbb F_{\mathfrak{C}}^G X = \mathbb F^G X"', end anchor = south west]
                \end{tikzcd}
        \end{equation}
        Then $\Op_\bullet^G(\V) = \Alg^\pi_{\mathbb F^G}(\Sym_\bullet^G(\V))$,
        with fibers $\Op_{\mathfrak C}^G(\V) = \Alg_{\mathbb F_{\mathfrak C}^G}(\Sym_{\mathfrak C}^G(\V))$ over each $\mathfrak C \in \Set^G$. 
\end{proposition}

\begin{remark}\label{FROPEXPG REM}
	By Remark \ref{CONVER REM} we now have the following analogue of
	\eqref{FROPEXP EQ}.
	\begin{equation}\label{FROPEXPG EQ}
	\mathbb{F}^G_{\mathfrak{C}} X (\vect{C})
	=
	\coprod_{[\vect{T}] \in 
		\mathsf{Iso}(\vect{C} \downarrow G^{op} \ltimes \Omega^0_{\mathfrak{C}})}
	\left(
	\left(
	\bigotimes_{v \in \boldsymbol{V}(T)} X(\vect{T}_v)
	\right)
	\cdot_{\mathsf{Aut}_{G^{op} \ltimes \Omega_{\mathfrak{C}}}(\vect{T})}
	\mathsf{Aut}_{G^{op} \ltimes \Sigma_{\mathfrak{C}}}(\vect{C})
	\right)
	\end{equation}
	In comparing \eqref{FROPEXPG EQ} with \eqref{FROPEXP EQ},
	note that,
	since $G^{op} \ltimes \Omega^0_{\mathfrak{C}}$
	has more morphisms than
	$\Omega^0_{\mathfrak{C}}$,
	equation \eqref{FROPEXPG EQ} has fewer coproduct summands than \eqref{FROPEXP EQ},
	though this is compensated by the fact that the inductions
	$(-)\cdot_{\mathsf{Aut}_{G^{op} \ltimes \Omega_{\mathfrak{C}}}(\vect{T})}
	\mathsf{Aut}_{G^{op} \ltimes \Sigma_{\mathfrak{C}}}(\vect{C})$
	are correspondingly larger than the inductions
	$(-)\cdot_{\mathsf{Aut}_{\Omega_{\mathfrak{C}}}(\vect{T})}
	\mathsf{Aut}_{\Sigma_{\mathfrak{C}}}(\vect{C})$.
\end{remark}

\begin{remark}\label{OP_MAP REM}
	Following Remark \ref{ALGPUSHLL REM},
	for any map of $G$-sets 
	$\varphi \colon \mathfrak C \to \mathfrak D$
	and $\mathfrak D$-symmetric sequence $X$,
	one has a pullback $\mathfrak C$-symmetric sequence $\varphi^{\**}X$
	given by
	$\varphi^{\**}X(\vect D) = X(\varphi(\vect D))$,
	which is an operad if $X$ itself is an operad.
	Moreover, one then has a pair of adjunctions 
	\begin{equation}\label{GC_CHANGE_EQ}
	\begin{tikzcd}
	\Op^G_{\mathfrak C}(\V) 
	\arrow[shift left]{r}{\check{\varphi}_!}
	\arrow[d, "\mathsf{fgt}"']
	&
	\Op^G_{\mathfrak D}(\V) 
	\arrow[shift left]{l}{\varphi^{\**}}
	\arrow[d, "\mathsf{fgt}"]
	\\
	\Sym^G_{\mathfrak C}(\V) 
	\arrow[shift left]{r}{\varphi_!}
	&
	\Sym^G_{\mathfrak D}(\V) 
	\arrow[shift left]{l}{\varphi^{\**}}
	\end{tikzcd}
	\end{equation}
	where we highlight that
	the right adjoints are compatible with the forgetful functors, in the sense that 
	$\varphi^{\**} \circ \mathsf{fgt} = 
	\mathsf{fgt} \circ \varphi^{\**}$, 
	but the left adjoints are not:
	$\varphi_!$ is simply a left Kan extension, while $\check{\varphi}_!$ is given by the coequalizer
	\begin{equation}\label{CFS_EQ}
	\check{\varphi}_! \O \simeq \mathop{coeq}(\mathbb F_{\mathfrak D} \varphi_! \mathbb F_{\mathfrak C}\O \rightrightarrows \mathbb F_{\mathfrak D} \varphi_! \O).
	\end{equation}
	In general, we do not have a more explicit description of $\check{\varphi}_!$.
	However, when $\varphi$ is injective, 
	$\varphi_!X$ is the extension by $\emptyset$,
	from which it follows that 
	$\mathbb F_{\mathfrak D} \varphi_! = \varphi_! \mathbb F_{\mathfrak C}$,
	and \eqref{CFS_EQ}
	then says that
	$\check{\varphi}_! \O
	\simeq
	coeq \left( \varphi_! \mathbb{F}_{\mathfrak{C}}  \mathbb{F}_{\mathfrak{C}} \O
	\rightrightarrows 
	\varphi_!  \mathbb{F}_{\mathfrak{C}} \O
	\right)
	\simeq 
	\varphi_! \left( coeq \left( 
	\mathbb{F}_{\mathfrak{C}}  \mathbb{F}_{\mathfrak{C}} \O
	\rightrightarrows 
	\mathbb{F}_{\mathfrak{C}} \O
	\right) \right)
	\simeq 
	\varphi_! \O$,
	so that  	
	$\varphi_! \circ \mathsf{fgt} \simeq 
	\mathsf{fgt} \circ \check{\varphi}_!$.
\end{remark}

\section{Equivariant homotopy theory}
\label{EHT_SEC}

This section develops the 
equivariant homotopy theory 
needed for our main proofs.
\S \ref{GMA_SEC} introduces the \emph{global monoid axiom}
in Definition \ref{GLOBMONAX_DEF}, 
with sufficient conditions for this axiom given in 
Proposition \ref{WEAKCELL PROP}.
In \S \ref{FGPP_SEC},
motivated by the identification
$\mathsf{Sym}^{G}_{\mathfrak{C}}(\V)
\simeq \V^{G \ltimes \Sigma^{op}_{\mathfrak{C}}}$
in Proposition \ref{EQUIVFNCON PROP}, 
we extend our discussion 
from \cite[\S 6]{BP21}
regarding model structures determined by families
from the context of $\V^G$ with $G$ a group to that of
$\V^{\G}$ with $\G$ a groupoid.
This culminates in Proposition \ref{SIGMAWRGF PROP}, an extension of \cite[Prop. 6.25]{BP21}
which will greatly simplify the proof of Theorem \ref{THMII},
cf. Remark \ref{STRTED REM}.

\subsection{Global monoid axiom}
\label{GMA_SEC}

\begin{definition}\label{GENMOD DEF}
	Let $\V$ be a model category and $G$ a group.
	
	The \emph{genuine (or fine) model structure} on $G$-objects $\V^G$
	is the model structure (if it exists)
	such that
	$f\colon X \to Y$
	is a weak equivalence (resp. fibration)
	if the fixed point maps
	$f^H\colon X^H \to Y^H$
	are weak equivalences (fibrations) in $\V$
	for all $H \leq G$.
\end{definition}

\begin{notation}
	We write $\mathcal{W}_G$
	to denote the class of genuine weak equivalences in $\V^G$.
\end{notation}

\begin{remark}
	If $\V$ is cofibrantly generated with 
	$\mathcal{I}$ (resp. $\mathcal{J}$)
	the sets of generating (trivial) cofibrations,
	then the genuine model structure on $\V^G$,
	should it exist,
	is again cofibrantly generated with generating sets
	\begin{equation}\label{GENGENSETEQ}
	\mathcal{I}_G = \{G/H \cdot i \ | \ H\leq G,i\in \mathcal{I}\}
	\qquad
	\mathcal{J}_G = \{G/H \cdot j \ | \ H\leq G,j\in \mathcal{J}\}.
	\end{equation}
\end{remark}

\begin{notation}
	\label{CELL_NOT}
	Let $I$ be a class of maps in a model category $\V$.
	Following \cite{Hov99}, we write 
	$I$-cell (resp. $I$-cof)
	to denote the closure of $I$ under pushouts and transfinite composition
	(and, in addition, retracts).
\end{notation}

We can now introduce the global monoid axiom.

\begin{definition}\label{GLOBMONAX_DEF}
	Let $(\V,\otimes)$ 
	be a cofibrantly generated monoidal model category.
	
	For a finite group $G$, let $\mathcal{J}^{\otimes}_G$ be the set of
	maps in $\V^G$ given by 
	\[
	\mathcal{J}^{\otimes}_G
	=
	\mathcal J \otimes \V^G
	=
	\sets{j \otimes X}{j \in \mathcal{J},X \in \V^G}
	\]
	We refer to the class of maps  
	$\mathcal{J}^{\otimes}_G$-cof in $\V^G$
	as the \emph{$G$-genuine $\otimes$-trivial cofibrations}.
	
	Moreover, we say $\V$ satisfies the \textit{global monoid axiom} if
	$G$-genuine $\otimes$-trivial cofibrations are $G$-genuine weak equivalences,
	i.e. if $\mathcal J^{\otimes}_G \text{-cof} \subseteq \mathcal W_G$,
	for all finite groups $G$.
\end{definition}

\begin{remark}
	The global monoid axiom holds provided
	$\mathcal J^{\otimes}_G \text{-cell} \subseteq \mathcal W_G$
	for all finite groups $G$.
\end{remark}

\begin{remark}\label{MONAX_REM}
	Restricting to $G = \**$, the global monoid axiom  
	recovers the monoid axiom of Schwede-Shipley \cite{SS00}.
\end{remark}


We now discuss convenient sufficient conditions 
for the existence of the genuine model structures in 
Definition \ref{GENMOD DEF}
and for the global monoid axiom in 
Definition \ref{GLOBMONAX_DEF}.
These are given by the following definition,
which is motivated by \cite[Remark 2.7]{Ste16},
and gives two variants of the 
\emph{cellular fixed points} conditions of
\cite[Prop. 2.6]{Ste16}.

\begin{definition}\label{WEAKCELL DEF}
	Let $\V$ (resp. $(\V,\otimes)$)
	be a cofibrantly generated (monoidal) model category.
	
	We say $\V$ has \textit{weak acyclic cellular fixed points}
	(resp. \textit{$(\V,\otimes)$ has monoidal weak acyclic cellular fixed points}) if,
	for all finite groups $G$ and subgroups $H \leq G$,
	the fixed point functor $(-)^H \colon \V^G \to \V$
	\begin{enumerate}[label = (\roman*)]
		\item preserves direct colimits of maps in $\mathcal J_G$-cof (resp. $\mathcal {J}^{\otimes}_G$-cof);
		\item preserves pushout diagrams 
		where one leg is in $\mathcal J_G$ 
		(resp. in $\mathcal {J}^{\otimes}_G$);
		\item sends maps in $\mathcal J_G$ to maps in $\mathcal J$-cof (resp. maps in $\mathcal {J}^{\otimes}_G$ to maps in 
		$\mathcal{J}^{\otimes}$-cof).
	\end{enumerate}
\end{definition}

\begin{proposition}\label{WEAKCELL PROP}
	Let $\V$ (resp. $(\V,\otimes)$) be a cofibrantly generated (monoidal) model category.
	\begin{enumerate}[label=(\roman*)]
		\item If $\V$ has weak acyclic cellular fixed points, 
		the genuine model structure on $\V^G$
		exists for any finite group $G$.
		
		Moreover, for any $H \leq G$,
		fixed points $(-)^H \colon \V^G \to \V$
		send genuine trivial cofibrations to trivial cofibrations. 
		\item If $(\V,\otimes)$ has monoidal weak acyclic cellular fixed points and satisfies the usual monoid axiom
		(cf. Remark \ref{MONAX_REM}),
		then $\V$ satisfies the global monoid axiom.
	\end{enumerate}
\end{proposition}

\begin{proof}
	For (i), we apply \cite[Theorem 2.1.19]{Hov99} to the generating sets in 
	\eqref{GENGENSETEQ}.
	All conditions therein are immediate 
	except for the requirement that
	$\mathcal{J}_{G}\text{-cell} \subseteq \mathcal{W}_G$,
	which holds since the conditions in Definition \ref{WEAKCELL DEF}
	guarantee that $H$-fixed points of maps in 
	$\mathcal{J}_{G}$-cell are in $\mathcal{J}$-cell.
	
	(ii) is similar. Definition \ref{WEAKCELL DEF}
	implies that $H$-fixed points of maps in 
	$\mathcal{J}^{\otimes}_{G}$-cell are in $\mathcal{J}^{\otimes}$-cell.
\end{proof}

\subsection{Families in groupoids and pushout powers}
\label{FGPP_SEC}

In addition to the genuine model structures
from Definition \ref{GENMOD DEF} in \S \ref{GMA_SEC},
we will need to consider variants
where the subgroups $H\leq G$ therein are restricted to a family of subgroups.
In fact, motivated by the identification
$\mathsf{Sym}^G_{\mathfrak{C}}(\V) \simeq 
\mathsf{Fun}(G \ltimes \Sigma^{op}_{\mathfrak{C}},\V)$
in Proposition \ref{EQUIVFNCON PROP},
we will further consider a groupoid variant.
We first extend the notion of family of subgroups to the context of groupoids.

\begin{definition}\label{FAMGROUPOID DEF}
	Let $\G$ be a groupoid.
	A \textit{family of subgroups} of $\G$
	is a collection 
	$\mathcal{F} = \left\{\mathcal{F}_x\ | \ x\in \G\right\}$
	where each $\F_x$ is itself a collection of subgroups
	$H \leq \mathsf{Aut}_{\mathcal{G}}(x)$ and such that:
	\begin{enumerate}[label = (\roman*)]
		\item if $H \in \F_x$ and $K \leq H$ then $K \in \mathcal{F}_x$;
		\item if $H \in \mathcal{F}_x$
		then for any arrow $g \colon x \to x'$
		it is $g H g^{-1} \in \mathcal{F}_{x'}$.
	\end{enumerate}
\end{definition}

\begin{remark}
	If $\G$ has a single object, i.e. if $\G$ is a group $G$ regarded as a category with a single object, Definition \ref{FAMGROUPOID DEF} recovers the usual definition of a family of subgroups of $G$ as a collection of subgroups $H\leq G$ closed under inclusion and conjugation.
	
	Moreover, if $\F$ is a family of subgroups of $\G$, each $\F_x$ is a family of subgroups of $\mathsf{Aut}(x)$ in the usual sense. 
	In fact, $\mathcal{F}$ is completely determined by a choice of families
	$\F_x$ for $x$ ranging over a set of representatives of the isomorphism classes/components of $\G$.
\end{remark}

\begin{definition}
        \label{FCMODEL_DEF}
	Let $\V$ be a model category, $\G$ a groupoid and 
	$\F$ a family of subgroups of $\G$.
	
	The \textit{$\F$-model structure} on $\V^\G = \Fun(\G, \V)$, 
	which we denote by $\V^\G_\F$,
	is the model structure (if it exists)
	such that 
	$f\colon X \to Y$
	is a weak equivalence (resp. fibration)
	if the maps $f(x)^H \colon X(x)^H \to Y(x)^H$ are weak equivalences (fibrations) in $\V$ for all $x \in \G$ and $H \in \F_x$.
\end{definition}

\begin{remark}\label{VGFGEN REM}
	Generalizing \eqref{GENGENSETEQ}, one has that,
	if $\V$ is cofibrantly generated and 
	the model structure $\V^{\G}_{\F}$ exists,
	then the latter is likewise cofibrantly generated with generating sets
	\begin{equation}\label{VGFGEN EQ}
	\mathcal I_{\F} := \left\{
	\G(x,-)/H \cdot i
	\ | \ x \in \G, H \in \F_x, i\in \mathcal{I}
	\right\}
	\qquad
	\mathcal J_{\F} := \left\{
	\G(x,-)/H \cdot j
	\ | \ x \in \G, H \in \F_x, j\in \mathcal{J}
	\right\}.
	\end{equation}
\end{remark}

\begin{remark}\label{SIGMACOF_REM}
	Since, for any groupoid, one has an equivalence of categories
	$\mathcal{G} \simeq 
	\coprod_{[x] \in \mathsf{Iso}(\mathcal{G})}
	\mathsf{Aut}(x)$,
	(trivial) cofibrations in $\V^{\G}_{\F}$
	also admit a pointwise description: 
	a map $f$ in $\V^\G_\F$ is a (trivial) cofibration iff $f(x)$ is a (trivial) cofibration in $\V^{\Aut(x)}_{\F_x}$ for all $x \in \G$.
\end{remark}

\begin{proposition}\label{ALLEQ PROP}
	Suppose $\V$ is cofibrantly generated.
	Then the $\V^{\G}_{\F}$
	model structures exist for all 
	groupoids $\G$ and families $\F$
	if and only if
	the genuine model structures on $\V^G$
	exist for all groups $G$.
	
	In particular, the $\V^{\G}_{\F}$
	model structures exist whenever $\V$ has weak acyclic cellular fixed points.
\end{proposition}

\begin{proof}
	For the main claim, only the ``if'' direction requires proof.
	Much as in Proposition \ref{WEAKCELL PROP}(i), 
	we consider \cite[Theorem 2.1.19]{Hov99}
	applied to the generating sets in \eqref{VGFGEN EQ},
	and again the only non immediate condition is the requirement
	$\mathcal{J}_{\F}\text{-cell} \subseteq \mathcal{W}_{\F}$,
	where $\mathcal{W}_{\F}$ denotes the weak equivalences 
	in $\V^{\G}_{\F}$.
	By Remark \ref{SIGMACOF_REM}, it suffices to show 
	$\mathcal{J}_{\F_x}\text{-cell} \subseteq \mathcal{W}_{\F_x}$
	for all $x \in \G$.
	This now follows since
	\begin{equation}\label{FAMTOGEN EQ}
	\mathcal{J}_{\F_x}\text{-cell} \subseteq 
	\mathcal{J}_{\Aut(x)}\text{-cell} \subseteq 
	\mathcal{W}_{\Aut(x)} \subseteq 
	\mathcal{W}_{\F_x}
	\end{equation}
	where the middle inclusion follows by the existence of the 
	genuine model structure on $\V^{\Aut(x)}$.
	
	The ``in particular'' claim now reduces to Proposition \ref{WEAKCELL PROP}(i).
\end{proof}

\begin{remark}\label{PULLFAM REM}
	The families of subgroups $\F$ of a groupoid $\G$ 
	form a poset (in fact lattice) under inclusion.
	Given a map of groupoids $\phi\colon \G \to \bar{\G}$
	and family $\bar{\F}$ of subgroups of $\bar{\G}$
	we have a ``pullback family'' 
	$\phi^{\**} \bar{\F}$ 
	of subgroups of $\G$ given by
	\begin{equation}
	\label{PHISTAR_EQ}
	\left(\phi^{\**} \bar{\F} \right)_{x}
	=
	\left\{H \leq \Aut(x) \ | \ \phi(H) \in \bar{\F}_{\phi(x)}\right\}.
	\end{equation}
\end{remark}

The purpose of the $\phi^{\**} \mathcal{F}$ families is given by the following,
which adapts \cite[Props. 6.5, 6.6]{BP21}
in light of Remark \ref{SIGMACOF_REM}.

\begin{proposition}\label{EQQUILADJ PROP}
	Suppose all the model structures appearing in \eqref{FAMADJ EQ} exist.
	
	Let $\phi \colon \G \to \bar{\G}$
	be a map of groupoids and
	$\F,\bar{\F}$ families of subgroups of $\G,\bar{\G}$.
	Then the adjunctions
	\begin{equation}\label{FAMADJ EQ}
	\phi_! \colon \V^{\G}_{\F}
	\rightleftarrows
	\V^{\bar{\G}}_{\bar{\F}} \colon \phi^{\**}
	\qquad
	\phi^{\**} \colon \V^{\bar{\G}}_{\bar{\F}} 
	\rightleftarrows
	\V^{\G}_{\F} \colon \phi_{\**}
	\end{equation}
	are Quillen provided
	$\F \subseteq \phi^{\**} \bar{\F}$
	for the left adjunction and 
	$\phi^{\**} \bar{\F} \subseteq \F$
	for the right adjunction.
\end{proposition}

We next discuss the interactions of equivariant model structures on 
$\mathcal{V}$ with the monoidal structure $\otimes$ on $\mathcal{V}$.

\begin{definition}\label{GGENOTITC DEF}
	Extending Definition \ref{GLOBMONAX_DEF},
	we write
	$
	\mathcal{J}^{\otimes}_{\G}
	=
	\mathcal J \otimes \V^{\G}
	=
	\sets{j \otimes X}{j \in \mathcal{J},X \in \V^{\G}}
	$
	and refer to the class of maps  
	$\mathcal{J}^{\otimes}_{\G}$-cof in $\V^{\G}$
	as the \emph{$\G$-genuine $\otimes$-trivial cofibrations}.
	
	Note that, as in Remark \ref{SIGMACOF_REM},
	$f$ is a $\G$-genuine $\otimes$-trivial cofibration if and only if $f(x)$ is a $\Aut(x)$-genuine $\otimes$-trivial cofibration for all $x \in \G$.
\end{definition}

\begin{proposition}\label{REGEOTCOF PROP}
	Both left adjoints in \eqref{FAMADJ EQ}
	send genuine $\otimes$-trivial cofibrations 
	to genuine $\otimes$-trivial cofibrations.
\end{proposition}

\begin{proof}
	This follows from the identifications
	$\phi_!\left(j \otimes X\right) \simeq j \otimes\phi_!\left( X\right)$
	and
	$\phi^{\**}\left(j \otimes X\right) \simeq j \otimes \phi^{\**}\left( X\right)$.
\end{proof}

The following is immediate from
\cite[Rem. 6.14]{BP21} and Remark \ref{SIGMACOF_REM}.

\begin{proposition}\label{RESGEN PROP}
	Suppose $(\V, \otimes)$ is a closed monoidal model category which is cofibrantly generated,
	and also that the model structures appearing below exist.
	Further, let $\G, \bar{\G}$ be groupoids and $\F,\bar{\F}$
	families of subgroups of $\G, \bar{\G}$.
	Then $\otimes$ includes a left Quillen bifunctor
	\[
	\V^{\G}_{\F} \times \V^{\bar \G}_{\bar{\F}} \xrightarrow{\otimes} \V^{\G \times \bar \G}_{\F \sqcap \bar{\F}}
	\]
	with the family $\F \sqcap \bar{\F}$ of $\G \times \bar{\G}$ is defined as follows 
	(for
	$\pi_\G \colon \G \times \bar{\G} \to \G$,
	$\pi_{\bar{\G}} \colon \G \times \bar{\G} \to \bar{\G}$
	the projections)
	\begin{equation}
	\label{FCAPBARF_EQ}
	\left(\F \sqcap \bar{\F}\right)_{(x,\bar{x})}
	=
	\pi_\G^{\**}(\F_x) \cap \pi_{\bar{\G}}^{\**}(\bar \F_{\bar x})
	=
	\left\{K\leq \Aut(x,\bar{x})\ |\ \pi_{\G} (K) \in \F_x,
	\pi_{\bar{\G}} (K) \in \F_{\bar{x}}
	\right\}.
	\end{equation}
\end{proposition}


Unpacking the definition of left Quillen bifunctor,
Proposition \ref{RESGEN PROP}
says that,
if maps $f, \bar{f}$
in $\V^{\G}$, $\V^{\bar{\G}}$
are $\F,\bar{\F}$ cofibrations,
then the \textit{pushout product} map
$f\square \bar{f}$ 
in $\mathcal{V}^{\G \times \bar{\G}}$ (defined in e.g. \cite[11.1.7]{Ri14})
is a 
$\F \sqcap \bar{\F}$-cofibration,
which is trivial if either $f$ or $\bar{f}$ are.

Note, however, that should $\otimes$ be a symmetric monoidal structure, then when $\G = \bar \G$ and $f = \bar f$
the map $f \square f$
admits an additional $\Sigma_2$-action
(and, more generally, $f^{\square n}$ admits a $\Sigma_n$-action)
which is ignored by Proposition \ref{RESGEN PROP}.
To discuss such ``actions on powers'' we need a few preliminaries, 
starting with the following additional hypothesis on $\V$.

\begin{definition}[{\cite[Def. 6.16]{BP21}}]\label{CSPP_DEF}
	We say a symmetric monoidal model category $\V$ has \textit{cofibrant symmetric pushout powers} if,
	for all (trivial) cofibrations $f$, 
	the pushout product power $f^{\square n}$
	is a $\Sigma_n$-genuine (trivial) cofibration in $\V^{\Sigma_n}$ for all $n \geq 1$. 
\end{definition}

\begin{remark}
	We purposely excluded the $n=0$ case in Definition \ref{CSPP_DEF}.
	
	Unwinding definitions, one sees that for any map $f$,
	$f^{\square 0}$
	is always the map 
	$\emptyset \to 1_{\V}$
        from the initial object to the monoidal unit
	(indeed, $\square$ determines a monoidal structure on the category of arrows, with $\emptyset \to 1$ being the unit).
	However, while we assume that 
	$1_{\V}$ is cofibrant at certain points in this work, 
	we never wish to assume that
	$\emptyset \to 1_{\V}$
	is a trivial cofibration.
\end{remark}

Given a map $f$ in $\mathcal{V}^{\G}$
we write 
$f^{\square n}$  and 
$f^{\otimes n}$
for the maps in $\mathcal{V}^{\Sigma_n \wr \G}$
given by
\begin{equation}\label{FSQNDEF EQ}
f^{\square n}\left((x_i)_{1\leq i \leq n }\right)
=
\underset{1\leq i \leq n}{\mathlarger{\mathlarger{\mathlarger{\square}}}} f(x_i),
\qquad
f^{\otimes n}\left((x_i)_{1\leq i \leq n }\right)
=
\underset{1\leq i \leq n}{\bigotimes} f(x_i).
\end{equation}
Next, we write
\[
\pi_{\Sigma} \colon \Sigma_n \ltimes \G^{\times n} \to \Sigma_n
\qquad
\pi^i_{\G} \colon \Sigma_n \ltimes \G^{\times n} \to \G, 1\leq i \leq n
\]
for the projections onto each coordinate.
We warn that, while $\pi_{\Sigma}$ is a map of groupoids, 
the $\pi^i_{\G}$ are not.
Nonetheless, writing $\Sigma^i_n \leq \Sigma_n$
for the subgroup of permutations that fix $i$, 
one has that $\pi_{G}^i$ is a homomorphism when restricted to 
$\pi^{-1}_{\Sigma}(\Sigma_n^i)$ 
(indeed, one has an isomorphism 
$\pi^{-1}_{\Sigma}(\Sigma_n^i) \simeq 
(\Sigma_{1} \wr \G) \times (\Sigma_{n-1} \wr \G)$,
which identifies $\pi_{\G}^i$ with the projection to 
$(\Sigma_{1} \wr \G) \simeq \G$).
Given $(x_i) \in \Sigma_n \wr \G$
and a subgroup
$H \leq \Aut((x_i))$,
we then write
$H_i = H \cap \pi^{-1}_{\Sigma}(\Sigma_n^i)$
for the subgroup of $H$ whose projection to $\Sigma$ fixes $i$.

Given a family $\F$ of subgroups of $\G$,
we can now finally define the family $\F^{\ltimes n}$
of subgroups of $\Sigma_n \wr \G$ for $n\geq 1$ by
\begin{equation}\label{FWRNXI EQ}
\left(\F^{\ltimes n}\right)_{(x_i)}
=
\left\{
H \leq \Aut((x_i))
\ | \
\pi^i_{\G}(H_i) \in \F_{x_i} \text{ for } 1 \leq i \leq n
\right\}.
\end{equation}
Note that this construction is compatible with pullbacks
along a functor $\phi \colon \bar{\G} \to \G$, in the sense that 
\begin{equation}
\label{PHIFLTIMES_EQ}
(\Sigma_n \wr \phi)^{\**}\F^{\ltimes n} = (\phi^{\**} \F)^{\ltimes n}.
\end{equation}

\begin{lemma}[cf. {\cite[Prop. 6.23]{BP21}}]
	\label{LTIMESNMINC LEM}
	Let $\F$ be a family of subgroups in the groupoid $\G$ and write
	$\iota \colon \left(\Sigma_n \wr \G\right) \times \left(\Sigma_m \wr \G\right) \to \Sigma_{n+m} \wr \G$ for the standard inclusion.
	Then 
	\begin{equation}\label{LTIMESNMINC EQ}
	\F^{\ltimes n} \sqcap \F^{\ltimes m} \subseteq 
	\iota^{\**}
	\F^{\ltimes n+m}
	\end{equation}
	so that the composite
	\begin{equation}\label{LTIMESNMINC2 EQ}
	\V^{\Sigma_n \wr \G}_{\F^{\ltimes n}} \times \V^{\Sigma_m \wr \G}_{\F^{\ltimes m}}
	\xrightarrow{\otimes}
	\V^{\left(\Sigma_n \wr \G\right) \times \left(\Sigma_m \wr \G\right)}_{\F^{\ltimes n} \sqcap \F^{\ltimes m}}
	\xrightarrow{\iota_!}
	\V^{\Sigma_{n+m} \wr \G}_{\F^{\ltimes n+m}}
	\end{equation}
	is a left Quillen bifunctor.
\end{lemma}

\begin{proof}
	To ease notation, we write
	$\{1,\cdots, n+m\} = \{1,\cdots,n\} \amalg \{1,\cdots,m\}$,
	and allow $1\leq i\leq n$ to range over the first summand and $1 \leq j \leq m$ to range over the second summand.
	
	Note now that, if $H$ is a group of automorphisms in 
	$(\Sigma_n \wr \G) \times (\Sigma_m \wr \G)$,
	then 
	$\pi^i_{\G}(H) = \pi^i_{\G}(\pi_{\Sigma_n \wr \G}(H))$
	and
	$\pi^j_{\G}(H) = \pi^j_{\G}(\pi_{\Sigma_m \wr \G}(H))$,
	so that 
	\eqref{LTIMESNMINC EQ} is immediate from definition of 
	$\F \sqcap \bar{\F}$.
	
	That \eqref{LTIMESNMINC2 EQ} is a Quillen bifunctor simply combines 
	Propositions \ref{EQQUILADJ PROP} and \ref{RESGEN PROP}.
\end{proof}

We now have the following, which is a strengthening of 
\cite[Prop. 6.25]{BP21}.

\begin{proposition}\label{SIGMAWRGF PROP}
	Suppose $(\V, \otimes)$ is as in Proposition \ref{RESGEN PROP} and,
	in addition, that it has cofibrant symmetric pushout powers.
	Further, let $\G$ be a groupoid and
	$\F$ a family of subgroups of $\G$.
	Then:
	\begin{enumerate}[label=(\roman*)]
		\item if $f$ is a (trivial) $\F$-cofibration in $\V^{\G}$
		then $f^{\square n}$ is a (trivial)
		$\F^{\ltimes n}$-cofibration in $\V^{\Sigma_n \wr \G}$;
		\item if $f$ is a (trivial) $\F$-cofibration between $\F$-cofibrant objects in $\V^{\G}$
		then $f^{\otimes n}$ is a (trivial)
		$\F^{\ltimes n}$-cofibration
		between $\F^{\ltimes n}$-cofibrant objects in $\V^{\Sigma_n \wr \G}$.
	\end{enumerate}
\end{proposition}

\begin{proof}
	We first prove the result when $\G = G$ is in fact a group.
	In this case,
	(i) is almost exactly \cite[Prop. 6.25]{BP21}, 
	except for the fact that throughout \cite{BP21}
	we assume $\V$ has cellular fixed points
	(i.e. that it satisfies the conditions in 
	\cite[Prop. 2.6]{Ste16}, which are a stronger version of Definition \ref{WEAKCELL DEF}), 
	so we must check that this assumption is not needed in the proof of 
	\cite[Prop. 6.25]{BP21}.
	Analyzing the proof therein, one sees that the only model structure requirements are the cofibrancy of pushout powers
	and that the functors in 
	\cite[Props 6.5 and 6.23]{BP21}
	are left Quillen (bi)functors.
	Noting that the latter results are generalized by Proposition \ref{EQQUILADJ PROP} and Lemma \ref{LTIMESNMINC LEM} yields (i).

	For (ii) when $\G=G$ is a group, we need to recall an argument 
	in the proof of \cite[Prop. 6.25]{BP21}. 
	Given composable arrows $Z_0 \xrightarrow{g} Z_1 \xrightarrow{f} Z_2$ in $\V$,
	denote by $Q^n(g), Q^n(f)$ the domains of $g^{\square n}, f^{\square n}$.
	There is a filtration of the box product of the composite
	$(fg)^{\square n}$ as
	\[
	Q^n(fg)
	\xrightarrow{k_0}
	\bullet
	\xrightarrow{k_1}
	\cdots
	\xrightarrow{k_{n-1}}
	\bullet
	\xrightarrow{k_{n}}
	Z_2^{\otimes n}
	\] 
	where each $k_r$, $0\leq r \leq n$ is given by a pushout as on the left below, with the right diagram specifying the $k_0$ case
	(this filtration is induced by a $\Sigma_n$-equivariant filtration 
	$P_0 \subseteq P_1 \subseteq \dots \subseteq P_n$ of the poset $P_n = (0 \to 1 \to 2)^{\times n}$,
	where $P_0$ consists of the tuples with at least one $0$-coordinate 
	and tuples in $P_r$ have at most $r$ $2$-coordinates;
	for more details, see the proof of \cite[Lemma 4.8]{Pe16})
	\begin{equation}\label{HGBOX_EQ}
	\begin{tikzcd}
	\bullet 
	\arrow[d, "\Sigma_n \cdot_{\Sigma_{n-r} \times \Sigma_r} (g^{\square n-r} \square f^{\square r})"'] \arrow[r]
	&
	\bullet \arrow[d, "k_r"]
	& &
	Q^n(g) \arrow[r] \arrow[d, "g^{\square n}"']
	&
	Q^n(fg) \arrow[d, "k_0"]
	\\
	\bullet \arrow[r]
	&
	\bullet
	& &
	Z_1^{\otimes n} \arrow[r]
	&
	\bullet
	\end{tikzcd}
	\end{equation}
	Specifying to the case $Z_0 = \emptyset$, one has
	$Q^n(g)= Q^n(fg) = \emptyset$
	so that $f^{\otimes n} \colon Z_1^{\otimes n} \to Z_2^{\otimes n}$
	is
	\[
	Z_1^{\otimes n}
	\xrightarrow{k_1}
	\bullet
	\xrightarrow{k_2}
	\cdots
	\xrightarrow{k_{n-1}}
	\bullet
	\xrightarrow{k_{n}}
	Z_2^{\otimes n}.
	\] 
	Part (i) of the result now implies that 
	$g^{\square n-r}$
	is a $\F^{\ltimes n-r}$-cofibration
	and
	$f^{\square r}$
	is a $\F^{\ltimes r}$-(trivial) cofibration
	(the trivial case uses $r\geq 1$),
	so (ii) for $\G=G$ a group follows from 
	Lemma \ref{LTIMESNMINC LEM}.

	We now prove the result for general groupoids.
	Since any groupoid is equivalent to a disjoint union of groups, 
	we may reduce to that case, i.e. we may assume that two objects of $\G$ are isomorphic iff they are equal.
	
	Given $(x_i) \in \Sigma_n \ltimes \G^{\times n}$,
	we now need to check that the maps in \eqref{FSQNDEF EQ}
	are (trivial) $\left(\F^{\ltimes n}\right)_{(x_i)}$-cofibrations.
	Form the partition 
	$\{1,\cdots,n\} = \lambda_1 \amalg \cdots \amalg \lambda_k$
	such that $i,j$ are in the same piece iff $x_i=x_j$, and write
	$n_l = |\lambda_l|$.
	Writing $x_{\lambda_l}$ for the common value of the $x_i$ with $i\in \lambda_l$, we then have
	\[
	f^{\square n}\left((x_i)_{1\leq i \leq n }\right)
	\simeq
	\underset{1\leq l \leq k}{\mathlarger{\mathlarger{\mathlarger{\square}}}} f(x_{\lambda_l})^{\square n_l},
	\qquad
	f^{\otimes n}\left((x_i)_{1\leq i \leq n }\right)
	\simeq
	\underset{1\leq l \leq k}{\bigotimes} f(x_{\lambda_l})^{\otimes n_l}.
	\]
	Writing $G_l$ for the automorphism group of $x_{\lambda_l}$,
	the group case shown above shows that the
	$f(x_{\lambda_l})^{\square n_l}$,
	$f(x_{\lambda_l})^{\otimes n_l}$
	are (trivial) $\F_{x_{\lambda_l}}^{\ltimes n_l}$-cofibrations
	in $\V^{\Sigma_{n_l}\ltimes G_l}$,
	the latter with 
	$\F_{x_{\lambda_l}}^{\ltimes n_l}$-cofibrant domain.
	Next, note that for any subgroup
	$H \leq \mathsf{Aut}_{\Sigma_n \wr \G}((x_i))$ the projection 
	$\pi_{\Sigma}(H)$ must preserve the partition 
	(or else there would be distinct $x_i$ which are isomorphic) 
	so that, writing 
	$\pi_{\G^{\times \lambda_l}} \colon
	\Sigma_n \ltimes \G^{\times n} \to \G^{\times \lambda_l}$
	for the projections,
	one has that the 
	$\pi_{\G^{\times \lambda_l}}$ are homomorphisms when restricted to $H$. The result now follows by Lemma \ref{LTIMESNMINC LEM}
	and the observation that
	$H \in \left(\F^{\ltimes n}\right)_{(x_i)}$
	iff
	$\pi_{\G^{\times \lambda_l}}(H) \in \F_{x_{\lambda_l}}^{\ltimes n_l}$
	for all $l$.
\end{proof}

\begin{remark}
	If one focuses exclusively on cofibrations and ignores statements concerning trivial cofibrations,
	Proposition \ref{SIGMAWRGF PROP} actually subsumes Proposition \ref{RESGEN PROP}. 
	Indeed, by considering the disjoint union groupoid 
	$\G \amalg \bar{\G}$
	with the disjoint union family 
	$\F \amalg \bar{\F}$,
	one can check that, for
	$x \in \G$, $\bar{x} \in \bar{\G}$,
	one has that for the tuple $(x,\bar{x}) \in \Sigma_2 \wr (\G \amalg \bar{\G})$
	it is
	$\left(\left(\F \amalg \bar{\F}\right)^{\ltimes 2}\right)_{(x,\bar{x})} \simeq \F \sqcap \bar{\F}$.
\end{remark}

\section{Model structures on equivariant operads with fixed set of colors}\label{FIXCOL SEC}

The goal of this section is to prove our main results,
Theorems \ref{THMI} and \ref{THMII}.

As discussed in the introduction to \S \ref{PRE SEC},
in the non-equivariant context the model structures on 
color fixed operads $\mathsf{Op}_{\mathfrak{C}}(\V)$
are obtained by via transfer from model structures on the categories
$\mathsf{Sym}_{\mathfrak{C}}(\V)$
of color fixed symmetric sequences.
Likewise, our model structures 
on color fixed equivariant operads 
$\mathsf{Op}^G_{\mathfrak{C}}(\V)$
will be transferred from model structures on 
color fixed equivariant symmetric sequences
$\mathsf{Sym}^G_{\mathfrak{C}}(\V)$
where, just as in \eqref{THMI_EQ},
the weak equivalences in 
$\mathsf{Sym}^G_{\mathfrak{C}}(\V)$
are determined by a $(G,\Sigma)$-family
(cf. Definition \ref{FAM1ST DEF}).
This section is then organized as follows.

In \S \ref{SYMC_MS_SEC} we simply translate
the work in \S \ref{FGPP_SEC} to the categories 
$\mathsf{Sym}^{G}_{\mathfrak{C}}(\V)
\simeq \V^{G \ltimes \Sigma^{op}_{\mathfrak{C}}}$.

Then in \S \ref{OPC_MS_SEC}
we prove Theorem \ref{THMI} by transferring the model structures on 
$\mathsf{Sym}^{G}_{\mathfrak{C}}(\V)$ to obtain model structures on
$\mathsf{Op}^{G}_{\mathfrak{C}}(\V)$.
The key ingredients to this proof are
Proposition \ref{SIGMAWRGF PROP} in \S \ref{FGPP_SEC} and
the filtrations of free operad extensions 
in Lemma \ref{OURE LEM}
(whose proof is deferred to Appendix \ref{MONAD_APDX}).

Lastly, \S \ref{INDSYS_SEC} proves 
Theorem \ref{THMII} via a more careful analysis
of the filtrations in the proof of Theorem \ref{THMI}.

\subsection{Homotopy theory of symmetric sequences with a fixed color $G$-set}
\label{SYMC_MS_SEC}

Following the identification 
$\mathsf{Sym}^G_{\mathfrak{C}}(\V)
\simeq \V^{G \ltimes \Sigma^{op}_{\mathfrak{C}}}$
in Proposition \ref{EQUIVFNCON PROP},
we now apply the framework from \S \ref{FGPP_SEC}  
to the groupoids
$\G = G \ltimes \Sigma_{\mathfrak C}^{op}$
to build model structures on the categories 
$\mathsf{Sym}^G_{\mathfrak{C}}(\V)$.

We now recall and elaborate on the $(G,\Sigma)$-families in Definition \ref{FAM1ST DEF}.
Notably, the families $\mathcal{F}_{\mathfrak{C}}$ below 
are such that 
change of colors $\varphi \colon \mathfrak{C} \to \mathfrak{D}$ 
induce Quillen adjunctions, cf. Corollary \ref{SYMADJ_COR}(i).

\begin{definition}\label{GSFAM_DEF}
	A \emph{$(G,\Sigma)$-family} is a family of subgroups $\mathcal{F}$ of the groupoid $G \times \Sigma^{op}$;
	i.e.
	a collection of families $\F_n$ of the groups $G\times \Sigma_n^{op}$ for each $n \geq 0$.
        
	Moreover, given a $(G,\Sigma)$-family $\F$ and a color $G$-set $\mathfrak C \in \mathsf{Set}^G$,
	we define the family
	$\mathcal{F}_{\mathfrak{C}}$ in
	$G \ltimes \Sigma^{op}_{\mathfrak{C}}$
	by $\mathcal{F}_{\mathfrak{C}} = \pi^{\**}_{\mathfrak{C}}(\mathcal{F})$,
	with $\pi_{\mathfrak{C}} \colon G \ltimes \Sigma_{\mathfrak{C}}^{op} \to G \times \Sigma^{op}$
	the canonical forgetful functor.
\end{definition}

\begin{remark}\label{FAMC_DEF_REM}
	The functors
	$\pi_{\mathfrak{C}} \colon
	G \ltimes \Sigma_{\mathfrak{C}}^{op} \to
	G \times \Sigma^{op}$
	are faithful and can thus be regarded as inclusions on hom sets.
	Thus, letting $\vect{C} \in \Sigma_{\mathfrak{C}}$ be a
	$\mathfrak{C}$-colored corolla with $n$ leaves,
	one has that
	$\F_{\mathfrak{C}} = \{\F_{\vect{C}}\}$ where
	\begin{equation}\label{FAMC_DEF_EQ}
	\F_{\vect{C}} = \F_n \cap \Aut_{G \ltimes \Sigma_{\mathfrak C}^{op}}(\vect C).
	\end{equation}
	Alternatively, following
	Definition \ref{STABS DEF},
	$\F_{\vect{C}}$ consists of the $\Lambda \in \F_n$
	such that $\Lambda$ stabilizes $\vect{C}$.
\end{remark}

In this paper and the sequel \cite{BP_TAS}, 
we will be interested in three main examples of $(G,\Sigma)$-families:

\begin{enumerate}[label = (\alph*)]
	\item First, there is the family $\F_{all}$ of all the subgroups of $G \times \Sigma^{op}$
	(in which case the $\F_{all,\mathfrak{C}}$ are also the families of all subgroups), which is useful mainly for technical purposes.
	
	\item Secondly, there is the family of $\F^{\Gamma}$
	of $G$-graph subgroups (e.g. \cite[Def. 6.36]{BP21}),
	where $\F^{\Gamma}_n$ consists of the subgroups
	$\Gamma \leq G \times \Sigma_n^{op}$
	such that $\Gamma \cap \Sigma_n^{op} = \{\**\}$.
	We note that the elements of such $\Gamma$
	have the form $(h,\phi(h)^{-1})$
	for $h$ ranging over some subgroup $H \leq G$
	and $\phi \colon H \to \Sigma_n$
	a homomorphism,
	motivating the ``graph subgroup'' terminology.
	
	Although secondary for our work in the current paper, we regard $\F^{\Gamma}$ as the ``canonical choice'' of
	$(G,\Sigma)$-family, 
	as this is the family featured in the Quillen equivalence
	$W_! \colon 
	\mathsf{dSet}^G \rightleftarrows 
	\mathsf{dSet}^G \colon hcN$
	in \cite[Thm. I]{BP_TAS}.
	
	\item Lastly, there are the indexing systems of Blumberg and Hill,
	which are special subfamilies of $\F^{\Gamma}$
	which share the key technical properties of 
	$\F^{\Gamma}$ itself,
	and are discussed in \S \ref{INDSYS_SEC}.
\end{enumerate}

\begin{example}
	Let $G = \mathbb{Z}_{/2} = \{\pm 1\}$ and 
	$\mathfrak{C} = \{\mathfrak{a}, -\mathfrak{a}, \mathfrak{b}\}$ where we implicitly have
	$-\mathfrak{b} = \mathfrak{b}$.
	Consider the two $\mathfrak{C}$-corollas 
	$\vect{C},\vect{D} \in \Sigma_{\mathfrak{C}}$ below.
	
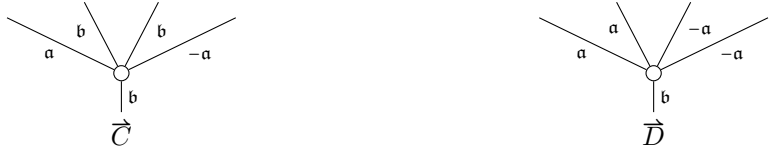
\begin{equation}
	\begin{tikzpicture}[auto,grow=up, level distance = 2.2em,
	every node/.style={font=\scriptsize,inner sep = 2pt}]%
	\tikzstyle{level 2}=[sibling distance=3em]%
	\node at (0,0) [font = \normalsize] {$\vect{C}$}%
	child{node [dummy] {}%
		child{node {}%
			edge from parent node [swap] {$-\mathfrak{a}$}}%
		child[level distance = 2.9em]{node {}%
			edge from parent node [swap,	near end] {$\mathfrak{b}$}}%
		child[level distance = 2.9em]{node {}%
			edge from parent node [near end] {$\mathfrak{b}$}}%
		child{node {}%
			edge from parent node  {$\mathfrak{a}$}}%
		edge from parent node [swap] {$\mathfrak{b}$}};%
	\node at (7,0) [font = \normalsize] {$\vect{D}$}%
	child{node [dummy] {}%
		child{node {}%
			edge from parent node [swap] {$-\mathfrak{a}$}}%
		child[level distance = 2.9em]{node {}%
			edge from parent node [swap,	near end] {$-\mathfrak{a}$}}%
		child[level distance = 2.9em]{node {}%
			edge from parent node [near end] {$\mathfrak{a}$}}%
		child{node {}%
			edge from parent node  {$\mathfrak{a}$}}%
		edge from parent node [swap] {$\mathfrak{b}$}};%
	\end{tikzpicture}%
	\end{equation}%
	The non-trivial $G$-graph subgroups of
	$\F^{\Gamma}_{\vect{C}}$,
	$\F^{\Gamma}_{\vect{D}}$
	then correspond to the possible $\mathbb{Z}_{/2}$-actions on the underlying trees $C,D$ which are compatible with the action on labels
	(in the sense that the composites
	$\boldsymbol{E}(C) \xrightarrow{-1} \boldsymbol{E}(C) \to \mathfrak{C}$
	and 
	$\boldsymbol{E}(C) \to \mathfrak{C} \xrightarrow{-1} \mathfrak{C}$ coincide).
	In this particular case, both 
	$\F^{\Gamma}_{\vect{C}}$,
	$\F^{\Gamma}_{\vect{D}}$
	have exactly two non-trivial groups,
	which correspond to the $\mathbb{Z}_{/2}$-actions on the underlying corollas that are depicted below.
	
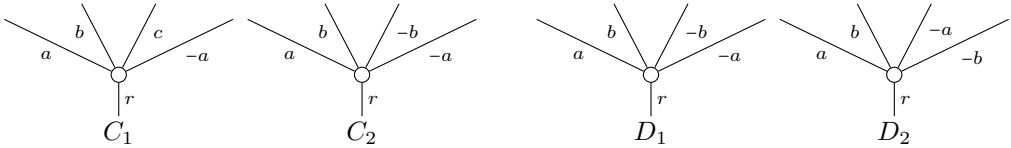
\begin{equation}
	\begin{tikzpicture}[auto,grow=up, level distance = 2.2em,
	every node/.style={font=\scriptsize,inner sep = 2pt}]%
	\tikzstyle{level 2}=[sibling distance=3em]%
	\node at (-1.6,0) [font = \normalsize] {$C_1$}%
	child{node [dummy] {}%
		child{node {}%
			edge from parent node [swap] {$-a$}}%
		child[level distance = 2.9em]{node {}%
			edge from parent node [swap,	near end] {$c\phantom{b}$}}%
		child[level distance = 2.9em]{node {}%
			edge from parent node [near end] {$b$}}%
		child{node {}%
			edge from parent node  {$a$}}%
		edge from parent node [swap] {$r$}};%
	\node at (1.6,0) [font = \normalsize] {$C_2$}%
	child{node [dummy] {}%
		child{node {}%
			edge from parent node [swap] {$-a$}}%
		child[level distance = 2.9em]{node {}%
			edge from parent node [swap,	near end] {$-b$}}%
		child[level distance = 2.9em]{node {}%
			edge from parent node [near end] {$b$}}%
		child{node {}%
			edge from parent node  {$a$}}%
		edge from parent node [swap] {$r$}};%
	\node at (5.4,0) [font = \normalsize] {$D_1$}%
	child{node [dummy] {}%
		child{node {}%
			edge from parent node [swap] {$-a$}}%
		child[level distance = 2.9em]{node {}%
			edge from parent node [swap,	near end] {$-b$}}%
		child[level distance = 2.9em]{node {}%
			edge from parent node [near end] {$b$}}%
		child{node {}%
			edge from parent node  {$a$}}%
		edge from parent node [swap] {$r$}};%
	\node at (8.6,0) [font = \normalsize] {$D_2$}%
	child{node [dummy] {}%
		child{node {}%
			edge from parent node [swap] {$-b$}}%
		child[level distance = 2.9em]{node {}%
			edge from parent node [swap,	near end] {$-a$}}%
		child[level distance = 2.9em]{node {}%
			edge from parent node [near end] {$b$}}%
		child{node {}%
			edge from parent node  {$a$}}%
		edge from parent node [swap] {$r$}};%
	\end{tikzpicture}%
	\end{equation}%
\end{example}

\begin{definition}\label{SYMGFV DEF}
	Let $\mathfrak C$ be a $G$-set and $\F$ a $(G, \Sigma)$-family.
	Then the \textit{$\F$-model structure} (if it exists) on the fiber $\Sym^{G}_{\mathfrak C}(\V)$ of $\Sym(\V)$ over $\mathfrak C$
	is the $\F_{\mathfrak{C}}$-model structure
	\begin{equation}
                \Sym^{G}_{\mathfrak{C},\F}(\V) = \V^{G \ltimes \Sigma_{\mathfrak C}^{op}}_{\F_{\mathfrak{C}}}
	\end{equation}
	where $\F_{\mathfrak{C}}$ is as in Definition \ref{GSFAM_DEF}.
        Explicitly, a map $X \to Y$ in $\Sym^G_{\mathfrak C, \mathcal F}(\V)$ is a weak equivalence (resp. fibration) if
        $X(\vect C)^\Lambda \to Y(\vect C)^\Lambda$       
        is so in $\V$ for all $\mathfrak C$-profiles $\vect C$ and $\Lambda \in \mathcal F_{\vect C}$. 

	When $\mathcal{F}=\mathcal{F}_{all}$ is the family of all subgroups,
        we refer to this model structure simply as the \emph{genuine model structure} on $\mathsf{Sym}^G_{\mathfrak{C}}(\V)$.
\end{definition}

\begin{remark}\label{VGSIGF REM}
	If $\V$ is cofibrantly generated,
	combining \eqref{VGFGEN EQ} with the 
	$\Sigma_{\mathfrak{C}}[G \cdot_{\mathfrak{C}} \vect{C}]$ notation
	for the representable functors 
	in Proposition \ref{REPALTDESC PROP},
	one has that the generating (trivial) cofibrations of
	$\Sym^{G}_{\mathfrak{C},\F}(\V)$
	are the sets of maps
	\begin{equation}\label{VGSIGF EQ}
	\mathcal{I}_{\mathfrak{C},\mathcal{F}}
	=
	\left\{
	\Sigma_{\mathfrak{C}}[G \cdot_{\mathfrak{C}} \vect{C}]/\Lambda \cdot i
	\right\}
	\qquad \qquad
	\mathcal{J}_{\mathfrak{C},\mathcal{F}}
	=
	\left\{
	\Sigma_{\mathfrak{C}}[G \cdot_{\mathfrak{C}} \vect{C}]/\Lambda \cdot j
	\right\}
	\end{equation}
	where $\vect{C}$ ranges over $\Sigma_{\mathfrak{C}}$,
	$\Lambda$ ranges over $\F_{\vect{C}}$,
	$i$ ranges over $\mathcal{I}$ and
	$j$ ranges over $\mathcal{J}$.
\end{remark}

\begin{corollary}\label{SYMADJ_COR}
        Suppose the model structures $\Sym_{\mathfrak C, \F}^G(\V)$ exist for all $(G, \Sigma)$-families $\F$ and $G$-sets $\mathfrak C$.
	\begin{enumerate}[label=(\roman*)]
        \item \label{SYMCOCHADJ_LBL}
                For any $(G,\Sigma)$-family $\F$ and
                map of colors $\varphi \colon \mathfrak{C} \to \mathfrak{D}$
                the induced adjunction
                \[
                        \varphi_! \colon \mathsf{Sym}^G_{\mathfrak{C},\F}(\V)
                        \rightleftarrows
                        \mathsf{Sym}^G_{\mathfrak{D},\F}(\V) \colon \varphi^{\**}
                \]
                is a Quillen adjunction.
                
        \item \label{FIXSETCHGR_LBL}
                For any homomorphism $\phi \colon G \to \bar G$ (and writing $\phi \colon G \times \Sigma^{op} \to \bar G \times \Sigma^{op}$ for the induced homomorphism),
                $(G,\Sigma)$-family $\F$ and $(\bar G,\Sigma)$-family $\bar \F$,
                and $\bar{G}$-set of colors $\mathfrak C$,
                the adjunction
                \[
                        \phi_! \colon \mathsf{Sym}^G_{\mathfrak{C},\F}(\V)
                        \rightleftarrows
                        \mathsf{Sym}^{\bar{G}}_{\mathfrak{C},\bar{\F}}(\V) \colon \phi^{\**}
                \]
                is a Quillen adjunction whenever $\F \subseteq \phi^{\**} \bar{\F}$.
                
        \item 
                For any homomorphism $\phi \colon G \to \bar G$,
                $(G,\Sigma)$-family $\F$ and $(\bar G,\Sigma)$-family $\bar{\F}$,
                and $G$-set of colors $\mathfrak C$,
                the adjunction
                \[
                        \bar{G} \cdot_G (-) \colon \mathsf{Sym}^G_{\mathfrak{C},\F}(\V)
                        \rightleftarrows
                        \mathsf{Sym}^{\bar{G}}_{\bar{G} \cdot_G \mathfrak{C},\bar{\F}}(\V) \colon \mathsf{fgt}
                \]
                is a Quillen adjunction whenever $\F \subseteq \phi^{\**} \bar{\F}$.
        \end{enumerate}
\end{corollary}

\begin{proof}
        Parts (i) and (ii) are immediate from Proposition \ref{EQQUILADJ PROP}.
        Part (iii) follows by combining (i) applied to 
        the map of $G$-sets $\mathfrak{C} \to \bar{G} \cdot_G \mathfrak{C}$
        with (ii) applied to the $\bar{G}$-set $\bar{G} \cdot_G \mathfrak{C}$.
\end{proof}

\subsection{Homotopy theory of operads with a fixed color $G$-set}
\label{OPC_MS_SEC}

Recall that $\mathbb{F}^G_{\mathfrak{C}}$
denotes the fiber of the monad $\mathbb{F}^G$
for $\mathfrak{C} \in \mathsf{Set}^G$
(see Proposition \ref{FGC PROP}).

Our goal in this section is to prove 
Theorem \ref{THMI} by transferring the $\F$-model structures on
$\mathsf{Sym}^G_{\mathfrak{C}}(\V)$
from Definition \ref{SYMGFV DEF}
to $\mathsf{Op}^G_{\mathfrak{C}}(\V)$
along the free-forgetful adjunction
\begin{equation}\label{OPAUTADJ EQ}
\mathbb{F}^G_{\mathfrak{C}} \colon
\mathsf{Sym}^G_{\mathfrak{C}}(\V)
\rightleftarrows
\mathsf{Op}^G_{\mathfrak{C}}(\V)
\colon \mathsf{fgt}
\end{equation}
so that a map in $\mathsf{Op}^G_{\mathfrak{C}}(\V)$
is a weak equivalence/fibration iff so is the underlying map in 
$\mathsf{Sym}^G_{\mathfrak{C}}(\V)$.

The key to proving Theorem \ref{THMI} will be a suitable filtration,
given by Lemma \ref{OURE LEM},
of free operad extensions in $\mathsf{Op}_{\mathfrak{C}}^G(\V)$,
i.e. of pushouts as in \eqref{OURE EQ} below.
The following discussion and lemma summarizes the key properties of the filtration we will need.

We denote by $\Omega^a_{\mathfrak{C}}$
the $\mathfrak{C}$-colored variant of the 
\emph{alternating trees} $\Omega^a$ of \cite[Def. 5.52]{BP21}
(formally, $\Omega^a_{\mathfrak{C}}$ is defined by copying
Definition \ref{COLFOR DEF} but with 
$F$, $\rho \colon F \to F'$ therein now in $\Omega^a$).
Its objects are trees $\vect{T}$ 
whose set of vertices is partitioned into ``active'' and ``inert'' vertices,
$\boldsymbol{V}(\vect T) = 
\boldsymbol{V}^{ac}(\vect T) \amalg \boldsymbol{V}^{in}(\vect T)$,
in such a way that adjacent vertices are in different sides of the partition and the ``outer vertices'' (i.e. those adjacent to a leaf or root) are active.
Its maps are informally described as tall maps of trees that 
``send inert vertices to inert vertices''.

\begin{example}\label{ALTMAP EX}
	Below we depict a planar alternating map $f \colon T \to S$
	between alternating trees $S,T \in \Omega^a$.
	Active vertices are black $\bullet$ and inert vertices are white $\circ$.
	Pictorially, active vertices $\bullet$
	can be expanded to an alternating tree of the same arity,
	while inert vertices $\circ$ are preserved.
	
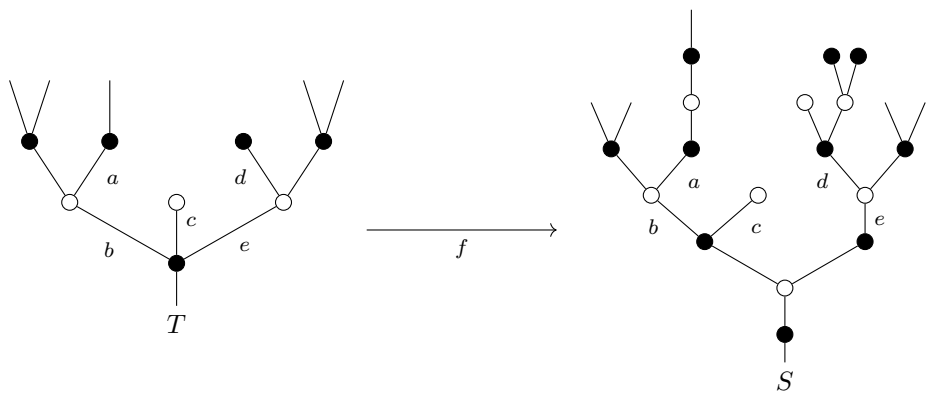
\[
	\begin{tikzpicture}[grow=up,auto,level distance=2.3em,every node/.style = {font=\footnotesize},dummy/.style={circle,draw,inner sep=0pt,minimum size=2.1mm}]
	\tikzstyle{level 2}=[sibling distance = 4em]
	\tikzstyle{level 3}=[sibling distance = 3em]
	\tikzstyle{level 4}=[sibling distance = 1.5em]
	\node at (0,0.75) [font = \normalsize] {$T$}
	child{node [dummy,fill = black] {}
		child{node [dummy,fill=white] {}
			child{node [dummy,fill = black] {}
				child
				child
			}
			child{node [dummy,fill = black] {}
				edge from parent node [near end] {$d$}}
			edge from parent node [swap] {$e$}}
		child{node [dummy,fill=white] {}
			edge from parent node [swap, near end] {$c$}}
		child{node [dummy,fill=white] {}
			child{node [dummy,fill = black] {}
				child
				edge from parent node [swap, near end] {$a\phantom{d}$}}
			child{node [dummy,fill = black] {}
				child
				child
			}
			edge from parent node {$b$}}
	};
	\begin{scope}[level distance=1.75em]
	\tikzstyle{level 3}=[sibling distance = 6em]
	\tikzstyle{level 4}=[sibling distance = 4em]
	\tikzstyle{level 5}=[sibling distance = 3em]
	\tikzstyle{level 6}=[sibling distance = 1.5em]
	\tikzstyle{level 7}=[sibling distance = 1em]
	\node at (8,0) [font = \normalsize] {$S$}
	child{node [dummy,fill = black] {}
		child{node [dummy,fill = white] {}
			child{node [dummy,fill = black] {}
				child{node [dummy,fill=white] {}
					child{node [dummy,fill = black] {}
						child
						child
					}
					child{node [dummy,fill = black] {}
						child{node [dummy,fill=white] {}
							child{node [dummy,fill = black] {}}
							child{node [dummy,fill = black] {}}
						}
						child{node [dummy,fill=white] {}}
						edge from parent node [near end] {$d$}}
					edge from parent node [swap] {$e\phantom{1}$}}
			}
			child{node [dummy,fill = black] {}
				child{node [dummy,fill=white] {}
					edge from parent node [swap, near end] {$c\phantom{1}$}}
				child{node [dummy,fill=white] {}
					child{node [dummy,fill = black] {}
						child{node [dummy,fill=white] {}
							child{node [dummy,fill = black] {}
								child
							}
						}
						edge from parent node [swap,near end] {$a\phantom{d}$}}
					child{node [dummy,fill = black] {}
						child
						child
					}
					edge from parent node [near end] {$\phantom{1}b$}}
			}
		}
	};
	\end{scope}
	\draw [->] (2.5,2) -- node [swap] {$f$} (5,2);
	\end{tikzpicture}
	\]
\end{example}

Furthermore, we write
$\Omega_{\mathfrak C}^a[k] \subseteq \Omega_{\mathfrak C}^a$
for the full subcategory of $\vect{T}$ such that 
$|\boldsymbol{V}^{in}(\vect{T})| = k$,
i.e. such that $\vect{T}$ has exactly $k$ inert vertices.
Crucially, $\Omega^a_{\mathfrak{C}}[k]$
turns out to be a groupoid
(indeed, maps in $\Omega^a_{\mathfrak{C}}[k]$ can only replanarize vertices, 
or else the number of $\circ$ vertices would increase).
Adapting \eqref{LRDEF EQ} and \eqref{VFUNDEF EQ}, 
we have functors
\[
\mathsf{lr}_{\mathfrak C}^{a} \colon
\Omega^a_{\mathfrak{C}}[k]
\to
\Sigma_{\mathfrak{C}}
\qquad
\boldsymbol{V}^{in} \colon
\Omega^a_{\mathfrak{C}}[k]
\to
\Sigma_k \wr \Sigma_{\mathfrak{C}}
\]
where $\mathsf{lr}_{\mathfrak C}^{a}$ is defined as before (one just ignores the vertex partition data)
and $\boldsymbol{V}^{in}$ is the tuple containing only the inert vertices.

The proof of the following key Lemma is postponed to \S \ref{PUSHOUT_SEC} in the Appendix.
\begin{lemma}\label{OURE LEM}
	Fix a $G$-set of colors $\mathfrak{C}$ and let
	$u\colon X \to Y$ be a map in $\mathsf{Sym}^G_{\mathfrak{C}}(\V)$
	and
	$\mathbb{F} X \to \O$ be a map in $\mathsf{Op}^G_{\mathfrak{C}}(\V)$.
	Then, for the pushout 
	\begin{equation}\label{OURE EQ}
	\begin{tikzcd}
	\mathbb F X \arrow[d, "\mathbb{F}u"'] \arrow[r]
	&
	\O \arrow[d]
	\\
	\mathbb F Y \arrow[r]
	&
	\O[u]
	\end{tikzcd}
	\end{equation}
	in the category of operads $\mathsf{Op}^G_{\mathfrak{C}}(\V)$
	the map $\O \to \O[u]$ admits an underlying filtration
	\begin{equation}\label{OUFILRE EQ}
	\O = \O_0 \to \O_1 \to \O_2 \to \cdots \to \O_{\infty} = \O[u]
	\end{equation}
	of maps in $\mathsf{Sym}^G_{\mathfrak{C}}(\V)$ 
	where each map $\O_{k-1} \to \O_k$ fits into a pushout
	\begin{equation}\label{OUPUSHRE EQ}
	\begin{tikzcd}
	\bullet 
	\arrow{d}[swap]{\left(\mathsf{lr}_{\mathfrak C}^{a,op}\right)_!
		n_k^{(\O,X,Y)}}
	\arrow[r]
	&
	\O_{k-1} \arrow[d]
	\\
	\bullet \arrow[r]
	&
	\O_k
	\end{tikzcd}
	\end{equation}
	for 
	$
	\left(\mathsf{lr}_{\mathfrak C}^{a,op}\right)_! \colon
	\mathcal{V}^{G \ltimes \Omega^a_{\mathfrak{C}}[k]^{op}}
	\to
	\mathcal{V}^{G \ltimes \Sigma_{\mathfrak{C}}^{op}}
	=
	\mathsf{Sym}^G_{\mathfrak{C}}(\V)
	$
	is as in Proposition \ref{EQQUILADJ PROP}
	and $n_k^{(\O,X,Y)}$ the natural transformation in 
	$\mathcal{V}^{G \ltimes \Omega^a_{\mathfrak{C}}[k]^{op}}$
	whose constituent arrows for $\vect{T} \in \Omega^a_{\mathfrak C}[k]$ are
	\begin{equation}\label{NKOXY EQ}
	n_k^{(\O,X,Y)}(\vect{T})=
	\left(
	\bigotimes_{v \in \boldsymbol{V}^{ac}(\vect{T})}\O(\vect{T}_v)
	\otimes
	\mathop{\mathlarger{\mathlarger{\mathlarger{\square}}}}\limits_{v \in \boldsymbol{V}^{in}(\vect{T})} u(\vect{T}_v)
	\right)
	=
	\left(
	\mathop{\mathlarger{\mathlarger{\mathlarger{\square}}}}\limits_{v \in \boldsymbol{V}^{ac}(\vect{T})} \left( \emptyset \to \O(\vect{T}_v)\right) 
	\square
	\mathop{\mathlarger{\mathlarger{\mathlarger{\square}}}}\limits_{v \in \boldsymbol{V}^{in}(\vect{T})} u(\vect{T}_v)
	\right)
	\end{equation}
\end{lemma}

We can now prove our first main result, Theorem \ref{THMI}.

\begin{proof}[Proof of Theorem \ref{THMI}]    
        Since we are assuming $\V$ is cofibrantly generated and that the genuine model structure on $\V^G$ exists,
        Propositions \ref{EQUIVFNCON PROP} and \ref{ALLEQ PROP}
        imply that the model structure
        $\Sym_{\mathfrak C, \F}^G(\V)$ exists for all $(G, \Sigma)$-families $\F$ and $G$-sets $\mathfrak C$.
        Unpacking Definition \ref{SYMGFV DEF}, the desired $\F$-model structure $\Op_{\mathfrak C, \F}^G(\V)$ from Theorem \ref{THMI}
        is the model structure transferred from $\Sym_{\mathfrak C, \F}^G(\V)$ along the free-forgetful adjunction.
        Thus, 
        writing
	$\mathcal{I}_{\mathfrak{C},\mathcal{F}}$,
	$\mathcal{J}_{\mathfrak{C},\mathcal{F}}$
	for the generating sets of
        $\mathsf{Sym}^G_{\mathfrak C, \F}(\V)$
	(cf. Remark \ref{VGSIGF REM}),
        the $\F$-model structure $\mathsf{Op}^G_{\mathfrak C, \F}(\V)$
	will have generating sets 
	$\mathbb{F}^G_{\mathfrak{C}}\mathcal{I}_{\mathfrak{C},\mathcal{F}}$,
	$\mathbb{F}^G_{\mathfrak{C}}\mathcal{J}_{\mathfrak{C},\mathcal{F}}$.
	
	By \cite[Thm. 11.3.2]{Hir03},
	we need only show that all maps in
	$(\mathbb{F}^G_{\mathfrak{C}}\mathcal{J}_{\mathfrak{C},\mathcal{F}})$-cell
	are $\F$-weak equivalences in $\mathsf{Sym}^G_{\mathfrak{C}}(\V)$.
	Moreover, since trivial $\F$-cofibrations are genuine trivial cofibrations and 
	genuine weak equivalences are $\F$-weak equivalences, one needs only consider the genuine case, i.e. the case of $\F_{all}$ the family of all subgroups (this repeats the argument in \eqref{FAMTOGEN EQ}).

	Since we are assuming $\V$ is a symmetric monoidal model category that satisfies the global monoid axiom (Definition \ref{GLOBMONAX_DEF}),
	it suffices to show that pushouts of maps in 
	$(\mathbb{F}^G_{\mathfrak{C}}\mathcal{J}_{\mathfrak{C},\mathcal{F}_{all}})$,
	i.e. maps $\O \to \O[u]$ in \eqref{OURE EQ}
	for $u \in \mathcal{J}_{\mathfrak{C},\mathcal{F}_{all}}$,
	are equivariant $\otimes$-trivial cofibrations
	in $\mathsf{Sym}^G_{\mathfrak{C}}(\V) = \V^{G \ltimes \Sigma^{op}_{\mathfrak{C}}}$ (Definition \ref{GGENOTITC DEF}).

	We thus now let 
	$u \in \mathcal{J}_{\mathfrak{C},\mathcal{F}_{all}}$.
	Now note that, for $\vect{T} \in \Omega^a_{\mathfrak{C}}[k]$,
	we have
	\[
	\mathop{\mathlarger{\mathlarger{\mathlarger{\square}}}}\limits_{v \in \boldsymbol{V}^{in}(\vect{U})} u(\vect{T}_v)
	=
	u^{\square k}\left( \boldsymbol{V}^{in}(\vect T) \right)
	\]
	where
	$u^{\square k} \in 
	\V^{\Sigma_k \wr (G \ltimes \Sigma^{op}_{\mathfrak{C}})}$
	is as defined in \eqref{FSQNDEF EQ}.
	Since we are assuming $\V$ has cofibrant symmetric pushout powers,
        Proposition \ref{SIGMAWRGF PROP} implies that
	$u^{\square k}$ is a genuine trivial cofibration in 
	$\V^{\Sigma_k \wr (G \ltimes \Sigma^{op}_{\mathfrak{C}})}$.
	In turn,
	Proposition \ref{EQQUILADJ PROP} implies that
	$u^{\square k}(\boldsymbol{V}^{in}(-))$
	is similarly a genuine trivial cofibration in 
	$\V^{G \ltimes \Omega^{a,op}_{\mathfrak{C}}[k]}$.
	It is now clear that
	the map $n_k^{(\O,X,Y)}$ in \eqref{NKOXY EQ}
	is an equivariant
	$\otimes$-trivial cofibration in 
	$\V^{G \ltimes \Omega^{a,op}_{\mathfrak{C}}[k]}$,
	so that by 
	Proposition \ref{REGEOTCOF PROP}
	the maps $\O_{k-1} \to \O_{k}$
	in the pushouts \eqref{OUPUSHRE EQ}
	are genuine $\otimes$-trivial cofibrations
	in 
	$\mathsf{Sym}^G_{\mathfrak{C}}(\V)
	= \V^{G\ltimes \Sigma^{op}_{\mathfrak{C}}}$.
	Thus the composite $\O \to \O[u]$
	is also an
	equivariant $\otimes$-trivial cofibration,
	and the result follows.
\end{proof}

For later reference, 
we highlight two consequences of the previous proof.

\begin{remark}\label{GOTC_REM}
	$\F$-trivial cofibrations in $\Op^G_{\mathfrak C}(\V)$ are underlying genuine $\otimes$-trivial cofibrations
	in $\mathsf{Sym}^G_{\mathfrak C}(\V)$.
\end{remark}

\begin{remark}\label{FVGSIGF REM}
	The generating (trivial) cofibrations in
	$\mathsf{Op}^G_{\mathfrak{C},\F}(\V)$
	are the sets
	\begin{equation}\label{FVGSIGF EQ}
	\mathbb{F}^G_{\mathfrak{C}}\mathcal{I}_{\mathfrak{C},\mathcal{F}}
	=
	\left\{
	\mathbb{F}^G_{\mathfrak{C}}
	\left(\Sigma_{\mathfrak{C}}[G \cdot_{\mathfrak{C}} \vect{C}]/\Lambda \cdot i \right)
	\right\}
	\qquad \qquad
	\mathbb{F}^G_{\mathfrak{C}}\mathcal{I}_{\mathfrak{C},\mathcal{F}}
	=
	\left\{
	\mathbb{F}^G_{\mathfrak{C}}
	\left(\Sigma_{\mathfrak{C}}[G \cdot_{\mathfrak{C}} \vect{C}]/\Lambda \cdot j \right)
	\right\}
	\end{equation}
	where $\vect{C}$ ranges over $\Sigma_{\mathfrak{C}}$,
	$\Lambda$ ranges over $\F_{\vect{C}}$,
	$i$ ranges over $\mathcal{I}$,
	and $j$ ranges over $\mathcal{J}$.
\end{remark}

\begin{corollary}\label{OPADJ_COR}
        Suppose the assumptions of Theorem \ref{THMI} hold, so the model structures discussed in the items below exist.
	\begin{enumerate}[label=(\roman*)]
		\item \label{OPCOCHADJ_LBL}
		For any $(G,\Sigma)$-family $\F$ and map of colors 
		$\varphi \colon \mathfrak C \to \mathfrak D$, the induced adjunction
		\[
		\check{\varphi}_! \colon \mathsf{Op}^G_{\mathfrak{C},\F}(\V)
		\rightleftarrows
		\mathsf{Op}^G_{\mathfrak{D},\F}(\V) \colon \varphi^{\**}
		\]
		is a Quillen adjunction.
		\item \label{OPFIXSETCHGR_LBL}
		For any homomorphism $\phi \colon G \to \bar G$,
		$(G,\Sigma)$-family $\F$ and $(\bar G,\Sigma)$-family $\bar{\F}$,
		and $\bar G$-set of colors $\mathfrak C$,
		the adjunction
		\[
		\check{\phi}_! \colon \mathsf{Op}^G_{\mathfrak{C},\F}(\V)
		\rightleftarrows
		\mathsf{Op}^{\bar{G}}_{\mathfrak{C},\bar{\F}}(\V) \colon \phi^{\**}
		\]
		is a Quillen adjunction whenever $\F \subseteq \phi^{\**} \bar{\F}$.
		\item \label{OPCOMBADJ_LBL}
		For any homomorphism $\phi \colon G \to \bar G$,
		$(G,\Sigma)$-family $\F$ and $(\bar G,\Sigma)$-family $\bar{\F}$,
		and $G$-set of colors $\mathfrak C$,
		the adjunction
		\[
		\bar{G} \cdot_G (-) \colon \mathsf{Op}^G_{\mathfrak{C},\F}(\V)
		\rightleftarrows
		\mathsf{Op}^{\bar{G}}_{\bar{G} \cdot_G \mathfrak{C},\bar{\F}}(\V) \colon \mathsf{fgt}
		\]
		is a Quillen adjunction whenever $\F \subseteq \phi^{\**} \bar{\F}$.
	\end{enumerate}
\end{corollary}

\begin{proof}
	Since operadic weak equivalences, fibrations and the forgetful functors $\varphi^{\**},\phi^{\**}$, $\mathsf{fgt}$
	are all defined in terms of the underlying symmetric sequences,
	this follows from Corollary \ref{SYMADJ_COR}.
\end{proof}

\begin{remark}\label{CSPNTHI REM}
	When \eqref{OURE EQ}
	is a diagram in $\mathsf{Cat}^G_{\mathfrak{C}}(\V)$, 
	the map 
	$n_k^{(\mathcal{O},X,Y)}(\vect{T})$
	in \eqref{NKOXY EQ}
	is necessarily $\emptyset \to \emptyset$
	unless $\vect{T}$ is a linear tree.
	But, if $\vect{T}$ is linear,
	its automorphism group in 
	$G \ltimes \Omega_{\mathfrak{C}}^{a,op}[k]$
	is simply a subgroup of $G$
	and does not permute the factors in \eqref{NKOXY EQ}.
	Hence, the key claim in the proof of Theorem \ref{THMI},
	i.e. that 
	$n_k^{(\mathcal{O},X,Y)}$
	is an equivariant $\otimes$-trivial cofibration in
	$\V^{G \ltimes \Omega_{\mathfrak{C}}^{a,op}[k]}$,
	follows by replacing the use of Proposition \ref{SIGMAWRGF PROP}
	with the more elementary
	Proposition \ref{RESGEN PROP},
	and thus does not require
	the cofibrant pushout powers condition
	in Definition \ref{CSPP_DEF} and Theorem \ref{THMI}.
\end{remark}

\begin{remark}\label{THMISM REM}
	If $\O$ in \eqref{OURE EQ}
	is known to be underlying genuine cofibrant
	(i.e. $\F_{all}$-cofibrant) in 
	$\mathsf{Sym}^{G}_{\mathfrak{C}}(\V)$ 
	then, replacing the use of the global monoid axiom with 
	Proposition \ref{RESGEN PROP},
	the argument in the proof of Theorem \ref{THMI}
	shows that the maps
	$\O_{k-1} \to \O_k$
	and their composite
	$\O \to \O[u]$
	are genuine trivial cofibrations
	(rather than just genuine $\otimes$-trivial cofibrations).
	
	On the other hand, conditions (i),(ii),(iii),(iv) in Theorem \ref{THMI}
	suffice to show that if $\O$ is $\mathcal{F}_{all}$-cofibrant in 
	$\mathsf{Op}^G_{\mathfrak{C}}(\V)$
	then it is underlying $\mathcal{F}_{all}$-cofibrant in 
	$\mathsf{Sym}^G_{\mathfrak{C}}(\V)$
	(this follows either by the proof of Theorem \ref{THMII} in \S \ref{INDSYS_SEC},
	or by repeating the argument in the proof of Theorem \ref{THMI} but now with $u \in \mathcal{I}_{\mathfrak{C},\F_{all}}$
	a generating cofibration).
	
	Therefore, by \cite[Thm. 2.2.2]{WY18},
	the semi-model structure analogue of Theorem \ref{THMI}
	does not require the global monoid axiom (v).
\end{remark}

\subsection{Pseudo indexing systems and underlying cofibrancy}
\label{INDSYS_SEC}

Our goal in this subsection is to prove Theorem \ref{THMII},
stating that 
for certain $(G,\Sigma)$-families $\F$
the forgetful functor
$\Op^G_{\mathfrak C, \F}(\V) \to \Sym^G_{\mathfrak C, \F}(\V)$ preserves cofibrations between cofibrant objects.

\begin{definition}
	Let $\F$ be a $(G,\Sigma)$-family.
	We say $\O \to \P$ in $\Op^G_{\mathfrak C}(\V)$
	is a \textit{$\Sigma_\F$-cofibration}
	if the underlying map is a cofibration in $\Sym^G_{\mathfrak C, \F}(\V)$.
	Similarly, we say    
	$\O$ is \textit{$\Sigma_\F$-cofibrant} if 
	the map $\emptyset \to \O$
	in $\mathsf{Sym}^G_{\mathfrak{C}}$ is an $\F$-cofibration.
\end{definition}

\begin{remark}
	Following Remark \ref{SIGMACOF_REM}, 
	note that 
	$\O \to \P$ is a $\Sigma_\F$-cofibration iff 
	$\O(\vect C) \to \P(\vect C)$ is an $\mathcal F_{\vect C}$-cofibration in $\V^{\Aut(\vect C)}$
	for all $\mathfrak C$-profiles $\vect C$.
\end{remark}

We now introduce some notation needed
to define the pseudo indexing systems in Theorem \ref{THMII}.

\begin{notation}\label{FSQUARE_NOT}
	Let $\F$ be a family of subgroups of the groupoid $\G$,
	cf. Definition \ref{FAMGROUPOID DEF}.
	We write $\F^{\ltimes}$
	for the family in 
	$\Sigma \wr \G \simeq \coprod_{n \geq 0} \Sigma_{n} \wr \G$
	given by the union
	$\coprod_{n \geq 0} \F^{\ltimes n}$,
	where $\F^{\ltimes n}$ for $n\geq 1$
	is as in \eqref{FWRNXI EQ},
	and $\F^{\ltimes 0}$ is the non-empty family of $\Sigma_0$.
	
	Similarly, for a map $f$ in $\V^{\mathcal G}$, 
	we write $f^{\square}$ for the map in $\V^{\Sigma \wr \mathcal G}$
	which is given by the map $f^{\square n}$
	in \eqref{FSQNDEF EQ}
	on each summand $\Sigma_n \wr \mathcal{G}$ with $n \geq 1$,
	and by $f^{\square 0} = (\emptyset \to 1_{\V})$
	on the summand 
	$\Sigma_0 \wr \mathcal{G} = \Sigma_0$.
\end{notation}

\begin{notation}
	For any $G$-set $\mathfrak C$, we let $\boldsymbol{V}_{\mathfrak C}$ denote composite below, which is 
	the (opposite of the) composite of the first two horizontal maps in \eqref{FGC_EQ}.
	\begin{equation}\label{VC_EQ}
	\boldsymbol{V}_{\mathfrak{C}} \ \colon \ 
	G^{op} \ltimes \Omega_{\mathfrak{C}}^{0}
	\xrightarrow{\ G^{op} \ltimes \boldsymbol{V} \ }
	G^{op} \ltimes \left(\Sigma \wr \Sigma_{\mathfrak{C}}\right)
	\xrightarrow{\ \Delta \ }
	\Sigma \wr  \left(G^{op} \ltimes \Sigma_{\mathfrak{C}} \right)
	\end{equation}
	If $\mathfrak C = \**$, we abuse notation slightly and simply write $\boldsymbol{V} = \boldsymbol{V}_{\**}$. 
\end{notation}

In the following 
we slightly abuse notation by conflating
a $(G,\Sigma)$-family $\F$,
i.e. a family in $G \times \Sigma^{op}$,
with its opposite family
in $G^{op} \times \Sigma$.

\begin{definition}\label{PIS_DEF}
	A $(G,\Sigma)$-family $\F$ is called a \textit{pseudo indexing system} if
	\begin{equation}\label{PIS_EQ}
	\boldsymbol{V}^{\**} \F^\ltimes \subseteq \mathsf{lr}^{\**}\F,
	\end{equation}
	for $\F^\ltimes$ as in Notation \ref{FSQUARE_NOT},
	and
	$\boldsymbol{V}^{\**} \F^\ltimes, \mathsf{lr}^{\**}\F$
	as defined by \eqref{PHISTAR_EQ}.
\end{definition}

\begin{remark}
	Pseudo indexing systems generalize the
	\textit{weak indexing systems} in our previous work
	\cite[Def. 9.5]{Per18}, \cite[Def. 4.58]{BP21}, \cite[Def. 6.2]{BP20}
	which are themselves a generalization of the 
	\textit{indexing systems} introduced by Blumberg-Hill \cite[Def. 3.22]{BH15}.
	
	More precisely, \cite[Rem. 6.47]{BP21} implies that
	a weak indexing system is precisely a pseudo indexing system consisting of $G$-graph subgroups (cf. \eqref{GRAPHEQ EQ}, \cite[Def. 6.36]{BP21}).
\end{remark}

\begin{remark}
	Unpacking \eqref{FWRNXI EQ}, one obtains a more 
	explicit description for the family
	$\boldsymbol{V}^{\**}\F^\ltimes$.
	
	For a tree $T \in \Omega$
	and vertex $v \in \boldsymbol{V}(T)$, write 
	$\mathsf{Aut}_{v}(T) \leq \mathsf{Aut}(T)$ for the subgroup which fixes the root of $T_v$.
	One then has a natural homomorphism
	$\pi_{T_v} \colon \mathsf{Aut}_v(T) \to \mathsf{Aut}(T_v)$
	(note that $\pi_{T_v}$ can not be defined on the full group
	$\mathsf{Aut}(T)$).
	Given a $(G,\Sigma)$-family $\F$ and a subgroup $\Lambda \leq G^{op} \times \mathsf{Aut}(T)$,
	one then has
	$\Lambda \in (\boldsymbol{V}^{\**}\F^\ltimes)_T$ iff
	$\pi_{T_v}\left( \Lambda \cap (G \times \mathsf{Aut}_{v}(T)) \right) \in \F_{T_v}$
	for all $v \in \boldsymbol{V}(T)$.
	
	Lastly, we note that the $\boldsymbol{V}^{\**}\F^\ltimes$
	construction generalizes a construction in
	\cite{BP21}.
	Following \cite[Rem. 6.48]{BP21}, 
	if $\F$ consists of $G$-graph subgroups
	(so that it can be identified with a 
	\textit{family of $G$-corollas} in the sense of \cite[Def. 4.53]{BP21})
	the family $(\boldsymbol{V}^{\**}\F^\ltimes)_T$ agrees
	with the family $\F_T$ from \cite[Prop. 6.46]{BP21}.
\end{remark}

\begin{remark}\label{PSEUSTI REM}
	Specifying \eqref{PIS_DEF}
	for the stick tree $\eta \in \Omega$,
	one has that 
	$
	\left(\boldsymbol{V}^{\**} \mathcal{F}^{\ltimes}\right)_{\eta}
	$
	is the family of all subgroups of
	$G^{op} \simeq G^{op} \times \mathsf{\mathsf{Aut}}(\eta)$,
	while
	$
	\left(\mathsf{lr}^{\**} \mathcal{F}\right)_{\eta}
	$
	is identified with the family 
	$\F_1$.
	Hence, in particular, if $\F$ is a pseudo indexing system then
	$\F_1$ must be the family of all subgroups of
	$G^{op} \simeq G^{op} \times \Sigma_1$.
\end{remark}

We can now prove the main result of this section, which is a refinement of Theorem
\ref{THMII}.

\begin{proposition}\label{SIGMAG_COF PROP}
	Suppose $(\V, \otimes)$ is as in Theorem \ref{THMI},
	and let $\F$ be a pseudo indexing system.
	
	If $F \colon \O \to \P$ is an $\F$-cofibration in $\Op^G_{\mathfrak C}(\V)$ and $\O$ is $\Sigma_\F$-cofibrant,
	then $F$ is a $\Sigma_\F$-cofibration.
\end{proposition}

\begin{remark}\label{STRTED REM}
	Theorem \ref{THMII} and Proposition \ref{SIGMAG_COF PROP}
	are the colored analogues of 
	\cite[Lemma 6.64]{BP21} and \cite[Rem. 6.70]{BP21},
	and it would be straightforward but tedious to adapt the proofs therein.
	However, as we have established 
	Proposition \ref{SIGMAWRGF PROP} for all groupoids
	(while in \cite{BP21}
	we only had access to \cite[Prop. 6.25]{BP21},
	which covers only groups),
	here we will be able to provide a significantly simpler proof.
\end{remark}

In the following, we use of the categories
$\Omega^{a}_{\mathfrak{C}}[k] 
\subset 
\Omega^{a}_{\mathfrak{C}}$
of alternating trees discussed in \S \ref{OPC_MS_SEC}.
Moreover, we write
$\boldsymbol{V}^a_{\mathfrak{C}}
\colon G^{op} \ltimes \Omega^{a}_{\mathfrak{C}}[k] 
\to 
\Sigma \wr 
\left((G^{op} \ltimes \Sigma_{\mathfrak{C}})^{\amalg 2}\right)$ 
for the natural variant of the vertex functor 
where the summands 
$(G^{op} \ltimes \Sigma_{\mathfrak{C}})^{\amalg 2} \simeq
(G^{op} \ltimes \Sigma_{\mathfrak{C}})^{\amalg \{a,i\}}$
separate active and inert vertices.

\begin{proof}[Proof of Proposition \ref{SIGMAG_COF PROP}]
	It is enough to show that,
	if $\O \to \O[u]$ is a free operad extension as in \eqref{OURE EQ},
	and for which
	$\O$ is $\Sigma_\F$-cofibrant and $u \colon X \to Y$ is a $\F$-cofibration in $\Sym^G_{\mathfrak C}(\V)$,
	then each of the filtration pieces 
	$\O_{k-1} \to \O_k$ in \eqref{OUFILRE EQ}
	are $\F$-cofibrations in $\Sym^G_{\mathfrak C}(\V)$.
	
	We now consider the following commutative diagrams, 
	where the composites of the top squares
	coincide.
	To ease notation, we have written
	$G^o$ as $G^{op}$
	and 
	$\Sigma \wr 
	\left((G^{op} \ltimes \Sigma_{\mathfrak{C}})^{\amalg 2}\right)$
	as
	$\Sigma \wr 
	(G^{o} \ltimes \Sigma_{\mathfrak{C}})^{\amalg 2}$, while $\nabla$
	denotes the fold map.
	\begin{equation}\label{FCA_SQ_EQ}
	\begin{tikzcd}[column sep = small]
	G^{o} \ltimes \Omega_{\mathfrak C}^{a}[k]
	\arrow[r] \arrow[d, "\pi_{\mathfrak C}^a"']
	&
	G^{o} \ltimes \Omega^{a}[k] \arrow[r, "\boldsymbol{V}^a"]
	\arrow[d, "\pi^a"]
	&
	\Sigma \wr (G^{o} \ltimes \Sigma)^{\amalg 2}
	\arrow[d, "\Sigma \wr \nabla"]
	& 
	G^{o} \ltimes \Omega_{\mathfrak C}^{a}[k] \arrow[r, "\boldsymbol{V}_{\mathfrak C}^{a}"]
	\arrow[d, "\pi_{\mathfrak C}^a"']
	&
	\Sigma \wr ( G^{o} \ltimes \Sigma_{\mathfrak C} )^{\amalg 2}
	\arrow[d, "\Sigma \wr \nabla_{\mathfrak C}"] \arrow[r]
	&
	\Sigma \wr ( G^{o} \ltimes \Sigma )^{\amalg 2}
	\arrow[d, "\Sigma \wr \nabla"]
	\\
	G^{o} \ltimes \Omega_{\mathfrak C}^{0} \arrow[r, "\pi_{\mathfrak C}"] \arrow{d}[swap]{\mathsf{lr}_{\mathfrak C}}
	&
	G^{o} \ltimes \Omega^{0} \arrow[r, "\boldsymbol{V}"] \arrow[d, "\mathsf{lr}"]
	&
	\Sigma \wr \left( G^{o} \ltimes \Sigma \right)
	& 
	G^{o} \ltimes \Omega_{\mathfrak C}^{0} \arrow[r, "\boldsymbol{V}_{\mathfrak C}"]
	&
	\Sigma \wr \left( G^{o} \ltimes \Sigma_{\mathfrak C} \right) \arrow[r, "\Sigma \wr \pi_{\mathfrak C}"]
	&
	\Sigma \wr \left( G^{o} \ltimes \Sigma \right)
	\\
	G^{o} \ltimes \Sigma_{\mathfrak C} \arrow[r, "\pi_{\mathfrak C}"]
	&
	G^{o} \times \Sigma
	\end{tikzcd}
	\end{equation}
	
	Recalling the family 
	$\mathcal{F}_{\mathfrak{C}}$ on 
	$G^{op} \ltimes \Sigma_{\mathfrak{C}}$
	in Definition \ref{GSFAM_DEF},
	we now have a sequence of identifications 
	(of families in $G^{op} \ltimes \Omega^a_{\mathfrak{C}}[k]$)
	\begin{equation}\label{VSTREWRT EQ}
	\left(\boldsymbol{V}_{\mathfrak C}^{a}\right)^{\**} \left( (\nabla_{\mathfrak C}^{\**} \F_{\mathfrak C})^\ltimes \right)
	=
	\left(\boldsymbol{V}_{\mathfrak C}^{a}\right)^{\**} \left( (\nabla_{\mathfrak C}^{\**} \pi^{\**}_{\mathfrak{C}} \F)^\ltimes \right) 
	=
	\left(\boldsymbol{V}_{\mathfrak C}^{a}\right)^{\**}  (\Sigma \wr \nabla_{\mathfrak C})^{\**} (\Sigma \wr \pi_{\mathfrak C})^{\**} \F^\ltimes
	=
	\left( \pi_{\mathfrak C}^a \right)^{\**} \pi_{\mathfrak C}^{\**} \boldsymbol{V}^{\**} \F^{\ltimes}
	\end{equation}
	where the first step is the definition 
	$\F_{\mathfrak{C}} = \pi^{\**}_{\mathfrak{C}} \mathcal{F}$,
	the second step follows from \eqref{PHIFLTIMES_EQ},
	and the third step follows since the composite top sections in 
	\eqref{FCA_SQ_EQ} coincide.
	Moreover, there are identifications
	\begin{equation}\label{LRSTR EQ}
	\left(\pi_{\mathfrak C}^a\right)^{\**} \mathsf{lr}_{\mathfrak C}^{\**} \F_{\mathfrak C}
	=
	\left(\pi_{\mathfrak C}^a\right)^{\**} \mathsf{lr}_{\mathfrak C}^{\**} \pi_{\mathfrak{C}}^{\**} \F
	=
	\left(\pi_{\mathfrak C}^a\right)^{\**} \pi_{\mathfrak C}^{\**} \mathsf{lr}^{\**} \F
	\end{equation}
	where the first step again uses the definition of $\F_{\mathfrak{C}}$
	and the second step uses the commutativity of the left bottom square in 
	\eqref{FCA_SQ_EQ}.
	
	It now follows from \eqref{VSTREWRT EQ},\eqref{LRSTR EQ} that,
	if $\F$ is a pseudo indexing system,
	i.e. if 
	$\boldsymbol{V}^{\**} \F 
	\subseteq 
	\mathsf{lr}^{\**} \F$,
	then
	\begin{equation}\label{THEPT EQ}
	\left(\boldsymbol{V}_{\mathfrak C}^{a}\right)^{\**} \left( (\nabla_{\mathfrak C}^{\**} \F_{\mathfrak C})^\ltimes \right)
	\subseteq
	\left(\pi_{\mathfrak C}^a\right)^{\**} \mathsf{lr}_{\mathfrak C}^{\**} \F_{\mathfrak C}.
	\end{equation}

	Abbreviating 
	$\mathsf{lr}_{\mathfrak{C}}^a
	=
	\mathsf{lr}_{\mathfrak{C}} \pi^a_{\mathfrak{C}}$
	for the left vertical composite in \eqref{FCA_SQ_EQ},
	applying Proposition \ref{EQQUILADJ PROP}
	to \eqref{THEPT EQ}
	yields that 
	$
	\left(\mathsf{lr}^{a,op}_{\mathfrak{C}}\right)_!
	\ \colon \
	\V^{G \ltimes \Omega_{\mathfrak C}^{a}[k]^{op}}
	\to
	\V^{G \ltimes \Sigma_{\mathfrak C}^{op}}
	$
	sends 
	$\left(\boldsymbol{V}_{\mathfrak C}^{a}\right)^{\**}\left( (\nabla_{\mathfrak C}^{\**} \F_{\mathfrak C})^\ltimes \right)$-cofibrations
	to
	$\F_{\mathfrak C}$-cofibrations.
	
	Next, notice that the map
	$(\O,u) = (\varnothing \to \O) \amalg (X \xrightarrow{u} Y)$
	is a $\nabla_{\mathfrak C}^{\**}\F_{\mathfrak C}$-cofibration
	in $\V^{(G \ltimes \Sigma_{\mathfrak C}^{op})^{\amalg 2}}$
	(this simply restates the assumption that
	$\varnothing \to \O$, $X \xrightarrow{u} Y$ 
	are $\F_{\mathfrak{C}}$-cofibrations
	in $\V^{G \ltimes \Sigma_{\mathfrak C}^{op}} = \mathsf{Sym}_{\mathfrak{C}}^G(\V)$).
	And, since the functor $n_k^{\O,X,Y} \colon G \ltimes \Omega_{\mathfrak C}^{a}[k]^{op} \to \V$ from \eqref{NKOXY EQ}
	can be described by
	\[
	n_k^{\O,X,Y} = \left( \boldsymbol{V}_{\mathfrak C}^{a}\right)^{\**}
	\left( \left(\O,u\right)^{\square} \right),
	\]
	Propositions \ref{SIGMAWRGF PROP}(i) and \ref{EQQUILADJ PROP}
	imply that $n_k^{\O,X,Y}$ is a 
	$(\boldsymbol{V}_{\mathfrak C}^{a})^{\**}\left( (\nabla_{\mathfrak C}^{\**} \F_{\mathfrak C})^\ltimes \right)$-cofibration.
	
	But it now follows that 
	$\left(\mathsf{lr}^{a,op}_{\mathfrak{C}}\right)_! n_k^{\O,X,Y}$
	is a $\F_{\mathfrak{C}}$-cofibration
	in $\V^{G \ltimes \Sigma_{\mathfrak C}^{op}} = \mathsf{Sym}_{\mathfrak{C}}^G(\V)$
	and thus,
	using the pushout squares \eqref{OUPUSHRE EQ},
	so is $\O_{k-1} \to \O_k$,
	finishing the proof.
\end{proof}

In the following,
we write $\eta_{\mathfrak C}$ for the initial 
object of $\Op^G_{\mathfrak C}(\V)$. 
This notation is motivated by the underlying identification
$\eta_{\mathfrak C} \simeq 
\coprod_{\mathfrak c \in \mathfrak C} \Sigma_{\mathfrak C}[\eta_{\mathfrak c}]$
in $\mathsf{Sym}^G_{\mathfrak{C}}(\V)$,
where $\eta_{\mathfrak c}$ is as in Remark \ref{ETACNOT REM}.

\begin{proof}[Proof of Theorem \ref{THMII}]
	Remark \ref{PSEUSTI REM}
	guarantees that the initial operad 
	$\eta_{\mathfrak{C}}$
	in $\mathsf{Op}^G_{\mathfrak{C}}(\V)$
	is $\Sigma_{\F}$-cofibrant for any
	pseudo indexing system $\F$.
	Hence, by first applying Proposition \ref{SIGMAG_COF PROP}
	to maps of the form $\eta_{\mathfrak{C}} \to \O$,
	one has that all $\F$-cofibrant operads
	$\O \in \mathsf{Op}^G_{\mathfrak{C}}(\V)$
	are $\Sigma_{\F}$-cofibrant,
	so that Theorem \ref{THMII} becomes a particular instance of 
	Proposition \ref{SIGMAG_COF PROP}.
\end{proof}


\appendix

\section{The monad for free colored operads}
\label{MONAD_APDX}

This appendix is fairly technical.
Its goal is to complete Definition \ref{FREEOP DEF}
by fully describing the monad structure 
on the fibered free operad monad $\mathbb{F}$
of \eqref{FREEOP_EQ},
and to use this to prove some necessary technical results,
most notably the key filtration result in Lemma \ref{OURE LEM}.

Our approach builds off our previous work in \cite{BP21},
and is motivated by the fact that the left Kan extension in \eqref{FREEOP_EQ} makes the monad structure somewhat awkward to describe and work with directly.
As such, our strategy is to note that there is an adjunction 
\begin{equation}\label{SPANSYMADJ EQ}
\mathsf{Lan} \colon
\mathsf{WSpan}^l(\Sigma_{\bullet}^{op},\mathcal{V}) 
\rightleftarrows
\mathsf{Sym}_{\bullet}(\mathcal{V})
\colon \upsilon
\end{equation}
that identifies $\mathsf{Sym}_{\bullet}(\V)$ 
as a reflexive subcategory of a larger category 
$\mathsf{WSpan}^l(\Sigma_{\bullet}^{op},\mathcal{V})$
of spans (Definition \ref{WSPANCAT DEF}),
with the left Kan extension being the reflection.
We then build a monad $N$ on the larger category 
$\mathsf{WSpan}^l(\Sigma_{\bullet}^{op},\mathcal{V})$,
and show that this monad can be transferred to 
$\mathsf{Sym}_{\bullet}(\mathcal{V})$.

In \S \ref{CSTRINGS_SEC}, \S \ref{WRACONST SEC}
we adapt (and improve on) 
some constructions in \cite{BP21}
to the colored context.
\S \ref{NONEQMON SEC} then describes the monads on
$\mathsf{Sym}_{\bullet}(\V)$, $\mathsf{WSpan}^l(\Sigma_{\bullet}^{op},\mathcal{V})$.
\S \ref{EQMON_SEC} is dedicated to proving Lemma \ref{OURE LEM}.
Lastly,
\S \ref{INJCOLCH AP}
proves some quick results that 
will be needed in the sequel \cite{BP_ACOP}.

\subsection{Colored strings}
\label{CSTRINGS_SEC}

We now recall some key notions in \cite{BP21}
regarding strings of trees.

Firstly, and as in \S \ref{EQCOSYMOP SEC},
we write 
$\Omega_{\mathfrak{C}}$
for the category of 
$\mathfrak{C}$-colored trees
$\vect{T} = (T,\mathfrak{c}\colon \boldsymbol{E}(T) \to \mathfrak{C})$
and color preserving maps.

We will make use of certain special types of maps in $\Omega_{\mathfrak{C}}$, all of which are defined in terms of the underlying map of trees.
A map $f \colon \vect{T} \to \vect{S}$
is called:
\emph{planar} is the underlying map of trees is planar;
\emph{tall} if it sends leaves to leaves and the root to the root;
an \emph{outer face}
if, for any factorization 
$f = f' t$ with $t$ a tall map, one has that $t$ is an isomorphism. 

Any map $\vect{T} \xrightarrow{f} \vect{S}$ in $\Omega_{\mathfrak{C}}$
the has a factorization
$\vect{T} \xrightarrow{f^t} \vect{R} \xrightarrow{f^o} \vect{S}$,
unique up to unique isomorphism,
with $f^t$ a tall map and $f^o$ an outer map
\cite[Prop. 3.36]{BP21}.
Moreover, this factorization is strictly unique if $f,f^o$ are required to be planar.

In the following,
we extend the notation $T_v$ for the corolla associated to a vertex 
$v \in \boldsymbol{V}(T)$.

\begin{notation}\label{SVNOT NOT}
	Given a planar map $\vect{T} \to \vect{S}$
	and $v \in \boldsymbol{V}(T)$,
	we write
	$\vect{T}_v \to \vect{S}_v \to \vect{S}$
	for the ``tall map followed by outer face''
	factorization of the composite
	$\vect{T}_v \to \vect{T} \to \vect{S}$.
\end{notation}

\begin{example}
	Consider the planar map $f \colon T \to S$ on the left below
	sending each edge of $T$ to the edge of $S$ with the same name.
	For each of the four vertices $v_1,v_2,v_3,v_4$ of $T$
	(these are ordered according to the planarization of $T$, 
	following the convention in \cite[\S 3.1]{BP21})
	we present the corresponding planar outer face
	$S_{v_i}$ on the right.
	
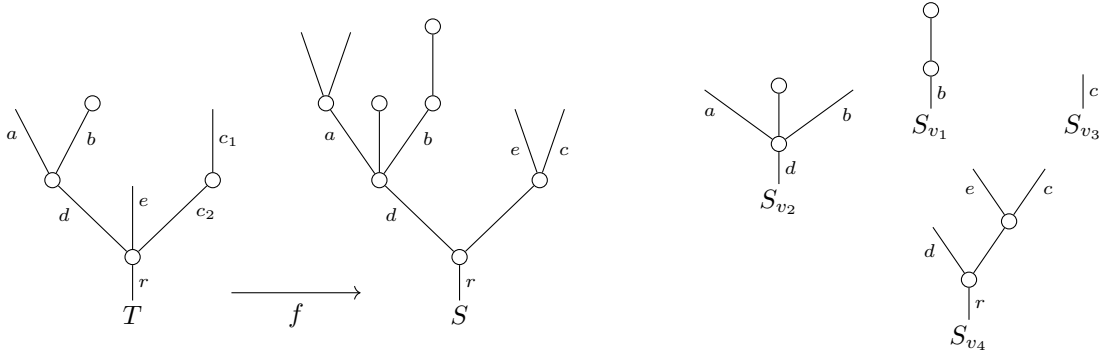
\begin{equation}
	\begin{tikzpicture}[auto,grow=up, level distance = 2.2em,
	every node/.style={font=\scriptsize,inner sep = 2pt}]%
	\tikzstyle{level 2}=[sibling distance=3em]%
	\node at (0,0) [font = \normalsize] {$T$}%
	child{node [dummy] {}%
		child[level distance = 2.9em]{node [dummy] {}%
			child{node {}%
				edge from parent node [swap] {$c_1$}}%
			edge from parent node [swap,near end] {$c_2$}}%
		child[level distance = 2.9em]{node {}%
			edge from parent node [swap,near end] {$e$}}%
		child[level distance = 2.9em]{node [dummy] {}%
			child[level distance = 2.9em]{node [dummy] {}%
				edge from parent node [swap, near end] {$b$}}%
			child[level distance = 2.9em]{node {}%
				edge from parent node [near end] {$a$}}%
			edge from parent node [near end] {$d$}}%
		edge from parent node [swap] {$r$}};%
	\node at (4.3,0) [font = \normalsize] {$S$}%
	child{node [dummy] {}%
		child[level distance = 2.9em, sibling distance = 6em]{node [dummy] {}%
			child[sibling distance = 2em]{node {}%
				edge from parent node [swap] {$c$}}%
			child[level distance = 2.9em, sibling distance = 2em]{node {}%
				edge from parent node {$e$}}%
			edge from parent node [swap,near end] {}}%
		child[level distance = 2.9em, sibling distance = 6em]{node [dummy] {}%
			child[level distance = 2.9em, sibling distance = 2em]{node [dummy] {}%
				child[level distance = 2.9em]{node [dummy] {}%
					edge from parent node [swap] {}}%
				edge from parent node [swap, near end] {$b$}}%
			child[level distance = 2.9em, sibling distance = 2em]{node [dummy] {}%
				edge from parent node {}}%
			child[level distance = 2.9em, sibling distance = 2em]{node [dummy] {}%
				child[level distance = 2.9em]{node {}%
					edge from parent node [swap] {}}%
				child[level distance = 2.9em]{node {}%
					edge from parent node [swap] {}}%
				edge from parent node [near end] {$a$}}%
			edge from parent node [near end] {$d$}}%
		edge from parent node [swap] {$r$}};%
	\node at (10.5,2.5) [font = \normalsize] {$S_{v_1}$}%
	child{node [dummy] {}%
		child{node [dummy] {}%
			edge from parent node [swap] {}}%
		edge from parent node [swap] {$b$}};%
	\node at (8.5,1.5) [font = \normalsize] {$S_{v_2}$}%
	child{node [dummy] {}%
		child{node {}%
			edge from parent node [swap, near end] {$b$}}%
		child{node [dummy]{}%
			edge from parent node {}}%
		child{node {}%
			edge from parent node [near end] {$a$}}%
		edge from parent node [swap] {$d$}};%
	\node at (11,-0.3) [font = \normalsize] {$S_{v_4}$}%
	child{node [dummy] {}%
		child{node [dummy]{}%
			child{node {}%
				edge from parent node [swap, near end] {$c$}}%
			child{node {}%
				edge from parent node [near end] {$e$}}%
			edge from parent node {}}%
		child{node {}%
			edge from parent node [near end] {$d$}}%
		edge from parent node [swap] {$r$}};%
	\node at (12.5,2.5) [font = \normalsize] {$S_{v_3}$}%
	child{node {}%
		edge from parent node [swap] {$c$}};%
	\draw[<-] (3,0.3) -- (1.3,0.3) 
	node[midway,inner sep=4pt,font=\normalsize]{$f$}; %
	\end{tikzpicture}%
	\end{equation}%
\end{example}

We next recall some notation concerning 
the $\Sigma \wr (-)$ construction in Notation \ref{SIGWR NOT}.
For any category $\mathcal{C}$,
there is a natural ``unary tuple functor''
$\delta^{0} \colon \mathcal{C} \to \Sigma \wr \mathcal{C}$
sending
$c \in \mathcal{C}$ to $(c) \in \Sigma \wr \mathcal{C}$,
as well as a natural ``concatenation functor''
$\sigma^0 \colon \Sigma \wr \Sigma \wr \mathcal{C} 
\to \Sigma \wr \mathcal{C}$ given by 
\[
\left(
(c_{1,1},\cdots,c_{1,n_1}),
(c_{2,1},\cdots,c_{2,n_2}),
\cdots,
(c_{k,1},\cdots,c_{k,n_k})
\right)
\mapsto
\left(
c_{1,1},\cdots,c_{1,n_1},
c_{2,1},\cdots,c_{2,n_2},
\cdots,
c_{k,1},\cdots,c_{k,n_k}
\right)
\]
More generally, these functors are part of 
an augmented cosimplicial object
$\Sigma^{\wr n+1} \wr \C, n\geq -1$ in $\mathsf{Cat}$,
meaning that one has maps
\[
\Sigma^{\wr n+1} \wr \mathcal{C} 
\xrightarrow{\sigma^i},
\Sigma^{\wr n} \wr \mathcal{C},
\quad
0 \leq i \leq n - 1
\qquad \qquad
\Sigma^{\wr n+1} \wr \mathcal{C} 
\xrightarrow{\delta^j},
\Sigma^{\wr n+2} \wr \mathcal{C},
\quad
0 \leq j \leq n+1
\]
satisfying the cosimplicial identities. 
Explicitly, $\sigma^i$ concatenates the $(i+1)$-th and $(i+2)$-th $\Sigma$ coordinates
and $\delta^j$ inserts a unary $\Sigma$ coordinate in the $(j+1)$-th position.

\begin{definition}
	Let $\mathfrak{C}$ be a set of colors.
	For $n \geq 0$, the category $\Omega_{\mathfrak{C}}^n$ of \textit{$n$-strings} has as objects strings $\vect{T}_0 \to \vect{T}_1 \to \cdots \to \vect{T}_n$ of maps that are planar and tall, and arrows given by tuples 
	$\vect{T}_i \xrightarrow{\simeq} \vect{T}'_i$
	of compatible isomorphisms.
	In addition, we set $\Omega^{-1}_{\mathfrak{C}} = \Sigma_{\mathfrak{C}}$.
\end{definition}

\begin{remark}\label{ADDCOROL REM}
	By definition of tall planar map, 
	all trees $T_i$ in a string have the same leaf-root
	$\mathsf{lr}(\vect{T}_i)$. The convention $\Omega^{-1}_{\mathfrak{C}} = \Sigma_{\mathfrak{C}}$ 
	is motivated by the fact that one can 
	canonically extend an $n$-string as
	\[\mathsf{lr}(\vect{T}_0)=\vect{T}_{-1} \to \vect{T}_0 \to \vect{T}_1 \to \cdots \to \vect{T}_n.\]
\end{remark}

The key functors in \cite[\S 3.4]{BP21} now extend to the strings
$\Omega_{\mathfrak{C}}^n$.
Firstly, one has simplicial operators
\[
d_i \colon \Omega_{\mathfrak{C}}^n \to \Omega_{\mathfrak{C}}^{n-1},
\quad 0 \leq i \leq n;
\qquad \qquad
s_j \colon \Omega_{\mathfrak{C}}^{n} \to \Omega_{\mathfrak{C}}^{n+1},
\quad -1 \leq j \leq n
\]
which remove (resp. repeat) the $i$-th (resp. $j$-th) tree in the string (where the boundary cases must be interpreted in light of 
Remark \ref{ADDCOROL REM}).
Secondly, there are vertex operators
\[
\Omega^{n}_{\mathfrak{C}}
\xrightarrow{\boldsymbol{V}^0}
\Sigma \wr \Omega^{n}_{\mathfrak{C}}
\qquad \qquad
(\vect{T}_0 \to \vect{T}_1 \to \cdots \to \vect{T}_n)
\mapsto
(\vect{T}_{1,v} \to \cdots \to \vect{T}_{n,v})_{v \in \boldsymbol{V}(T_0)}
\]
where we note that the indexing set 
$\boldsymbol{V}(T_0)$ of the tuple
is ordered according to the planarization of $T_0$.
Moreover, one iteratively defines higher order vertex functors
by setting 
$\boldsymbol{V}^{k}\colon \Omega^{n}_{\mathfrak{C}}
\to \Sigma \wr \Omega^{n-k-1}_{\mathfrak{C}}$
to be the composite
\begin{equation}\label{VKDEF EQ}
\Omega^{n}_{\mathfrak{C}}
\xrightarrow{\boldsymbol{V}}
\Sigma \wr\Omega^{n-1}_{\mathfrak{C}}
\xrightarrow{\Sigma \wr \boldsymbol{V}^{k-1}}
\Sigma^{\wr 2} \wr \Omega^{n-k-1}_{\mathfrak{C}}
\xrightarrow{\sigma^0}
\Sigma \wr \Omega^{n-k-1}_{\mathfrak{C}}.
\end{equation}
One has an identification
$\boldsymbol{V}^k(\vect{T}_0 \to \cdots \to \vect{T}_n)
\simeq (\vect{T}_{k+1,v} \to \cdots \to \vect{T}_{n,v})_{v \in \boldsymbol{V}(T_k)}$, though some care is needed,
as the order on $\boldsymbol{V}(T_k)$ in this formula
does not depend only on the planarization of $T_k$.
Rather, cf. \cite[(3.99)]{BP21},
one has natural maps 
$\boldsymbol{V}(T_k) \to \boldsymbol{V}(T_{k-1}) 
\to \cdots
\to \boldsymbol{V}(T_{0})$
and, for $v,w \in \boldsymbol{V}(T_k)$,
the condition $v<w$
is determined by the lowest $i$
such that $v,w$ have distinct images in 
$\boldsymbol{V}(T_i)$.

Lastly, for a map of colors 
$\varphi \colon \mathfrak{C} \to \mathfrak{D}$,
there are change of color functors
$\varphi \colon \Omega^n_{\mathfrak{C}} \to \Omega^n_{\mathfrak{D}}$
defined by
$(\vect{T}_0 \to \cdots \to \vect{T}_n)
\mapsto 
(\varphi\vect{T}_0 \to \cdots \to \varphi\vect{T}_n)$,
where $\varphi \vect{T}_i$
applies $\varphi$ to the colors 
(cf. Definition \ref{COLFOR DEF}).

These operators satisfy a number of compatibilities (cf. \cite[Prop. 3.102]{BP21}). Firstly, the $d_i$, $s_j$ operators satisfy the  simplicial identities, 
and the $\boldsymbol{V}^k$ operators are ``additive'' in the sense that
the composite below is $\boldsymbol{V}^{k+l+1}$.
\begin{equation}\label{VKADD EQ}
\Omega^{n}_{\mathfrak{C}} \xrightarrow{\boldsymbol{V}^l} 
\Sigma \wr \Omega^{n-l-1}_{\mathfrak{C}} \xrightarrow{\Sigma \wr \boldsymbol{V}^k}
\Sigma^{\wr 2} \wr \Omega^{n-k-l-2}_{\mathfrak{C}} \xrightarrow{\sigma^0}
\Sigma \wr \Omega^{n-k-l-2}_{\mathfrak{C}}.
\end{equation}

The next results list the compatibilities between $d_i$, $s_j$ and the $\boldsymbol{V}^k$ operators.
We note that the natural isomorphisms $\pi_{i,k}$
are needed to account for different orderings on 
$\boldsymbol{V}(T_k)$,
cf. the comment following \eqref{VKDEF EQ}.

\begin{proposition}\label{CATDIAG PROP}
	One has the following diagrams in the $2$-category
	$\mathsf{Cat}$.
	\begin{enumerate}[label = (\roman*)]
		\item
		For $0\leq i < k \leq n$ there are $2$-isomorphisms  $\pi_{i,k}$,
		and for $-1 \leq j \leq k \leq n$ there are commutative diagrams
		\begin{equation}
		\begin{tikzcd}[row sep = 2pt, column sep = 35pt]
		\Omega_{\mathfrak{C}}^n
		\arrow{dr}[swap,name=U]{}{\boldsymbol{V}^k} \arrow{dd}[swap]{d_i} &
		&
		\Omega_{\mathfrak{C}}^n
		\arrow{dr}{\boldsymbol{V}^k} \arrow{dd}[swap]{s_j} &
		\\
		& \Sigma \wr \Omega_{\mathfrak{C}}^{n-k-1}
		&
		& \Sigma \wr \Omega_{\mathfrak{C}}^{n-k-1}
		\\
		|[alias=V]|
		\Omega_{\mathfrak{C}}^{n-1} \arrow{ur}[swap]{\boldsymbol{V}^{k-1}} &
		&
		\Omega_{\mathfrak{C}}^{n+1} \arrow{ur}[swap]{\boldsymbol{V}^{k+1}} &
		\arrow[Leftrightarrow, from=V, to=U,shorten >=0.15cm,shorten <=0.15cm
		,swap, near end, "\pi_{i,k}"
		]
		\end{tikzcd}
		\end{equation}
		\item
		For $-1 \leq k < i \leq n$ and for $-1 \leq k \leq j \leq n$
		there are commutative diagrams
		\begin{equation}
		\begin{tikzcd}[row sep = 8pt, column sep = 35pt]
		\Omega^n_{\mathfrak{C}}
		\arrow{r}[swap,name=U]{}{\boldsymbol{V}^k} \arrow{dd}[swap]{d_i} &
		\Sigma \wr \Omega^{n-k-1}_{\mathfrak{C}} \ar{dd}{d_{i-k-1}}
		&
		\Omega^n_{\mathfrak{C}}
		\arrow{r}{\boldsymbol{V}^k} \arrow{dd}[swap]{s_j} &
		\Sigma \wr \Omega^{n-k-1}_{\mathfrak{C}} \ar{dd}{s_{j-k-1}}
		\\
		\\
		|[alias=V]|
		\Omega^{n-1}_{\mathfrak{C}} \arrow{r}[swap]{\boldsymbol{V}^{k}} &
		\Sigma \wr \Omega^{n-k-2}_{\mathfrak{C}}
		&
		\Omega^{n+1}_{\mathfrak{C}} \arrow{r}[swap]{\boldsymbol{V}^{k}} &
		\Sigma \wr \Omega^{n-k}_{\mathfrak{C}}
		\end{tikzcd}
		\end{equation}
		Furthermore, these diagrams are pullback squares in $\mathsf{Cat}$.
		\item 
		all $d_i$, $s_j$, $\boldsymbol{V}^k$, and $\pi_{i,k}$
		are natural in $\mathfrak{C}$, i.e. for each map of colors
		$\varphi \colon \mathfrak{C} \to \mathfrak{D}$ one has commutative diagrams
		(for the prism, commutativity means that
		$\varphi \pi = \pi \varphi$)
		\[
		\begin{tikzcd}[column sep = 10pt, row sep = 4pt]
		\Omega^n_{\mathfrak{C}} \ar{r}{d_i} \ar{dd}[swap]{\varphi} &
		\Omega^{n-1}_{\mathfrak{C}} \ar{dd}{\varphi}
		&
		\Omega^n_{\mathfrak{C}} \ar{r}{s_j} \ar{dd}[swap]{\varphi} &
		\Omega^{n+1}_{\mathfrak{C}} \ar{dd}{\varphi}
		&
		\Omega^n_{\mathfrak{C}} \ar{r}{\boldsymbol{V}^k} \ar{dd}[swap]{\varphi} &
		\Sigma \wr \Omega^{n-k-1}_{\mathfrak{C}} \ar{dd}{\varphi}
		&
		\Omega^n_{\mathfrak{C}}
		\ar{rrrrr}[name=toE]{\boldsymbol{V}^k} \ar{rd}[swap]{d_i} \ar{dd}[swap]{\varphi}
		&&&
		&&
		\Sigma \wr \Omega^{n-k-1}_{\mathfrak{C}} \ar{dd}{\varphi}
		\\
		&
		&
		&
		&
		&
		&
		&
		|[alias=DBE]|
		\Omega^{n-1}_{\mathfrak{C}} \ar{rrrru}[swap]{\boldsymbol{V}^{k-1}}
		\\
		\Omega^n_{\mathfrak{D}} \ar{r}{d_i} &
		\Omega^{n-1}_{\mathfrak{D}}
		&
		\Omega^n_{\mathfrak{D}} \ar{r}{s_j} &
		\Omega^n_{\mathfrak{D}}
		&
		\Omega^n_{\mathfrak{D}} \ar{r}{\boldsymbol{V}^k} &
		\Sigma \wr \Omega^{n-k-1}_{\mathfrak{D}}
		&
		\Omega^n_{\mathfrak{D}} \ar{rrrrr}[name=toB]{\boldsymbol{V}^k} \ar{rd}[swap]{d_i}
		&&&
		&&
		\Sigma \wr \Omega^{n-k-1}_{\mathfrak{D}}
		\\
		&
		&
		&
		&
		&
		&
		&
		|[alias=D]| \Omega^{n-1}_{\mathfrak{D}} \ar{rrrru}[swap]{\boldsymbol{V}^{k-1}}
		\arrow[Leftrightarrow, from=DBE, to=toE, shorten <=0.1cm,shorten >=0.1cm
		,swap,near end,"\pi"
		]
		\arrow[Leftrightarrow, from=D, to=toB, shorten <=0.1cm,shorten >=0.1cm,swap,near end,"\pi"]
		\arrow[from=DBE, to=D, crossing over, near start, swap, "\varphi"]
		\end{tikzcd}
		\]
	\end{enumerate}
\end{proposition}

The following lists the compatibilities between the $\pi_{i,k}$ isomorphisms, 
which are extensions of the additivity of $\boldsymbol{V}^k$ in \eqref{VKADD EQ} and of the simplicial identities between the $d_i$, $s_j$ operators.

\begin{proposition}\label{CATDIAG2 PROP}
	In each item, the two composite natural transformations coincide.
	\begin{itemize}
		\item[(IT1)]
		For $0 \leq i < k $ and $-1 \leq l \leq n-k-1$
		\begin{equation}
		\begin{tikzcd}[row sep = 12pt, column sep = 18pt]
		|[alias=V]|
		\Omega_{\mathfrak{C}}^{n} \ar{r}{\boldsymbol{V}^{k}}[swap,name=UU]{} \arrow{d}[swap]{d_i}&
		\Sigma \wr \Omega_{\mathfrak{C}}^{n-k-1} \ar{r}{\boldsymbol{V}^l} &
		\Sigma^{\wr 2} \wr \Omega^{n-k-l-2}_{\mathfrak{C}} \ar{r}{\sigma^0} &
		\Sigma \wr \Omega^{n-k-l-2}_{\mathfrak{C}}
		&
		\Omega^{n}_{\mathfrak{C}} \ar{r}{\boldsymbol{V}^{k+l+1}}[swap,name=UUU]{} \arrow{d}[swap]{d_i}&
		\Sigma \wr \Omega^{n-k-l-2}_{\mathfrak{C}} &
		\\
		|[alias=VV]|
		\Omega^{n-1}_{\mathfrak{C}} \arrow{ur}[swap]{\boldsymbol{V}^{k-1}} & & &
		&
		|[alias=VVV]|
		\Omega^{n-1}_{\mathfrak{C}} \arrow{ur}[swap]{\boldsymbol{V}^{k+l}} &
		\arrow[Leftrightarrow, from=VV, to=UU,shorten >=0.05cm,shorten <=0.05cm
		,swap,near end,"\pi"
		]
		\arrow[Leftrightarrow, from=VVV, to=UUU,shorten >=0.05cm,shorten <=0.05cm
		,swap,near end,"\pi"
		]
		\end{tikzcd}
		\end{equation}
		
		\item[(IT2)]
		For $-1 \leq k < i < k + l + 1 \leq n$
		\begin{equation}
		\begin{tikzcd}[row sep = 12pt, column sep = 18pt]
		\Omega^n_{\mathfrak{C}} \ar{r}{\boldsymbol{V}^k} \ar{d}[swap]{d_i} &
		|[alias=V]|
		\Sigma \wr \Omega^{n-k-1}_{\mathfrak{C}} \ar{r}{\boldsymbol{V}^{l}}[swap,name=UU]{} \arrow{d}[swap]{d_{i-k-1}} &
		\Sigma^{\wr 2} \wr \Omega^{n-k-l-2}_{\mathfrak{C}} \ar{r}{\sigma^0} &
		\Sigma \wr \Omega^{n-k-l-2}_{\mathfrak{C}}
		&
		\Omega^{n}_{\mathfrak{C}} \ar{r}{\boldsymbol{V}^{k+l+1}}[swap,name=UUU]{} \arrow{d}[swap]{d_i}&
		\Sigma \wr \Omega^{n-k-l-2}_{\mathfrak{C}} &
		\\
		\Omega^{n-1}_{\mathfrak{C}} \ar{r}[swap]{\boldsymbol{V}^k} &
		|[alias=VV]|
		\Sigma \wr \Omega^{n-1}_{\mathfrak{C}} \arrow{ur}[swap]{\boldsymbol{V}^{l-1}} & &
		&
		|[alias=VVV]|
		\Omega^{n-1}_{\mathfrak{C}} \arrow{ur}[swap]{\boldsymbol{V}^{k+l}} &
		\arrow[Leftrightarrow, from=VV, to=UU,shorten >=0.05cm,shorten <=0.05cm
		,swap,near end,"\pi"
		]
		\arrow[Leftrightarrow, from=VVV, to=UUU,shorten >=0.05cm,shorten <=0.05cm
		,swap,near end,"\pi"
		]
		\end{tikzcd}
		\end{equation}
		
		\item[(FF1)]
		For $0 \leq i < i' < k \leq n$
		\begin{equation}
		\begin{tikzcd}[row sep = 12pt, column sep = 35pt]
		\Omega^n_{\mathfrak{C}}
		\arrow{dr}[swap,name=U]{}{\boldsymbol{V}^k} \arrow{d}[swap]{d_{i'}} &
		&
		\Omega^n_{\mathfrak{C}}
		\arrow{dr}[swap,name=UUU]{}{\boldsymbol{V}^k} \arrow{d}[swap]{d_i} &
		\\
		|[alias=V]|
		\Omega^{n-1}_{\mathfrak{C}} \ar{r}[near start,swap]{\boldsymbol{V}^{k-1}}[swap,name=UU]{} \arrow{d}[swap]{d_i}&
		\Sigma \wr \Omega^{n-k-1}_{\mathfrak{C}}
		&
		|[alias=VVV]|
		\Omega^{n-1}_{\mathfrak{C}} \ar{r}[near start, swap]{\boldsymbol{V}^{k-1}}[swap,name=UUUU]{} \ar{d}[swap]{d_{i'-1}} &
		\Sigma \wr \Omega^{n-k-1}_{\mathfrak{C}}
		\\
		|[alias=VV]|
		\Omega^{n-2}_{\mathfrak{C}} \arrow{ur}[swap]{\boldsymbol{V}^{k-2}} &
		&
		|[alias=VVVV]|
		\Omega^{n-2}_{\mathfrak{C}} \arrow{ur}[swap]{\boldsymbol{V}^{k-2}} &
		\arrow[Leftrightarrow, from=V, to=U,shorten >=0.05cm,shorten <=0.05cm
		,swap,near end,"\pi"
		]
		\arrow[Leftrightarrow, from=VV, to=UU,shorten >=0.25cm,shorten <=0.05cm
		,swap,near end,"\pi"
		]
		\arrow[Leftrightarrow, from=VVV, to=UUU,shorten >=0.05cm,shorten <=0.05cm
		,swap,near end,"\pi"
		]
		\arrow[Leftrightarrow, from=VVVV, to=UUUU,shorten >=0.25cm,shorten <=0.05cm
		,swap,near end,"\pi"
		]
		\end{tikzcd}
		\end{equation}
		
		\item[(FF2)]
		For $0 \leq i < k < i' \leq n$
		\begin{equation}
		\begin{tikzcd}[row sep = 12pt, column sep = 35pt]
		\Omega^n_{\mathfrak{C}}
		\arrow{r}[swap,name=U]{}{\boldsymbol{V}^k} \arrow{d}[swap]{d_{i'}} &
		\Sigma \wr \Omega^{n-k-1}_{\mathfrak{C}} \ar{d}{d_{i'-k-1}}
		&
		\Omega^n_{\mathfrak{C}}
		\arrow{dr}[swap,name=UUU]{}{\boldsymbol{V}^k} \arrow{d}[swap]{d_i} &
		\\
		|[alias=V]|
		\Omega^{n-1}_{\mathfrak{C}} \ar{r}{\boldsymbol{V}^{k}}[swap,name=UU]{} \arrow{d}[swap]{d_i}&
		\Sigma \wr \Omega^{n-k-2}_{\mathfrak{C}}
		&
		|[alias=VVV]|
		\Omega^{n-1}_{\mathfrak{C}} \ar{r}[near start, swap]{\boldsymbol{V}^{k-1}}[swap,name=UUUU]{} \ar{d}[swap]{d_{i'-1}} &
		\Sigma \wr \Omega^{n-k-1}_{\mathfrak{C}} \ar{d}{d_{i'-k-1}}
		\\
		|[alias=VV]|
		\Omega^{n-2}_{\mathfrak{C}} \arrow{ur}[swap]{\boldsymbol{V}^{k-1}} &
		&
		|[alias=VVVV]|
		\Omega^{n-2}_{\mathfrak{C}} \ar{r}[swap]{\boldsymbol{V}^{k-1}} &
		\Sigma \wr \Omega^{n-k-2}_{\mathfrak{C}}
		\arrow[Leftrightarrow, from=VV, to=UU,shorten >=0.05cm,shorten <=0.05cm
		,swap,near end,"\pi"
		]
		\arrow[Leftrightarrow, from=VVV, to=UUU,shorten >=0.05cm,shorten <=0.05cm
		,swap,near end,"\pi"
		]
		\end{tikzcd}
		\end{equation}
		
		\item[(DF1)]
		For $0 \leq j+1 < i \leq k \leq n $
		\begin{equation}
		\begin{tikzcd}[row sep = 12pt, column sep = 35pt]
		\Omega^{n}_{\mathfrak{C}}
		\arrow{dr}[swap,name=U]{}{\boldsymbol{V}^{k}} \arrow{d}[swap]{s_j} &
		&
		\Omega^{n}_{\mathfrak{C}}
		\arrow{dr}[swap,name=UUU]{}{\boldsymbol{V}^{k}} \arrow{d}[swap]{d_{i-1}} &
		\\
		|[alias=V]|
		\Omega^{n+1}_{\mathfrak{C}} \ar{r}{\boldsymbol{V}^{k+1}}[swap,name=UU]{} \arrow{d}[swap]{d_i}&
		\Sigma \wr \Omega^{n-k-1}
		&
		|[alias=VVV]|
		\Omega^{n-1}_{\mathfrak{C}} \ar{r}[near start, swap]{\boldsymbol{V}^{k-1}}[swap,name=UUUU]{} \ar{d}[swap]{s_j} &
		\Sigma \wr \Omega^{n-k-1}
		\\
		|[alias=VV]|
		\Omega^{n}_{\mathfrak{C}} \arrow{ur}[swap]{\boldsymbol{V}^{k}} &
		&
		|[alias=VVVV]|
		\Omega^{n}_{\mathfrak{C}} \arrow{ur}[swap]{\boldsymbol{V}^{k}} &
		\arrow[Leftrightarrow, from=VV, to=UU,shorten >=0.25cm,shorten <=0.05cm
		,swap,near end,"\pi"
		]
		\arrow[Leftrightarrow, from=VVV, to=UUU,shorten >=0.05cm,shorten <=0.05cm
		,swap,near end,"\pi"
		]
		\end{tikzcd}
		\end{equation}
		
		\item[(DF2)]
		For $0 \leq j+1 = i \leq k \leq n$ or 
		$0 \leq j = i \leq k \leq n$
		\begin{equation}
		\begin{tikzcd}[row sep = 12pt, column sep = 35pt]
		\Omega^n_{\mathfrak{C}}
		\arrow{dr}[swap,name=U]{}{\boldsymbol{V}^k} \arrow{d}[swap]{s_j} &
		&
		\Omega^n_{\mathfrak{C}}
		\arrow{dr}[swap,name=UUU]{}{\boldsymbol{V}^k} \arrow[equal]{dd} &
		\\
		|[alias=V]|
		\Omega^{n+1}_{\mathfrak{C}} \ar{r}{\boldsymbol{V}^{k+1}}[swap,name=UU]{} \arrow{d}[swap]{d_i}&
		\Sigma \wr \Omega^{n-k-1}_{\mathfrak{C}}
		&
		&
		\Sigma \wr \Omega^{n-k-1}_{\mathfrak{C}}
		\\
		|[alias=VV]|
		\Omega^{n}_{\mathfrak{C}} \arrow{ur}[swap]{\boldsymbol{V}^k} &
		&
		|[alias=VVVV]|
		\Omega^{n}_{\mathfrak{C}} \arrow{ur}[swap]{\boldsymbol{V}^k} &
		\arrow[Leftrightarrow, from=VV, to=UU,shorten >=0.25cm,shorten <=0.05cm
		,swap,near end,"\pi"
		]
		\end{tikzcd}
		\end{equation}
		
		\item[(DF3)]
		For $0\leq i < j \leq k \leq n$
		\begin{equation}
		\begin{tikzcd}[row sep = 12pt, column sep = 35pt]
		\Omega^n_{\mathfrak{C}}
		\arrow{dr}[swap,name=U]{}{\boldsymbol{V}^k} \arrow{d}[swap]{s_j} &
		&
		\Omega^n_{\mathfrak{C}}
		\arrow{dr}[swap,name=UUU]{}{\boldsymbol{V}^k} \arrow{d}[swap]{d_{i}} &
		\\
		|[alias=V]|
		\Omega^{n+1}_{\mathfrak{C}} \ar{r}{\boldsymbol{V}^{k+1}}[swap,name=UU]{} \arrow{d}[swap]{d_i}&
		\Sigma \wr \Omega^{n-k-1}_{\mathfrak{C}}
		&
		|[alias=VVV]|
		\Omega^{n-1}_{\mathfrak{C}} \ar{r}[near start, swap]{\boldsymbol{V}^{k-1}}[swap,name=UUUU]{} \ar{d}[swap]{s_{j-1}} &
		\Sigma \wr \Omega^{n-k-1}_{\mathfrak{C}}
		\\
		|[alias=VV]|
		\Omega^{n}_{\mathfrak{C}} \arrow{ur}[swap]{\boldsymbol{V}^k} &
		&
		|[alias=VVVV]|
		\Omega^{n}_{\mathfrak{C}} \arrow{ur}[swap]{\boldsymbol{V}^k} &
		\arrow[Leftrightarrow, from=VV, to=UU,shorten >=0.25cm,shorten <=0.05cm
		,swap,near end,"\pi"
		]
		\arrow[Leftrightarrow, from=VVV, to=UUU,shorten >=0.05cm,shorten <=0.05cm
		,swap,near end,"\pi"
		]
		\end{tikzcd}
		\end{equation}
		
		\item[(DF4)]
		For $0 \leq i < k \leq j \leq n$
		\begin{equation}
		\begin{tikzcd}[row sep = 12pt, column sep = 35pt]
		\Omega^n_{\mathfrak{C}}
		\arrow{r}[swap,name=U]{}{\boldsymbol{V}^k} \arrow{d}[swap]{s_j} &
		\Sigma \wr \Omega^{n-k-1}_{\mathfrak{C}} \ar{d}{s_{j-k-1}}
		&
		\Omega^n_{\mathfrak{C}}
		\arrow{dr}[swap,name=UUU]{}{\boldsymbol{V}^k} \arrow{d}[swap]{d_i} &
		\\
		|[alias=V]|
		\Omega^{n+1}_{\mathfrak{C}} \ar{r}{\boldsymbol{V}^{k}}[swap,name=UU]{} \arrow{d}[swap]{d_i}&
		\Sigma \wr \Omega^{n-k}_{\mathfrak{C}}
		&
		|[alias=VVV]|
		\Omega^{n-1}_{\mathfrak{C}} \ar{r}[near start, swap]{\boldsymbol{V}^{k-1}}[swap,name=UUUU]{} \ar{d}[swap]{s_{j-1}} &
		\Sigma \wr \Omega^{n-k-1}_{\mathfrak{C}} \ar{d}{s_{j-k-1}}
		\\
		|[alias=VV]|
		\Omega^{n}_{\mathfrak{C}} \arrow{ur}[swap]{\boldsymbol{V}^{k-1}} &
		&
		|[alias=VVVV]|
		\Omega^{n}_{\mathfrak{C}} \ar{r}[swap]{\boldsymbol{V}^{k-1}} &
		\Sigma \wr \Omega^{n-k}_{\mathfrak{C}}
		\arrow[Leftrightarrow, from=VV, to=UU,shorten >=0.05cm,shorten <=0.05cm
		,swap,near end,"\pi"
		]
		\arrow[Leftrightarrow, from=VVV, to=UUU,shorten >=0.05cm,shorten <=0.05cm
		,swap,near end,"\pi"
		]
		\end{tikzcd}
		\end{equation}
	\end{itemize}
\end{proposition}

\subsection{The $(-)\wr A$ construction}\label{WRACONST SEC}

One of the key ingredients used in \cite[\S 4.2]{BP21} when describing the monad on spans (cf. \eqref{SPANSYMADJ EQ})
is the use of categories 
$\Omega^n \wr A$ defined by pullbacks diagrams of the form
\begin{equation}\label{WRASAMPLE EQ}
\begin{tikzcd}
\Omega^n \wr A \ar{r}{\boldsymbol{V}^n} \ar{d} &
\Sigma \wr A  \ar{d}
\\
\Omega^n \ar{r}{\boldsymbol{V}^n} &
\Sigma \wr \Sigma
\end{tikzcd}
\end{equation}
Moreover, these categories are related by analogues of the operators $d_i$, $s_j$, $\boldsymbol{V}^k$, $\pi_{i,k}$
which satisfy all the analogues of the compatibilities 
listed in Propositions \ref{CATDIAG PROP} and \ref{CATDIAG2 PROP}.

In \cite{BP21} we built these analogue operators in a somewhat ad-hoc manner. Here we use a more systematic approach,
which views the $(-) \wr A$ construction as a $2$-categorical extension of the pullback operation in $\Cat$.
We now define the required $2$-categories,
which are a variant of the 
$\mathsf{Cat}\downarrow^l \V$
categories in Remark \ref{SUBCATDOWNL REM} with regard to 
a split Grothendieck fibration $\mathcal{E} \to \mathcal{B}$
(Remark \ref{SPLITOPFIB REM}).
In what follows, arrows 
$\varphi^{\**}e \to e$ in
the chosen cleavage of $\mathcal{E}$ are called \emph{pullback arrows}.

\begin{definition}
	Let $\mathcal{E} \to \mathcal{B}$ be a split Grothendieck fibration.
	We write $\mathsf{Cat}\downarrow^r_{\mathcal{B}} \mathcal{E}$ for the $2$-category such that:
	\begin{itemize}
		\item objects are functors $F \colon \mathcal{C} \to \mathcal{E}$; 
		
		\item an $1$-arrow from 
		$F \colon \mathcal{C} \to \mathcal{E}$
		to
		$F' \colon \mathcal{C}' \to \mathcal{E}$
		is a pair $(\varphi,\phi)$
		formed by a functor $\varphi\colon \mathcal{C} \to \mathcal{C}'$ and a natural transformation $\phi \colon F' \varphi \Rightarrow F$ consisting of pullback arrows over $\mathcal{B}$
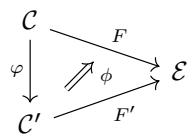
		\begin{equation}
		\begin{tikzcd}[row sep = 5pt, column sep = 40pt]
		\mathcal{C} \arrow{dr}[name=U]{F} \arrow{dd}[swap]{\varphi}
		\\
		& \mathcal{E}
		\\
		|[alias=V]| \mathcal{C}' \arrow{ur}[swap]{F'}
		\arrow[Rightarrow, from=V, to=U,shorten >=0.25cm,shorten <=0.25cm
		,swap,near end,"\phi"
		]
		\end{tikzcd}
		\end{equation}
		
		\item a $2$-arrow from $(\varphi,\phi)$ to $(\varphi',\phi')$ is a $2$-arrow $\Phi \colon \varphi \Rightarrow \varphi'$ such that
		$\phi' \circ F' \Phi = \phi$.
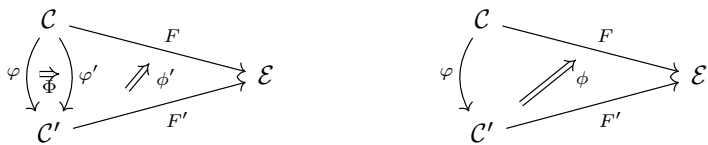
		\begin{equation}
		\begin{tikzcd}[column sep = 65pt,row sep = 7pt]
		\mathcal{C} \arrow{dr}[name=U]{F} 
		\arrow[bend right]{dd}[swap]{\varphi}[name=F]{}
		\arrow[bend left]{dd}{\varphi'}[swap,name=FF]{}
		&
		&
		\mathcal{C} \arrow{dr}[name=U2]{F} 
		\arrow[bend right]{dd}[swap]{\varphi}
		&
		\\
		& \mathcal{E}
		&
		& \mathcal{E}
		\\
		\mathcal{C}' \arrow{ur}[swap]{F'}[near start, name=V]{}
		&
		&
		|[alias=V2]| \mathcal{C}' \arrow{ur}[swap]{F'}
		&
		\arrow[Rightarrow, from=V, to=U,shorten >=0.25cm,shorten <=0.25cm
		,swap,near end,"\phi'"
		]
		\arrow[Rightarrow, from=F, to=FF,shorten >=0.0cm,shorten <=0.0cm
		,swap,"\Phi"
		]
		\arrow[Rightarrow, from=V2, to=U2,shorten >=0.25cm,shorten <=0.25cm
		,swap,near end,"\phi"
		]
		\end{tikzcd}
		\end{equation}
	\end{itemize}
\end{definition}

Given a map $\rho \colon \mathcal{E} \to \mathcal{F}$
of split Grothendieck fibrations over $\mathcal{B}$
(i.e. $\rho$ sends pullback arrows to pullback arrows),
we now define a pullback $2$-functor 
\begin{equation}\label{WSPANPULL EQ}
\rho^{\**} \colon
\mathsf{Cat} \downarrow^r_\mathcal{B} \mathcal{F} 
\to
\mathsf{Cat} \downarrow^r_\mathcal{B} \mathcal{E}.
\end{equation}

On objects, i.e. functors $F \colon \mathcal{C} \to \mathcal{F}$, one sets 
$\rho^{\**}(\mathcal{C} \to \mathcal{F})=
(\mathcal{C} \times_{\mathcal{F}} \mathcal{E}
\to \mathcal{E})
$.

On $1$-arrows, i.e. pairs 
$(\varphi,\phi \colon F_2 \circ \varphi \Rightarrow F_1)$
as in the bottom of the diagram below
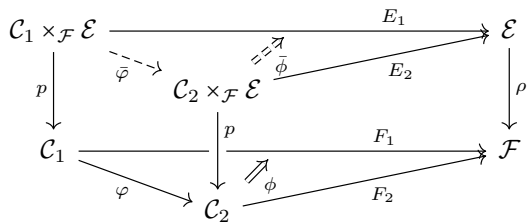
\[
\begin{tikzcd}[column sep = 20pt, row sep = 7pt]
\mathcal{C}_1 \times_{\mathcal{F}} \mathcal{E} 
\ar{rrrrr}[name=toE]{}[near end]{E_1} \ar[dashed]{rd}[swap]{\bar{\varphi}} \ar{dd}[swap]{p}
&&&
&&
\mathcal{E}  \ar{dd}{\rho}
\\
&
|[alias=DBE]|
\mathcal{C}_2 \times_{\mathcal{F}} \mathcal{E} \ar{rrrru}[swap]{E_2}
\\
\mathcal{C}_1 \ar{rrrrr}[name=toB]{}[near end]{F_1} \ar{rd}[swap]{\varphi}
&&&
&&
\mathcal{F} 
\\
&
|[alias=D]| \mathcal{C}_2 \ar{rrrru}[swap]{F_2}
\arrow[Rightarrow, from=DBE, to=toE, shorten <=0.15cm,shorten >=0.15cm,dashed
,swap,"\bar{\phi}"
]
\arrow[Rightarrow, from=D, to=toB, shorten <=0.15cm,shorten >=0.15cm,swap,"\phi"]
\arrow[from=DBE, to=D, crossing over, near start, "p"]
\end{tikzcd}
\]
we define $\rho^{\**}(\varphi,\phi)$ as the only possible choice of dashed data
$(\bar{\varphi},\bar{\phi})$ such that:
\begin{enumerate*}
\item[(i)] $\bar{\phi}$ consists of pullback arrows over $\mathcal{B}$
and;
\item[(ii)] the diagram commutes in the sense that
$p \bar{\varphi} = \varphi p$ and 
$\rho \bar{\phi} = \phi p$.
\end{enumerate*}

Alternatively, one has the explicit formula
\[
\rho^{\**} (\varphi,\phi)=
\left(
\left( \varphi p,
\left( \phi p \right)^{\**} E_1 \right),
\left( \phi p \right)^{\**} E_1 \Rightarrow E_1
\right).
\]


Lastly, on a $2$-arrow $\Phi \colon (\varphi,\phi) \Rightarrow (\varphi',\phi')$
as on the bottom of the leftmost diagram below
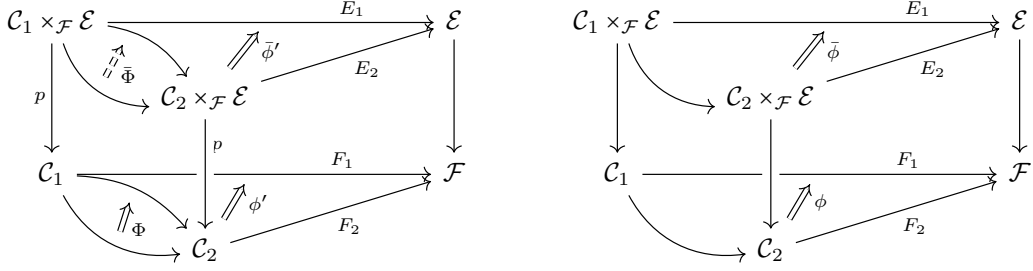
\begin{equation}\label{PULL2ARR EQ}
\begin{tikzcd}[column sep = 16pt, row sep = 13pt]
\mathcal{C}_1 \times_{\mathcal{F}} \mathcal{E} 
\ar{rrrrr}[name=toE]{}[near end]{E_1} 
\ar[bend left=23]{rd}[near start, swap, name=FE]{}
\ar[bend right=35]{rd}[name=FFE]{} \ar{dd}[swap]{p}
&&&
&&
\mathcal{E}  \ar{dd}
&&
\mathcal{C}_1 \times_{\mathcal{F}} \mathcal{E} 
\ar{rrrrr}[name=toE2]{}[near end]{E_1} 
\ar[bend right=35]{rd}{} \ar{dd}
&&&
&&
\mathcal{E}  \ar{dd}
\\
&
|[alias=DBE]|
\mathcal{C}_2 \times_{\mathcal{F}} \mathcal{E} \ar{rrrru}[swap]{E_2}
\ar{dd}[near start]{p} &&&&
&&
&
|[alias=DBE2]|
\mathcal{C}_2 \times_{\mathcal{F}} \mathcal{E} \ar{rrrru}[swap]{E_2} &&&&
\\
\mathcal{C}_1 \ar{rrrrr}[name=toB]{}[near end]{F_1} 
\ar[bend left=23]{rd}[swap,name=FF]{}
\ar[bend right=35]{rd} [name=F]{}
&&&
&&
\mathcal{F} 
&&
\mathcal{C}_1 \ar{rrrrr}[name=toB2]{}[near end]{F_1} 
\ar[bend right=35]{rd}{}
&&&
&&
\mathcal{F} 
\\
&
|[alias=D]| \mathcal{C}_2 \ar{rrrru}[swap]{F_2} &&&&
&&
&
|[alias=D2]|
\mathcal{C}_2 \ar{rrrru}[swap]{F_2} &&&&
\arrow[Rightarrow, from=DBE, to=toE, shorten <=0.15cm,shorten >=0.15cm
,swap,near end,"\bar{\phi}'"
]
\arrow[Rightarrow, from=DBE2, to=toE2, shorten <=0.15cm,shorten >=0.15cm
,swap,near end,"\bar{\phi}"
]
\arrow[Rightarrow, from=D, to=toB, shorten <=0.15cm,shorten >=0.15cm,swap,near end,"\phi'"]
\arrow[Rightarrow, from=D2, to=toB2, shorten <=0.15cm,shorten >=0.15cm,swap,near end,"\phi"]
\arrow[Rightarrow, from=F, to=FF, shorten <=0cm,shorten >=0cm,swap,"\Phi"]
\arrow[Rightarrow, from=FFE, to=FE, shorten <=0cm,shorten >=0cm,swap,dashed,"\bar{\Phi}"]
\arrow[from=DBE, to=D, crossing over]
\arrow[from=DBE2, to=D2, crossing over]
\end{tikzcd}
\end{equation}
we define $\rho^{\**}(\Phi)$
as the only choice (recall that $\bar{\varphi}'$ is a pullback arrow)
of dashed $\bar{\Phi}$
such that $\bar{\phi}' \circ E_2\bar{\Phi} = \bar{\phi}$
and $p \bar{\Phi} = \Phi p$.




\vskip 10pt

We are now ready to extend the $(-) \wr A$
construction from \eqref{WRASAMPLE EQ}.

First, note that, using the functor
$\boldsymbol{V}^n \colon \Omega^n_{\mathfrak{C}} \to \Sigma \wr \Sigma_{\mathfrak{C}}$,
the categories 
$\Omega^n_{\mathfrak{C}}$ may be regarded as objects in
$\mathsf{Cat} \downarrow^r_{\Sigma} \Sigma \wr \Sigma_{\mathfrak{C}}$.
Hence, given a functor $A \to \Sigma_{\mathfrak{C}}$
we define 
\begin{equation}\label{WRADEF EQ}
(-) \wr A \colon 
\mathsf{Cat} \downarrow^r_{\Sigma} \Sigma \wr \Sigma_{\mathfrak{C}}
\to
\mathsf{Cat} \downarrow^r_{\Sigma} \Sigma \wr A
\end{equation}
as the pullback $2$-functor \eqref{WSPANPULL EQ} for the map
$\Sigma \wr A \to \Sigma \wr \Sigma_{\mathfrak{C}}$.

Focusing on objects, \eqref{WRADEF EQ} then defines 
$\Omega_{\mathfrak{C}}^n \wr A$
together with maps 
$\Omega_{\mathfrak{C}}^n \wr A 
\xrightarrow{\boldsymbol{V}^n} \Sigma \wr A$.

Next, by specifying Proposition \ref{CATDIAG PROP}(i) to $k=n$,
one obtains $1$-arrows 
$(d_i,\pi_{i_n}), 0 \leq i < n$ and 
$(s_j,id_{\boldsymbol{V}^n}), -1 \leq j \leq n$
in 
$\mathsf{Cat} \downarrow^r_{\Sigma} \Sigma \wr \Sigma_{\mathfrak{C}}$.
Thus, by specifying \eqref{WRADEF EQ} to $1$-arrows
one also obtains $1$-arrows
$(d_i,\pi_{i_n}), 0 \leq i < n$ and 
$(s_j,id_{\boldsymbol{V}^n}), -1 \leq j \leq n$
in $\mathsf{Cat} \downarrow^r_{\Sigma} \Sigma \wr A$
between the $\Omega_{\mathfrak{C}}^n \wr A \to \Sigma \wr A$.

We next describe the functors
$\boldsymbol{V}^k \colon \Omega^n_{\mathfrak{C}} \wr A
\to 
\Sigma \wr \left(\Omega^{n-k-1}_{\mathfrak{C}} \wr A \right)$.
Note that the $\Sigma \wr (-)$ operation can be extended to a $2$-endofunctor
\begin{equation}\label{SIGMAENDO EQ}
\begin{tikzcd}[row sep = 0pt]
\mathsf{Cat} \downarrow^r_{\Sigma} \Sigma \wr A \ar{r}{\Sigma \wr (-)} &
\mathsf{Cat} \downarrow^r_{\Sigma} \Sigma \wr A
\\
\mathcal{C} \to \Sigma \wr A \ar[mapsto]{r} &
\Sigma \wr \mathcal{C} \to \Sigma^{\wr 2} \wr A \xrightarrow{\sigma^{0}} \Sigma \wr A 
\end{tikzcd}
\end{equation}
so that, noting that the diagram below consists of pullback squares,
\begin{equation}
\begin{tikzcd}[row sep = 8pt]
\Sigma \wr (\Omega^n \wr A) \ar{r}{\Sigma \wr \boldsymbol{V}^n} \ar{d} &
\Sigma \wr (\Sigma \wr A) \ar{d} \ar{r}{\sigma^0} &
\Sigma \wr A  \ar{d}
\\
\Sigma \wr \Omega^n \ar{r}{\Sigma \wr \boldsymbol{V}^n} &
\Sigma \wr (\Sigma \wr \Sigma_{\mathfrak{C}}) \ar{r}{\sigma^0} &
\Sigma \wr \Sigma_{\mathfrak{C}}
\end{tikzcd}
\end{equation}
one has identifications
$\Sigma \wr \left(\Omega^{n}_{\mathfrak{C}} \wr A \right)
\simeq 
\left(\Sigma \wr \Omega^{n}_{\mathfrak{C}}\right) \wr A $,
which are seen to agree with the $\sigma^i,\delta^j$ operators on $\Sigma^{\wr k} \wr (-)$.
Thus, we henceforth omit parenthesis and write 
$\Sigma \wr \Omega^{n}_{\mathfrak{C}} \wr A$
to denote 
$\Sigma \wr \left(\Omega^{n}_{\mathfrak{C}} \wr A \right)$.

Specifying \eqref{SIGMAENDO EQ} to $A = \Sigma_{\mathfrak{C}}$,
by \eqref{VKADD EQ} when $k=n-l-1$ one has that
$(\boldsymbol{V}^l,id_{\boldsymbol{V}^{n}})$
defines a $1$-arrow from 
$\Omega^n_{\mathfrak{C}} \to \Sigma \wr \Sigma_{\mathfrak{C}}$
to 
$\Sigma \wr \Omega^{n-l-1}_{\mathfrak{C}} \to \Sigma \wr \Sigma_{\mathfrak{C}}$
in
$\mathsf{Cat} \downarrow^r_{\Sigma} \Sigma \wr \Sigma_{\mathfrak{C}}$. 
Thus, applying \eqref{WRADEF EQ} yields a $1$-arrow
$(\boldsymbol{V}^l,id_{\boldsymbol{V}^{n}})$
from
$\Omega^n_{\mathfrak{C}}\wr A \to \Sigma \wr A$
to 
$\Sigma \wr \Omega^{n-l-1}_{\mathfrak{C}} \wr A \to \Sigma \wr A$
in
$\mathsf{Cat} \downarrow^r_{\Sigma} \Sigma \wr A$.

It is now immediate from $2$-functoriality of
\eqref{WRADEF EQ}
that one has natural transformations $\pi_{i,k}$
between the $\Omega^n_{\mathfrak{C}}\wr A$
satisfying the analogues of 
Proposition \ref{CATDIAG PROP}(i)(ii) and Proposition \ref{CATDIAG2 PROP}.

In what follows, we will find it convenient to abbreviate
$\Omega_{\mathfrak{C}}^n \wr A \to \Sigma \wr A$
as 
$\Omega_{\mathfrak{C}}^n \wr A$
and
$(d_i,\pi_{i_n}), (s_j,id_{\boldsymbol{V}^n}),
(\boldsymbol{V}^k,id_{\boldsymbol{V}^n})
$
as 
$d_i,s_j,\boldsymbol{V}^k$, 
i.e. we will leave the $\Sigma \wr A$
data implicit.

The analogue of Proposition \ref{CATDIAG PROP}(iii) requires an extra argument, and is stated in the following.

\begin{proposition}\label{SPANPIECE PROP}
	A commutative square
	\begin{equation}\label{SPANPIECE EQ}
	\begin{tikzcd}
	A \ar{d} \ar{r}{\varphi} &  \ar{d} B
	\\
	\Sigma_{\mathfrak{C}} \ar{r}[swap]{\varphi} & \Sigma_{\mathfrak{D}}
	\end{tikzcd}
	\end{equation}
	induces natural maps 
	$\varphi \colon
	\Omega_{\mathfrak{C}}^n \wr A \to 
	\Omega_{\mathfrak{D}}^n \wr B $
	such that the diagrams below commute.
	\[
	\begin{tikzcd}[column sep = 16pt, row sep = 10pt]
	\Omega^n_{\mathfrak{C}} \wr A \ar{r}{d_i} \ar{d}[swap]{\varphi} &
	\Omega^{n-1}_{\mathfrak{C}} \wr A \ar{d}{\varphi}
	&
	\Omega^n_{\mathfrak{C}} \wr A \ar{r}{s_j} \ar{d}[swap]{\varphi} &
	\Omega^{n+1}_{\mathfrak{C}}\wr A \ar{d}{\varphi}
	&
	\Omega^n_{\mathfrak{C}} \wr A \ar{r}{\boldsymbol{V}^k} \ar{d}[swap]{\varphi} &
	\Sigma \wr \Omega^{n-k-1}_{\mathfrak{C}} \wr A \ar{d}{\varphi}
	\\
	\Omega^n_{\mathfrak{D}} \wr B \ar{r}{d_i} &
	\Omega^{n-1}_{\mathfrak{D}} \wr B
	&
	\Omega^n_{\mathfrak{D}} \wr B \ar{r}{s_j} &
	\Omega^n_{\mathfrak{D}} \wr B
	&
	\Omega^n_{\mathfrak{D}} \wr B \ar{r}{\boldsymbol{V}^k} &
	\Sigma \wr \Omega^{n-k-1}_{\mathfrak{D}} \wr B
	\end{tikzcd}
	\]
	\[
	\begin{tikzcd}[column sep = 16pt, row sep = small]
	\Omega^n_{\mathfrak{C}} \wr A
	\ar{rrrrr}[name=toE]{\boldsymbol{V}^k} \ar{rd}[swap]{d_i} \ar{dd}[swap]{\varphi}
	&&&
	&&
	\Sigma \wr \Omega^{n-k-1}_{\mathfrak{C}} \wr A  \ar{dd}{\varphi}
	\\
	&
	|[alias=DBE]|
	\Omega^{n-1}_{\mathfrak{C}} \wr A \ar{rrrru}[swap]{\boldsymbol{V}^{k-1}}
	\\
	\Omega^n_{\mathfrak{D}} \wr B \ar{rrrrr}[name=toB]{\boldsymbol{V}^k} \ar{rd}[swap]{d_i}
	&&&
	&&
	\Sigma \wr \Omega^{n-k-1}_{\mathfrak{D}} \wr B
	\\
	&
	|[alias=D]| \Omega^{n-1}_{\mathfrak{D}} \wr B \ar{rrrru}[swap]{\boldsymbol{V}^{k-1}}
	\arrow[Leftrightarrow, from=DBE, to=toE, shorten <=0.15cm,shorten >=0.15cm
	,swap,near end,"\pi"
	]
	\arrow[Leftrightarrow, from=D, to=toB, shorten <=0.15cm,shorten >=0.15cm,swap,near end,"\pi"]
	\arrow[from=DBE, to=D, crossing over, near start, swap, "\varphi"]
	\end{tikzcd}
	\]
\end{proposition}

\begin{proof}
	The desired maps 
	$\varphi \colon
	\Omega_{\mathfrak{C}}^n \wr A \to 
	\Omega_{\mathfrak{D}}^n \wr B $
	are obtained by just drawing the pullback diagrams defining each term,
	so we focus on the more interesting claim that the given diagrams commute.
	To see this, we first factor \eqref{SPANPIECE EQ} as
	\[
	\begin{tikzcd}
	A \ar{d} \ar{r} & B \times_{\Sigma_{\mathfrak{D}}} \Sigma_{\mathfrak{C}} \ar{r} \ar{d} &  \ar{d} B
	\\
	\Sigma_{\mathfrak{C}} \ar[equal]{r} & \Sigma_{\mathfrak{C}} \ar{r} & \Sigma_{\mathfrak{D}}
	\end{tikzcd}
	\]
	and note that it suffices to prove the result separately for each half.
	For the left half, the desired commutativity claims 
	follow from naturality of $(-) \wr A$ with respect to $A$.
	On the other hand, for the right square the commutativity claims follow by instead noting that all diagrams in 
	Proposition \ref{CATDIAG PROP}(iii)
	can be regarded as diagrams in the $2$-category
	$\mathsf{Cat} \downarrow^r_{\Sigma} \Sigma \wr \Sigma_{\mathfrak{D}}$
	(by using the composites 
	$\Omega_{\mathfrak{C}}^n \xrightarrow{\varphi} \Omega_{\mathfrak{D}}^n \xrightarrow{\boldsymbol{V}^n} \Sigma \wr \Sigma_{\mathfrak{D}}$) 
	and then applying the pullback functor $(-) \wr B$. 
\end{proof}

In what follows, 
for an increasing sequence $i_1<i_2<\cdots<i_k$,
we write
$d_{i_1,i_2,\cdots,i_k}=d_{i_1} d_{i_2} \cdots d_{i_k}$.

Next, note that, using the composite functors 
$\Omega^n_{\mathfrak{C}} \wr A \to \Omega^n_{\mathfrak{C}} 
\xrightarrow{d_{0,\cdots,n}} \Sigma_{\mathfrak{C}}$,
one can regard the $\Omega_{\mathfrak{C}}^n \wr (-)$ constructions
as endofunctors on the usual $1$-overcategory
$\mathsf{Cat} \downarrow \Sigma_{\mathfrak{C}}$.

\begin{proposition}\label{ASSOCIDS PROP}
	Let $k,l\geq -1$. One has canonical natural identifications 
	$\Omega^k_{\mathfrak{C}} \wr \Omega^l_{\mathfrak{C}} \wr A
	\simeq 
	\Omega^{k+l+1}_{\mathfrak{C}} \wr A $.
	
	Moreover, these identifications are associative in the sense that, for any $k,l,m \leq -1$,
	the iterated composite identifications below coincide.
\[
	\Omega^k_{\mathfrak{C}} \wr \Omega^l_{\mathfrak{C}} \wr \Omega^m_{\mathfrak{C}} \wr A
	\simeq 
	\Omega^{k+l+1}_{\mathfrak{C}} \wr \Omega^m_{\mathfrak{C}} \wr A
	\simeq 
	\Omega^{k+l+m+2}_{\mathfrak{C}} \wr A
	\qquad
	\Omega^k_{\mathfrak{C}} \wr \Omega^l_{\mathfrak{C}} \wr \Omega^m_{\mathfrak{C}} \wr A
	\simeq 
	\Omega^{k}_{\mathfrak{C}} \wr \Omega^{l+m+1}_{\mathfrak{C}} \wr A
	\simeq 
	\Omega^{k+l+m+2}_{\mathfrak{C}} \wr A
\]
	Furthermore, the identifications above also induce the following identifications
	\[
	d_i \wr \Omega^l \wr A \simeq d_i \wr A
	\quad
	\pi_{i,k} \wr \Omega^l \wr A \simeq \pi_{i,k} \wr A
	\quad
	s_j \wr \Omega^l \wr A \simeq s_j \wr A
	\quad
	\Omega^k \wr d_i \wr A \simeq d_{k+i+1} \wr A
	\quad
	\Omega^k \wr s_j \wr A \simeq s_{k+j+1} \wr A
	\]
\end{proposition}

\begin{proof}
	The first claim follows by noting that all squares in the diagram below are pullback squares
	\[
	\begin{tikzcd}
	\Omega^{k+l+1}_{\mathfrak{C}} \wr A \ar{r}{\boldsymbol{V}^k} \ar{d} &
	\Sigma \wr \Omega^{l}_{\mathfrak{C}} \wr A  \ar{d} \ar{r}{\boldsymbol{V}^l} &
	\Sigma^{\wr 2} \wr A \ar{d}
	\\
	\Omega^{k+l+1}_{\mathfrak{C}} \ar{r}{\boldsymbol{V}^k} 
	\ar{d}[swap]{d_{k+1,\cdots,k+l+1}} &
	\Sigma \wr \Omega^{l}_{\mathfrak{C}} \ar{r}{\boldsymbol{V}^l}
	\ar{d}{d_{0,\cdots,l}} &
	\Sigma^{\wr 2} \wr \Sigma_{\mathfrak{C}}
	\\
	\Omega^{k}_{\mathfrak{C}} \ar{r}{\boldsymbol{V}^k} &
	\Sigma \wr \Sigma_{\mathfrak{C}}
	\end{tikzcd}
	\]
	while associativity follows from the obvious 
	``3 level analogue'' of the diagram above.

	For the additional identifications, 
	those identifications concerning $d_i$ and $\pi_{i,k}$
	follow from the left diagram below 
	(the bottom section of which commutes by 
	Proposition \ref{CATDIAG2 PROP}(FF2)),
	the identification concerning $d_{k+i+1}$ follows from the rightmost diagram, and the identifications concerning 
	$s_j$ and $s_{k+j+1}$
	follow from obvious analogues of these diagrams.
	\[
	\begin{tikzcd}[row sep = 5pt,column sep = 13pt]
	\Omega^{k+l+1}_{\mathfrak{C}} \wr A \ar{rrr}[near end, name=TT,swap]{} \ar{rd}[swap]{d_i} \ar{dd} &&&
	\Sigma \wr \Omega^{l}_{\mathfrak{C}} \wr A \ar{dd}
	&
	\Omega^{k+l+1}_{\mathfrak{C}} \wr A \ar{rr} \ar{rd}[swap]{d_{k+i+1}} \ar{dd} &&
	\Sigma \wr \Omega^{l}_{\mathfrak{C}} \wr A \ar{rd} \ar{dd}
	\\
	&
	|[alias=TD]|
	\Omega^{k+l}_{\mathfrak{C}} \wr A \ar{rru} &&
	&
	&
	\Omega^{k+l}_{\mathfrak{C}} \wr A \arrow[rr, crossing over] &&
	\Sigma \wr \Omega^{l-1}_{\mathfrak{C}} \wr A 	 \ar{dd}
	\\
	\Omega^{k+l+1}_{\mathfrak{C}} \ar{rrr}[near end, name=MT,swap]{} \ar{rd}[swap]{d_i} \ar{dd} &&&
	\Sigma \wr \Omega^{l}_{\mathfrak{C}} \ar{dd}
	&
	\Omega^{k+l+1}_{\mathfrak{C}} \ar{rr} \ar{rd}[swap]{d_{k+i+1}} \ar{dd}&&
	\Sigma \wr \Omega^{l}_{\mathfrak{C}} \ar{rd} \ar{dd}
	\\
	&
	|[alias=MD]|
	\Omega^{k+l}_{\mathfrak{C}} \ar{rru} \arrow[uu, leftarrow, crossing over] &&
	&
	&
	\Omega^{k+l}_{\mathfrak{C}} \arrow[rr, crossing over] \arrow[uu, leftarrow, crossing over] &&
	\Sigma \wr \Omega^{l-1}_{\mathfrak{C}} \ar{dd}
	\\
	\Omega^{k}_{\mathfrak{C}} \ar{rrr}[near end, name=DT,swap]{} \ar{rd}[swap]{d_i} &&&
	\Sigma \wr \Sigma_{\mathfrak{C}}
	&
	\Omega^{k}_{\mathfrak{C}} \ar{rr} \ar[equal]{rd} &&
	\Sigma \wr \Sigma_{\mathfrak{C}} \ar[equal]{rd}
	\\
	&
	|[alias=DD]|
	\Omega^{k-1}_{\mathfrak{C}} \ar{rru} \arrow[uu, leftarrow, crossing over] &&
	&
	&
	\Omega^{k}_{\mathfrak{C}} \ar{rr} \arrow[uu, leftarrow, crossing over] &&
	\Sigma \wr \Sigma_{\mathfrak{C}} 
	\arrow[Leftrightarrow, from=TT, to=TD,shorten >=0.05cm,shorten <=0.05cm, near start,
	"\pi"
	]
	\arrow[Leftrightarrow, from=MT, to=MD,shorten >=0.05cm,shorten <=0.05cm, near start,
	"\pi"
	]
	\arrow[Leftrightarrow, from=DT, to=DD,shorten >=0.05cm,shorten <=0.05cm, near start,
	"\pi"
	]
	\end{tikzcd}
	\]
\end{proof}

\subsection{The fibered monads on spans and symmetric sequences} \label{NONEQMON SEC}

In this section we finally complete the definition of the fibered free operad monad
$\mathbb{F}$ in Definition \ref{FREEOP DEF},
starting with the promised monad $N$ on the category $\mathsf{WSpan}^l(\Sigma_{\bullet}^{op},\mathcal{V})$ of spans
(cf. \eqref{SPANSYMADJ EQ}).

\subsubsection*{The fibered monad on colored spans}

\begin{definition}\label{WSPANCAT DEF}
	The category $\mathsf{WSpan}^l(\Sigma_{\bullet}^{op},\mathcal{V})$ has
	\begin{itemize}
		\item objects given by a choice of a set of colors $\mathfrak{C}$
		and a span $\Sigma^{op}_{\mathfrak{C}} \leftarrow A^{op} \rightarrow \mathcal{V}$
		\item morphisms given by a choice of a map of colors
		$\varphi \colon \mathfrak{C} \to \mathfrak{D}$,
		together with a commutative square and natural transformation as given below.
		\begin{equation}\label{COLORSPANMAP EQ}
		\begin{tikzcd}[column sep = 20pt]
		\Sigma_{\mathfrak{C}}^{op}
		\ar{d}[swap]{\varphi^{op}} &
		A^{op}
		\ar{r}[name=U,swap]{} \ar{d} \ar{l} &
		\mathcal{V}	
		\\
		\Sigma_{\mathfrak{D}}^{op}
		&
		|[alias=V]|
		B^{op} \ar{l}
		\ar{ru}
		\arrow[Leftarrow, from=V, to=U,shorten >=0.05cm,shorten <=0.1cm]
		\end{tikzcd}
		\end{equation}
	\end{itemize}
\end{definition}

\begin{remark}
	By definition, there is a forgetful functor
	$\mathsf{WSpan}^l(\Sigma_{\bullet}^{op},\mathcal{V}) \to \mathsf{Set}$
	which remembers the set of colors.
	Moreover, this is a Grothendieck fibration, with the cartesian arrows the diagrams \eqref{COLORSPANMAP EQ} where the square is a pullback square and the natural transformation is an isomorphism.
\end{remark}

\begin{remark}\label{LANADJ REM}
	Given a span $\Sigma^{op}_{\mathfrak{C}} \leftarrow A^{op} \rightarrow \mathcal{V}$
	one can the form the left Kan extension
	\[
	\begin{tikzcd}[column sep = 30pt]
	A^{op}
	\ar{r}[name=U,swap]{}{F} \ar{d} 
	&
	\mathcal{V}	
	\\
	|[alias=V]|
	\Sigma_{\mathfrak{C}}^{op} 
	\ar{ru}[swap]{\mathsf{Lan} F}
	\arrow[Leftarrow, from=V, to=U,shorten >=0.05cm,shorten <=0.05cm]
	\end{tikzcd}
	\]
	This defines an adjunction (cf. Definition \ref{FIBADJ DEF}),
	fibered over $\mathsf{Set}$,
	\[
	\mathsf{Lan} \colon
	\mathsf{WSpan}^l(\Sigma_{\bullet}^{op},\mathcal{V}) 
	\rightleftarrows
	\mathsf{Sym}_{\bullet}(\mathcal{V})
	\colon \upsilon
	\]
	where the inclusion $\upsilon$ sends $\Sigma^{op}_{\mathfrak{C}} \to \mathcal{V}$ 
	to the span
	$\Sigma^{op}_{\mathfrak{C}} \xleftarrow{=} \Sigma^{op}_{\mathfrak{C}} \to \mathcal{V}$.
\end{remark}

\begin{remark}
	One can also define a larger category 
	$\mathsf{WSpan}^l( - ,\mathcal{V})$
	where the categories $\Sigma_{\mathfrak{C}}^{op}$ in the spans (and functors between them) are allowed to be any category (any functor),
	in which case left Kan extension defines a fibered adjunction over 
	$\mathsf{Cat}$ (cf. Remark \ref{SUBCATDOWNL REM})
	\[
	\mathsf{Lan} \colon
	\mathsf{WSpan}^l( - ,\mathcal{V}) 
	\rightleftarrows
	\mathsf{Cat} \downarrow ^l \mathcal{V}
	\colon \upsilon.
	\]
\end{remark}

\begin{definition}[{cf. \cite[Def. 4.15]{BP21}}]
        \label{NCOLOR DEF}
	The monad $N$ on 
	$\mathsf{WSpan}^l(\Sigma_{\bullet}^{op},\mathcal{V})$
	sends the span 
	$\Sigma^{op}_{\mathfrak{C}} \leftarrow A^{op} \to \mathcal{V}$
	to the (opposite of the) composite span in
	\begin{equation}\label{NCOLOR EQ}
	\begin{tikzcd}
	\Omega^0_{\mathfrak{C}} \wr A \ar{r}{\boldsymbol{V}^0} \ar{d} &
	\Sigma \wr A  \ar{d} \ar{r} &
	\Sigma \wr \mathcal{V}^{op} \ar{r}{\otimes} &
	\mathcal{V}^{op}
	\\
	\Omega^0_{\mathfrak{C}} \ar{r}{\boldsymbol{V}^0} \arrow[d, "\mathsf{lr}"'] &
	\Sigma \wr \Sigma_{\mathfrak{C}} 
	\\
	\Sigma_{\mathfrak{C}}
	\end{tikzcd}
	\end{equation}
	has monad multiplication
	$\mu \colon N N
	\Rightarrow 
	N$ given by the diagram
	(note that Proposition \ref{ASSOCIDS PROP} identifies 
	$NN$ evaluated at 
	$\Sigma^{op}_{\mathfrak{C}} \leftarrow A^{op} \to \V$
	as (the opposite of) the composite top span)
	\begin{equation}\label{NMONMULTTR EQ}
	\begin{tikzcd}
	\Sigma_{\mathfrak{C}} \ar[equal]{d}&
	\Omega^1_{\mathfrak{C}} \wr A \ar{l} \arrow[r, "\boldsymbol{V}^0"] \ar{d}[swap]{d_0}&
	\Sigma \wr \Omega^0_{\mathfrak{C}} \wr A \arrow[r, "\boldsymbol{V}^0"] &
	|[alias=U]|
	\Sigma^{\wr 2} \wr A \ar{r} \ar{d}[swap]{\sigma^0} &
	\Sigma^{\wr 2} \wr \mathcal{V}^{op} \ar{r}{\otimes} \ar{d}[swap]{\sigma^0} &
	\Sigma \wr \mathcal{V}^{op} \ar{r}{\otimes} &
	|[alias=UU]|
	\mathcal{V}^{op} \ar[equal]{d}
	\\
	\Sigma_{\mathfrak{C}} &
	|[alias=V]|
	\Omega^0_{\mathfrak{C}} \wr A \ar{l} \ar{rr}[swap]{\boldsymbol{V}^0} & &
	\Sigma \wr A \ar{r} &
	|[alias=VV]|
	\Sigma \wr \mathcal{V}^{op} \ar{rr}[swap]{\otimes} & &
	\mathcal{V}^{op}
	\arrow[Leftrightarrow, from=V, to=U,shorten >=0.15cm,shorten <=0.15cm
	,swap,"\pi"
	]
	\arrow[Leftrightarrow, from=VV, to=UU,shorten >=0.15cm,shorten <=0.15cm,swap,"\alpha"
	]
	\end{tikzcd}
	\end{equation}
	(where $\alpha$ is the associativity data for
	$\otimes$, cf. \cite[(2.14)]{BP21})
	and unit
	$\eta \colon id \Rightarrow N$ given by
	\begin{equation}\label{NMONIDTR EQ}
	\begin{tikzcd}
	\Sigma_{\mathfrak{C}} \ar[equal]{d} & 
	A \ar{d}[swap]{s_{-1}} \ar{l} \ar[equal]{r} &
	A \ar{d}[swap]{\delta^0} \ar{r} &
	\mathcal{V}^{op} \ar{d}[swap]{\delta^0} \ar[equal]{r} &
	\mathcal{V}^{op} \ar[equal]{d}
	\\
	\Sigma_{\mathfrak{C}} &
	\Omega^0_{\mathfrak{C}} \wr A \ar{l} \ar{r} &
	\Sigma \wr A \ar{r} &
	\Sigma \wr \mathcal{V}^{op} \ar{r}{\otimes} &
	\mathcal{V}^{op}
	\end{tikzcd}
	\end{equation}
	
	Functoriality of $N$ with respect to maps that change colors follows from Proposition \ref{SPANPIECE PROP}.
\end{definition}

\begin{proposition}[{cf. \cite[Prop. 4.18]{BP21}}]
        \label{MONISMON PROP}
	$N$ is a monad on $\mathsf{WSpan}^l(\Sigma_{\bullet}^{op},\mathcal{V})$,
	fibered over $\mathsf{Set}$.
\end{proposition}

\begin{proof}
	To check associativity, the functor $\mu N \colon 
	N N N
	\Rightarrow N N$
	is encoded by the diagram
	\[
	\begin{tikzcd}[column sep=12pt]
	\Omega^2_{\mathfrak{C}} \wr A \ar{r} \ar{d}[swap]{d_0} &
	\Sigma \wr \Omega^1_{\mathfrak{C}} \wr A \ar{r} &
	|[alias=UUU]|
	\Sigma^{\wr 2} \wr \Omega^0_{\mathfrak{C}} \wr A
	\ar{d}[swap]{\sigma^0} \ar{r} &
	\Sigma^{\wr 3} \wr A \ar{d}[swap]{\sigma^0} \ar{r} &
	\Sigma^{\wr 3} \wr \mathcal{V}^{op} \ar{d}[swap]{\sigma^0} \ar{r}{\otimes} &
	\Sigma^{\wr 2} \wr \mathcal{V}^{op} \ar{d}[swap]{\sigma^0} \ar{r}{\otimes} &
	\Sigma \wr \mathcal{V}^{op} \ar{r}{\otimes} & 
	|[alias=UUUU]|
	\mathcal{V}^{op} \ar[equal]{d}
	\\
	|[alias=VVV]|
	\Omega^1_{\mathfrak{C}} \wr A \ar{rr} \ar{d}[swap]{d_0} & &
	\Sigma \wr \Omega^0_{\mathfrak{C}} \wr A \ar{r} &
	|[alias=U]|
	\Sigma^{\wr 2} \wr A \ar{r} \ar{d}[swap]{\sigma^0} &
	\Sigma^{\wr 2} \wr \mathcal{V}^{op} \ar{r}{\otimes} \ar{d}[swap]{\sigma^0} &
	|[alias=VVVV]|
	\Sigma \wr \mathcal{V}^{op} \ar{rr}{\otimes} & &
	|[alias=UU]|
	\mathcal{V}^{op} \ar[equal]{d}
	\\
	|[alias=V]|
	\Omega^0_{\mathfrak{C}} \wr A \ar{rrr} & & &
	\Sigma \wr A \ar{r} &
	|[alias=VV]|
	\Sigma \wr \mathcal{V}^{op} \ar{rrr}{\otimes} & & &
	\mathcal{V}^{op}
	\arrow[Leftrightarrow, from=V, to=U,shorten >=0.15cm,shorten <=0.15cm
	,swap,"\pi"
	]
	\arrow[Leftrightarrow, from=VV, to=UU,shorten >=0.15cm,shorten <=0.15cm,swap,"\alpha"
	]
	\arrow[Leftrightarrow, from=VVV, to=UUU,shorten >=0.15cm,shorten <=0.15cm
	,swap,"\pi"
	]
	\arrow[Leftrightarrow, from=VVVV, to=UUUU,shorten >=0.15cm,shorten <=0.15cm,swap,"\alpha"
	]
	\end{tikzcd}
	\]
	while the functor
	$ N \mu \colon 
	N N N
	\Rightarrow N N$
	is encoded by
	\[
	\begin{tikzcd}[column sep=12pt]
	\Omega^2_{\mathfrak{C}} \wr A \ar{d}[swap]{d_1} \ar{r} &
	\Sigma \wr \Omega^1_{\mathfrak{C}} \wr A \ar{d}[swap]{d_0} \ar{r} &
	\Sigma^{\wr 2} \wr \Omega^0_{\mathfrak{C}} \wr A \ar{r} &
	|[alias=UUU]|
	\Sigma^{\wr 3} \wr A \ar{d}[swap]{\sigma^1} \ar{r} &
	\Sigma^{\wr 3} \wr \mathcal{V}^{op} \ar{d}[swap]{\sigma^1} \ar{r}{\otimes} &
	\Sigma^{\wr 2} \wr \mathcal{V}^{op} \ar{r}{\otimes} &
	|[alias=UUUU]|
	\Sigma \wr \mathcal{V}^{op} \ar{r}{\otimes} \ar[equal]{d} &
	\mathcal{V}^{op} \ar[equal]{d}
	\\
	\Omega^1_{\mathfrak{C}} \wr A \ar{r} \ar{d}[swap]{d_0} &
	|[alias=VVV]|
	\Sigma \wr \Omega^0_{\mathfrak{C}} \wr A \ar{rr} & &
	|[alias=U]|
	\Sigma^{\wr 2} \wr A \ar{r} \ar{d}[swap]{\sigma^0} &
	|[alias=VVVV]|
	\Sigma^{\wr 2} \wr \mathcal{V}^{op} \ar{rr}{\otimes} \ar{d}[swap]{\sigma^0} & &
	\Sigma \wr \mathcal{V}^{op} \ar{r}{\otimes} &
	|[alias=UU]|
	\mathcal{V}^{op} \ar[equal]{d}
	\\
	|[alias=V]|
	\Omega^0_{\mathfrak{C}} \wr A \ar{rrr} & & &
	\Sigma \wr A \ar{r} &
	|[alias=VV]|
	\Sigma \wr \mathcal{V}^{op} \ar{rrr}{\otimes} & & &
	\mathcal{V}^{op}
	\arrow[Leftrightarrow, from=V, to=U,shorten >=0.15cm,shorten <=0.15cm
	,swap,"\pi"
	]
	\arrow[Leftrightarrow, from=VV, to=UU,shorten >=0.15cm,shorten <=0.15cm
	]
	\arrow[Leftrightarrow, from=VVV, to=UUU,shorten >=0.15cm,shorten <=0.15cm
	,swap,"\pi"
	]
	\arrow[Leftrightarrow, from=VVVV, to=UUUU,shorten >=0.15cm,shorten <=0.15cm
	]
	\end{tikzcd}
	\]
	That the leftmost sections of these diagrams match follows by 
	parts (IT1) and (FF1) of Proposition \ref{CATDIAG2 PROP},
	while the fact that the rightmost sections match follows since
	$\mathcal{V}$ is a monoidal category.
	
	Unitality of $N$
	follows by a simpler version of the argument above,
	cf. \cite[(4.21)(4.22)]{BP21}.
\end{proof}

\subsubsection*{The fibered monad on colored symmetric sequences}

We will now use the fibered adjunction
\[
\mathsf{Lan} \colon
\mathsf{WSpan}^l(\Sigma_{\bullet}^{op},\mathcal{V}) 
\rightleftarrows
\mathsf{Sym}_{\bullet}(\mathcal{V})
\colon \upsilon
\]
from Remark \ref{LANADJ REM} to induce a fibered monad on 
$\mathsf{Sym}_{\bullet}(\mathcal{V})$.
To do so, we will verify the conditions in \cite[Prop. 2.27]{BP21},
requiring that the natural transformations
\[
\mathsf{Lan} \upsilon \xrightarrow{\epsilon} id
\qquad
\mathsf{Lan} N \xrightarrow{\eta} \mathsf{Lan} N \upsilon \mathsf{Lan}
\]
are natural isomorphisms.
This is clear for $\epsilon$ while for $\eta$ it follows from the following two lemmas, the first of which is proven exactly as in \cite[Lemma 2.21]{BP21}.

\begin{lemma}[{cf. \cite[Lemma 2.21]{BP21}}]
	\label{FINWRPRODLIM LEM}
	If in $\mathcal{V}$ the monoidal product 
	commutes with colimits in each variable, and the leftmost diagram
	\begin{equation}\label{WRLAN EQ}
	\begin{tikzcd}[column sep = 4.5em]
	\mathcal{C} \ar{r}[swap,name=F]{}{F} \ar{d}[swap]{k} & 
	\mathcal{V}^{op} 
& 
	\Sigma \wr \mathcal{C} \ar{d}[swap]{\Sigma \wr k} 
	\ar{r}[swap,name=FF]{}{\Sigma \wr F} & 
	\Sigma \wr \mathcal{V}^{op} \ar{r}{\otimes} &
	\mathcal{V}^{op}
\\
	|[alias=D]|\mathcal{D} \ar{ru}[swap]{H} &
& 
	|[alias=FD]|\Sigma \wr \mathcal{D}
	\ar{ru}[swap]{\Sigma \wr H}
	\ar[bend right=13]{rru}[swap]{\otimes \circ (\Sigma \wr H)}
	&
	\arrow[Rightarrow, from=D, to=F,shorten <=0.10cm]
	\arrow[Rightarrow, from=FD, to=FF,shorten <=0.10cm]
	\end{tikzcd}
	\end{equation}
	is a right Kan extension diagram,
	then so is the composite of the rightmost diagram. 
\end{lemma}

\begin{lemma}[{cf. \cite[Lemma 4.27]{BP21}}]
	\label{LANPULLCOMA LEM}
	Suppose that $\mathcal{V}$ is complete. If the rightmost triangle in 
	\[
	\begin{tikzcd}
	\Omega_{\mathfrak{C}}^{0} \wr A \ar{r}{\boldsymbol{V}^0} 
	\ar{d} & 
	\Sigma \wr A  
	\ar{d}  \ar{r}[swap,name=F]{}&
	\mathcal{V}^{op}
	\\
	\Omega_{\mathfrak{C}}^{0} \ar{r}[swap]{\boldsymbol{V}^0} & 
	|[alias=FEG]|\Sigma \wr \Sigma_{\mathfrak{C}} \ar{ru}
	\arrow[Rightarrow, from=FEG, to=F,shorten <=0.15cm]
	\end{tikzcd}
	\]
	is a right Kan extension diagram then so is the composite diagram.
\end{lemma}

Our proof Lemma \ref{LANPULLCOMA LEM} will be a more formalized version of the proof in \cite[Lemma 4.27]{BP21}.
For 
$\pi \colon \mathcal{E} \to \mathcal{B}$
a Grothendieck fibration and $e \in \mathcal{E}$,
we write $e \downarrow_{\mathcal{B}} \mathcal{E}$
for the subcategory of the undercategory
$e \downarrow \mathcal{E}$
consisting of the objects and maps
over $id_{\pi(b)}$.
Pullbacks over $\mathcal{B}$ 
then provide a retraction 
$r \colon e \downarrow \mathcal{E} \to
e \downarrow_{\mathcal{B}} \mathcal{E}$,
which is a right adjoint to the inclusion
$e \downarrow_{\mathcal{B}} \mathcal{E} 
\hookrightarrow
e \downarrow \mathcal{E}$.
In other words, 
$e \downarrow_{\mathcal{B}} \mathcal{E} $
is a coreflexive subcategory of 
$e \downarrow \mathcal{E} $,
so that the inclusion
$e \downarrow_{\mathcal{B}} \mathcal{E} 
\hookrightarrow
e \downarrow \mathcal{E}$
is final.

\begin{proof}
	Firstly, note that the composite
	\begin{equation}\label{COMPISO EQ}
	\begin{tikzcd}
	\vect{T} \downarrow \Omega^0_{\mathfrak{C}} 
	\ar{r} & 
	\left( \vect{T}_v \right)_{\boldsymbol{V}(T)} \downarrow \Sigma \wr \Sigma_{\mathfrak{C}}
	\ar{r}{r} & 
	\left( \vect{T}_v \right)_{\boldsymbol{V}(T)} \downarrow_{\Sigma} \Sigma \wr \Sigma_{\mathfrak{C}}
	\end{tikzcd}
	\end{equation}
	is an isomorphism. 
	Indeed, the objects of 
	$\vect{T} \downarrow \Omega^0_{\mathfrak{C}}$
	are determined by underlying isomorphisms 
	$f \colon T \xrightarrow{\simeq} T'$ in $\Omega$,
	which are in turn determined by 
	a tuple of isomorphisms of vertices 
	$f_v \colon T_v \xrightarrow{\simeq} T'_v$ in $\Sigma$
	for each $v \in \boldsymbol{V}(T)$,
	cf. \cite[Prop. 3.12]{BP21}.
	We now claim that the maps
	\begin{equation}\label{COMPISO2 EQ}
	\begin{tikzcd}
	\left(\vect{T},(a_v)_{\boldsymbol{V}(T)}\right) \downarrow \Omega^0_{\mathfrak{C}} \wr A
	\ar{r} &  
	\left( a_v \right)_{\boldsymbol{V}(T)} \downarrow \Sigma \wr A
	\ar{r}{r} &  
	\left( a_v \right)_{\boldsymbol{V}(T)} \downarrow_{\Sigma} \Sigma \wr A
	\end{tikzcd}
	\end{equation}
	are likewise isomorphisms.
	To see this, we first write
	$D$, $\bar{D}$ for the composite functors in \eqref{COMPISO EQ},\eqref{COMPISO2 EQ},
	and $\rho \colon A \to \Sigma_{\mathfrak{C}}$
	for the given map.
	An object in the target of \eqref{COMPISO2 EQ}
	is a tuple
	$\bar{f}_v \colon a_v \to b_v$ of maps in $A$ for $v \in \boldsymbol{V}(T)$.
	Writing $\rho (\bar{f}_v) \colon  \rho (a_v) \to \rho( b_v)$
	as
	$f_v \colon \vect{T}_v \to \vect{T'_v}$
	and
	$D^{-1}\left(\rho (\bar{f}_v)_{\boldsymbol{V}(T)}\right)$
	as 
	$f \colon \vect{T} \to \vect{T'}$
	one then has
	\[
	\bar{D}^{-1}
	\left(( a_v \xrightarrow{\bar{f}_v} b_v)_{v \in \boldsymbol{V}(T)}\right)=
	\left(
	\vect{T} \xrightarrow{f} \vect{T'},
	(a_v)_{v \in \boldsymbol{V}(T)} \to 
	(b_w)_{w \in \boldsymbol{V}(T')}
	\right)
	\]
	where we note that the map 
	$(a_v)_{v \in \boldsymbol{V}(T)} \to 
	(b_w)_{w \in \boldsymbol{V}(T')}$
	involves a permutation of tuples induced by the isomorphism
	$\boldsymbol{V}(T) \simeq \boldsymbol{V}(T')$.
	Now consider the diagram 
	\begin{equation}\label{COMPISO3 EQ}
	\begin{tikzcd}
	\left(\vect{T},(a_v)\right)
	\ar{r} & 
	\left( a_v \right) \downarrow \Sigma \wr A
	\ar{r}{r} &
	\left(a_v \right) \downarrow_{\Sigma} \Sigma \wr A
	\ar{r} &
	\left(a_v \right) \downarrow \Sigma \wr A.
	\end{tikzcd}
	\end{equation}
	To finish the proof,
	we show that the first map in \eqref{COMPISO3 EQ} is final. 
	Since this first map is naturally isomorphic to the full composite in \eqref{COMPISO3 EQ}, we need only show that the latter is final.
	But this follows since \eqref{COMPISO2 EQ} is an isomorphism and the last map in \eqref{COMPISO3 EQ} is known to be final.
\end{proof}

Since Lemmas \ref{FINWRPRODLIM LEM},\ref{LANPULLCOMA LEM}
verify the hypotheses of \cite[Prop. 2.27]{BP21},
we can finally complete Definition \ref{FREEOP DEF}
by describing the monad structure on $\mathbb{F}$.

\begin{definition}\label{COLORMON_DEF}
	The fibered free operad monad $\mathbb{F}$
	on $\mathsf{Sym}_{\bullet}(\mathcal{V})$
	has underlying functor
	$\mathbb{F} = \mathsf{Lan} N \upsilon$ and multiplication and unit given by
	\[
	\mathsf{Lan} N \upsilon \mathsf{Lan} N \upsilon \xleftarrow{\simeq} 
	\mathsf{Lan} N N \upsilon \to 
	\mathsf{Lan} N \upsilon
	\qquad
	id \xleftarrow{\simeq} 
	\mathsf{Lan} \upsilon \to
	\mathsf{Lan} N \upsilon.
	\]
\end{definition}

\subsection{Free extensions of operads}\label{PUSHOUT_SEC}

Our overall goal in this section is to prove Lemma \ref{OURE LEM},
which allows us to understand free operad extensions,
i.e. pushouts of the form 
\begin{equation}\label{OU EQ}
\begin{tikzcd}
\mathbb F X \arrow[d, "\mathbb{F}u"'] \arrow[r]
&
\O \arrow[d]
\\
\mathbb F Y \arrow[r]
&
\O[u].
\end{tikzcd}
\end{equation}
where $u \colon X \to Y$ is a map of symmetric sequences,
and where we moreover require that \eqref{OU EQ} is a fibered diagram over $\mathsf{Set}$, i.e. that all maps therein are the identity on color sets.

Moreover, in order to understand the equivariant case, 
we will \emph{not} fix the set of colors,  
but rather consider all colors simultaneously, 
and note that our constructions are natural 
on the diagrams \eqref{OU EQ} with respect to change of colors.
More explicitly, this means that the work in this section will be natural with regard to commutative diagrams
\begin{equation}\label{COLORCHNAT EQ}
\begin{tikzcd}
X \arrow[r, "u"',swap] \arrow[d]
&
Y \arrow[d]
&
\mathbb F X \arrow[d] \arrow[r]
&
\O \arrow[d]
\\
X' \arrow[r, "u'"']
&
Y'
&
\mathbb F X' \arrow[r]
&
\O'
\end{tikzcd}
\end{equation}
where all vertical maps induce the same map 
$\varphi \colon \mathfrak{C} \to \mathfrak{D}$ on objects.

To understand the pushouts \eqref{OU EQ},
we will produce a filtration
\begin{equation}\label{FILT EQ}
\O = \O_0 \into \O_1 \into \O_2 \into \dots \into \colim_k \O_k = \O[u]
\end{equation}
of the underlying symmetric sequences, i.e. with 
$\mathcal{O}_i \in \mathsf{Sym}_{\bullet}(\mathcal{V})$
(moreover, all maps in \eqref{FILT EQ} will, again, be the identity on colors).

Writing $\amalg_{\mathsf{Set}}$ and $\mathbin{\check\amalg}_{\mathsf{Set}}$
for the fibered coproducts in 
$\mathsf{Sym}_{\bullet}(\mathcal{V})$ and
$\mathsf{Op}_{\bullet}(\mathcal{V})$
(i.e. these are the coproducts within each fixed color fiber over $\mathsf{Set}$, rather than the coproducts in the overall categories),
the discussion in $(5.3)$ through $(5.7)$ of \cite{BP21}
yields that
\begin{align*}
\O[u]
&
\simeq \mathrm{coeq}\left(
\O \mathbin{\check\amalg_{\mathsf{Set}}} \mathbb F X \mathbin{\check\amalg}_{\mathsf{Set}} \mathbb F Y \rightrightarrows \O \mathbin{\check\amalg}_{\mathsf{Set}} \mathbb F Y
\right)
\\
&
\simeq \colim_{[l] \in \Delta^{op}} 
B_l \left( \O, \mathbb F X, \mathbb F X, \mathbb F X, \mathbb FY \right)
\\
&
\simeq \colim_{[l] \in \Delta^{op},[n] \in \Delta^{op}} 
B_l \left( \mathbb F^{\circ n+1} \O, \mathbb F X, \mathbb F X, \mathbb F X, \mathbb FY \right)
\\
&
\simeq \colim_{[l] \in \Delta^{op},[n] \in \Delta^{op}} 
\mathsf{Lan} N \circ \left( N^{\circ n} \upsilon \O \amalg_{\mathsf{Set}} \upsilon X^{\amalg_{\mathsf{Set}} 2l +1} \amalg_{\mathsf{Set}} \upsilon Y \right),
\stepcounter{equation}\tag{\theequation}\label{OU EQ1}
\end{align*}
where $B_{\bullet}$ denotes the \textit{double bar construction}
with respect to $\mathbin{\check\amalg}_{\mathsf{Set}}$,
$\mathbb{F}^{\bullet +1} \mathcal{O}$ denotes the simplicial resolution of $\mathcal{O}$, 
and $N$ is the monad on spans in Definition \ref{NCOLOR DEF}.
Crucially, we note that colimits over $\Delta^{op}$
are computed by the reflexive coequalizer determined by levels $0$ and $1$, 
so that the colimits in \eqref{OU EQ1}
can be computed in $\mathsf{Sym}_{\bullet}(\mathcal{V})$
rather than in $\mathsf{Op}_{\bullet}(\mathcal{V})$.

By construction,
$N \left(N^{\circ n} \upsilon \O \amalg_{\mathsf{Set}} \upsilon X^{\amalg_{\mathsf{Set}} 2l +1}\amalg_{\mathsf{Set}} \upsilon Y \right)$
denotes a certain span
$\Sigma_{\mathfrak{C}}^{op} \leftarrow 
\left(\Omega^{n,\lambda_l}_{\mathfrak{C}}\right)^{op} \to \mathcal{V}$
which we will explicitly identify 
in \S \ref{LCS_SEC}.
This will then allows us to 
apply (the natural analogue of)
\cite[Prop. 5.42]{BP21}
along each simplicial direction
to convert the last line of \eqref{OU EQ1}
into a $\mathsf{Lan}$
over a single span
$\Sigma_{\mathfrak{C}}^{op} \leftarrow 
\left|\Omega^{n,\lambda_l}_{\mathfrak{C}}\right|^{op} \to \mathcal{V}$.

The task of describing 
$\Omega^{n,\lambda_l}_{\mathfrak{C}}$
is similar to the spirit of Proposition \ref{ASSOCIDS PROP}, which shows that
$N^{\circ n+1}$ is naturally calculated using the
$\Omega_{\mathfrak{C}}^{n} \wr (-)$ 
construction.

In practice, we will prefer to describe a slightly more general variant of the $\Omega^{n,\lambda_l}_{\mathfrak{C}}$ categories.
For $\lambda = \lambda_a \amalg \lambda_i$
a partition of $\set{1,2,\dots,l}$,
we write 
$N^{\times \lambda}$
for the monad (cf. \cite[\S 2.3]{BP21}) on 
$\left(\mathsf{WSpan}_l(\Sigma_{\bullet}^{op},\mathcal{V})\right)^{\times l}$
given by
\begin{equation}\label{NLAMBMON EQ}
\left(N^{\times \lambda} (A_j)\right)_k = 
\begin{cases}
N(A_k) & \text{if } k\in \lambda_a
\\
A_k & \text{if } k\in \lambda_i
\end{cases}
\end{equation}
Note that, writing 
$\langle\langle l\rangle\rangle = 
\{-\infty,-l,\cdots,0,\cdots,l,\infty\}$
and $\lambda_l$ for the partition on 
$\langle\langle l\rangle\rangle$ with
$\left(\lambda_l\right)_{a}=\{-\infty\}$,
the last term in \eqref{OU EQ1} is then
$\mathsf{Lan}N\coprod_{\mathsf{Set}}
\left(N^{\lambda_l}\right)^{\circ n}
(\upsilon\O,\upsilon X,\cdots,\upsilon X,\upsilon Y)$.

We note that $N^{\times \lambda}$
preserves the fibered product
$\left(\mathsf{WSpan}_l(\Sigma_{\bullet}^{op},\mathcal{V})\right)^{\times_{\mathsf{Set}} l}$,
i.e. the subcategory of those tuples $(A_j)$
where all $A_j$ have the same color set $\mathfrak{C} \in \mathsf{Set}$ (and likewise for maps), 
and we abuse notation by also writing 
$N^{\times \lambda}$
for the monad restricted to this subcategory.

Out next task is to understand the composites
$N\coprod_{\mathsf{Set}}
\left(N^{\lambda}\right)^{\circ n}$.

\subsubsection*{Labeled colored strings}
\label{LCS_SEC}

Let $l \geq 1$. A \emph{$l$-labeling} of a tree $\vect{T} \in \Omega_{\mathfrak{C}}$ is a map
$\boldsymbol{V}(T) \to \{1,\cdots,l\}$.
Further, a map 
$\vect{T} \to \vect{S}$ of labeled trees
is called a \emph{label map}
if, for all $v \in \boldsymbol{V}(T)$,
all the vertices in $\vect{S}_v$ (Notation \ref{SVNOT NOT})
have the same label as $v$.
Lastly, given a subset $\lambda_i \subseteq \{1,\cdots,l\}$,
a label map 
$\vect{T} \to \vect{S}$
is called \emph{$\lambda_i$-inert}
if $\vect{S}_v$ is a corolla whenever the label of
$v \in \boldsymbol{V}(T)$ is in $\lambda_i$.

The categories $\Omega_{\mathfrak C}^{n,s,\lambda}$
defined below will represent
the functors
$N^{\circ s+1} \circ \coprod_{\mathsf{Set}} \circ \left(N^{\times \lambda}\right)^{\circ n-s}$.

\begin{definition}[{cf. \cite[Def. 5.11]{BP21}}]\label{CLPS DEF}
	Given $-1 \leq s \leq n$, $l \geq 0$, and a partition $\lambda = \lambda_a \amalg \lambda_i$ of $\set{1,2,\dots,l}$,
	define $\Omega_{\mathfrak C}^{n,s,\lambda}$ to have as objects
	$n$-planar strings
	\begin{equation}
	\mathsf{lr}(\vect{T}_0)=
	\vect{T}_{-1} \xrightarrow{f_0} \vect{T}_0 
	\xrightarrow{f_1} \vect{T}_1 
	\xrightarrow{f_2} \dots
	\vect{T}_{s} \xrightarrow{f_{s+1}} \vect{T}_{s+1}
	\xrightarrow{f_{s+2}}  \dots
	\xrightarrow{f_n} \vect{T}_n
	\end{equation}
	together with $l$-labelings of $\vect{T}_s, \vect{T}_{s+1}, \cdots, \vect{T}_n$,
	such that
	$f_{r}, r>s$ are $\lambda_i$-inert label maps.
	
	Arrows in $\Omega_{\mathfrak C}^{n,s,\lambda}$
	are tuples of isomorphisms 
	$\left(\rho_r \colon \vect{T}_r \to \vect{T}'_r\right)$
	such that $\rho_r,r \geq s$ are label maps.
	
	Further, for any $s<0$ or $n<s'$, we write
	\[
	\Omega_{\mathfrak{C}}^{n,s,\lambda} = \Omega_{\mathfrak{C}}^{n,-1,\lambda},
	\qquad
	\Omega_{\mathfrak{C}}^{n,s',\lambda} = \Omega_{\mathfrak{C}}^{n}.
	\]
\end{definition}

We now discuss the functors relating the $\Omega_{\mathfrak{C}}^{n,s,\lambda}$ categories. Firstly, for 
$s \leq s'$ 
and map of labels 
$\gamma \colon \{1,\cdots,l'\} \to \{1,\cdots,l\}$
such that $\lambda'_a \subseteq \gamma^{-1}\left( \lambda_a\right)$
there are natural functors
\[
\Omega_{\mathfrak{C}}^{n,s,\lambda} \to \Omega_{\mathfrak{C}}^{n,s',\lambda},
\qquad
\Omega_{\mathfrak{C}}^{n,s,\lambda'} \xrightarrow{\gamma} \Omega_{\mathfrak{C}}^{n,s,\lambda}.
\]
Second, by keeping track of labels on vertices,
the functors from \S \ref{CSTRINGS_SEC} relating the categories 
$\Omega^n_{\mathfrak{C}}$ extend to the categories
$\Omega_{\mathfrak{C}}^{n,s,\lambda}$. Indeed, for 
$k \leq n$
and 
$\varphi \colon \mathfrak{C} \to \mathfrak{D}$ a map of colors
one has functors
\begin{equation}\label{FGTLABEL EQ}
\Omega_{\mathfrak{C}}^{n,s,\lambda} \xrightarrow{\boldsymbol{V}^k} \Sigma \wr\Omega_{\mathfrak{C}}^{n-k-1,s-k-1,\lambda},
\qquad
\Omega_{\mathfrak{C}}^{n,s,\lambda} \xrightarrow{\varphi} \Omega_{\mathfrak{D}}^{n,s,\lambda}.
\end{equation}

Lastly, one also has simplicial operators $d_i$, $s_j$, 
but some care is needed with the way these interact with the index $s$. To do so, defining functions $d_i,s_j\colon \mathbb{Z} \to \mathbb{Z}$ by
\begin{equation}\label{SIMPLEXP EQ}
d_i(s) = 
\begin{cases}
s-1, & i<s
\\
s, & s\leq i
\end{cases}
\qquad
s_j(s) = 
\begin{cases}
s+1, & j<s
\\
s, & s\leq j
\end{cases}
\end{equation}
one has simplicial operators
\[
\Omega_{\mathfrak{C}}^{n,s,\lambda} \xrightarrow{d_i} \Sigma \wr\Omega_{\mathfrak{C}}^{n,d_i(s),\lambda},
\qquad
\Omega_{\mathfrak{C}}^{n,s,\lambda} \xrightarrow{s_j} \Sigma \wr\Omega_{\mathfrak{C}}^{n,s_j(s),\lambda},
\]
for $0\leq i \leq n$ and $-1\leq j \leq n$.
In practice, we will prefer to suppress $s$ from the notation,
and write 
$\Omega_{\mathfrak{C}}^{n,\bullet,\lambda}$ to denote the string of categories 
$\Omega_{\mathfrak{C}}^{n,s,\lambda}$ as a whole.
Lastly, the $\pi_{i,k}$ natural isomorphisms for $i<k$ from Proposition \ref{CATDIAG PROP}
generalize to natural isomorphisms
\begin{equation}
\begin{tikzcd}[row sep = tiny, column sep = 35pt]
\Omega_{\mathfrak{C}}^{n,s,\lambda}
\arrow{r}{\boldsymbol{V}^k} \arrow{dd}[swap]{d_i} &
|[alias=U]|
\Sigma \wr \Omega_{\mathfrak{C}}^{n-k-1,s-k-1,\lambda}
\ar[equal]{dd}{}
\\
\\
|[alias=V]|
\Omega_{\mathfrak{C}}^{n-1,d_i(s),\lambda} \arrow{r}[swap]{\boldsymbol{V}^{k-1}} &
\Sigma \wr \Omega_{\mathfrak{C}}^{n-k-1,d_i(s)-k,\lambda}
\arrow[Leftrightarrow, from=V, to=U,shorten >=0.15cm,shorten <=0.15cm
,swap,"\pi_{i,k}"
]
\end{tikzcd}
\end{equation}
(note that the right vertical map is an identity even if
$s-k-1 \neq d_i(s)-k$, since that can only occur if $s\leq i \leq k$, implying that the rightmost terms are both $\Sigma \wr \Omega_{\mathfrak{C}}^{n-k-1,-1,\lambda}$).

\begin{remark}
	We now discuss the naturality of the given functors on the categories
	$\Omega_{\mathfrak{C}}^{n,s,\lambda}$ just described.
	\begin{enumerate}[label=(\roman*)]
		\item by keeping track of vertex labels, all the analogues of the properties in Propositions \ref{CATDIAG PROP} and \ref{CATDIAG2 PROP} extend (note that this includes the pullback claims in 
		Proposition \ref{CATDIAG PROP}(ii)).
		\item the change of color functors $\varphi$, change of label functors $\gamma$, and the forgetful functors in
		\eqref{FGTLABEL EQ} are all natural with respect to each other.
		\item $d_i$, $s_j$, $\boldsymbol{V}^k$
		and $\pi_{i,k}$, are natural with respect to the change of color functors $\varphi$, change of label functors $\gamma$, and the forgetful functors in
		\eqref{FGTLABEL EQ}, 
		in the sense that they satisfy the analogues of  
		Proposition \ref{CATDIAG PROP}(iii) with 
		the role of $\varphi$ replaced with the latter functors.
		\item
		For $k \leq s \leq s'$ the following squares are pullback squares
		\[
		\begin{tikzcd}[column sep = small, row sep = small]
		\Omega^{n,s,\lambda}_{\mathfrak{C}} \ar{r}{\boldsymbol{V}^k} \ar{dd} &
		\Sigma \wr \Omega^{n-k-1,s-k-1,\lambda}_{\mathfrak{C}} \ar{dd}
		\\
		\\
		\Omega^{n,s',\lambda}_{\mathfrak{C}} \ar{r}[swap]{\boldsymbol{V}^k} &
		\Sigma \wr \Omega^{n-k-1,s'-k-1,\lambda}_{\mathfrak{C}}
		\end{tikzcd}
		\]
	\end{enumerate}
\end{remark}

The following is the main purpose of the 
$\Omega_{\mathfrak{C}}^{n,s,\lambda}$ categories,
adapting the work in \S \ref{WRACONST SEC}.

\begin{definition}[{cf. \cite[Not. 5.25]{BP21}}]
        \label{NA_DEF}
	Given a $l$-tuple of functors
	$\left(A_j \to \Sigma_{\mathfrak C} \right)_{1\leq j \leq l}$,
	we write
	\begin{equation}\label{WRAJDEF EQ}
	(-) \wr (A_j) \colon 
	\mathsf{Cat} \downarrow^r_{\Sigma} \Sigma \wr \Sigma_{\mathfrak{C}}^{\amalg l}
	\to
	\mathsf{Cat} \downarrow^r_{\Sigma} \Sigma \wr \amalg_j A_j
	\end{equation}
	for the pullback \eqref{WSPANPULL EQ} for the map
	$\Sigma \wr \amalg_j A_j \to \Sigma \wr \Sigma_{\mathfrak{C}}^{\amalg l}$.
	
	In particular, for all $-1\leq s \leq n$, this defines categories
	$\Omega^{n,s,\lambda}_{\mathfrak{C}} \wr (A_j)$ via pullbacks
	(note that the $s \leq n$ restriction guarantees that the target of the lower $\boldsymbol{V}^n$ is indeed $\Sigma \wr \Sigma_{\mathfrak{C}}^{\amalg l}$)
	\begin{equation}\label{WRAJSAMPLE EQ}
	\begin{tikzcd}
	\Omega^{n,s,\lambda}_{\mathfrak{C}} \wr (A_j) \ar{r}{\boldsymbol{V}^n} \ar{d} &
	\Sigma \wr \amalg_j A_j  \ar{d}
	\\
	\Omega^{n,s,\lambda}_{\mathfrak{C}} \ar{r}{\boldsymbol{V}^n} &
	\Sigma \wr \Sigma_{\mathfrak{C}}^{\amalg l}
	\end{tikzcd}
	\end{equation}
	along with analogues of $d_i$ (for $i<n$), $s_j$, $\boldsymbol{V}^k$, $\pi_{i,k}$
	and of the forgetful functors in \eqref{FGTLABEL EQ}
	(cf. the discussion following \eqref{WRADEF EQ}).
\end{definition}

\begin{proposition}\label{SPANPIECEJ PROP}
	A tuple of commutative squares
	\begin{equation}\label{SPANPIECEJ EQ}
	\begin{tikzcd}[row sep = 6pt]
	A_j \ar{d} \ar{r}{\varphi} &  \ar{d} B_j
	\\
	\Sigma_{\mathfrak{C}} \ar{r}[swap]{\varphi} & \Sigma_{\mathfrak{D}}
	\end{tikzcd}
	\end{equation}
	induces natural maps 
	$\varphi \colon
	\Omega_{\mathfrak{C}}^{n,\bullet,\lambda} \wr (A_j) \to 
	\Omega_{\mathfrak{D}}^{n,\bullet,\lambda} \wr (B_j) $.
	
	Similarly, a map of tuples $A_j \to B_{g(j)}$ for 
	$\gamma \colon \{1,\cdots,l\} \to \{1,\cdots,l'\}$
	induces natural maps 
	$\gamma \colon
	\Omega_{\mathfrak{C}}^{n,\bullet,\lambda} \wr (A_j) \to 
	\Omega_{\mathfrak{C}}^{n,\bullet,\lambda'} \wr (B_{j'}) $.

	Moreover, both $\varphi$ and $\gamma$ satisfy the analogues of the commutativity properties in Proposition \ref{SPANPIECE PROP}.
	In particular, the diagram below commutes.
	\[
	\begin{tikzcd}[column sep = 5pt, row sep = 4pt]
	\Omega^{n,\bullet,\lambda}_{\mathfrak{C}} \wr (A_j)
	\ar{rrrrr}[name=toE2,near end]{\boldsymbol{V}^k} \ar{rd}[swap]{d_i} \ar{dd}[swap]{\varphi \gamma}
	&&&
	&&
	\Sigma \wr \Omega^{n-k-1,\bullet,\lambda}_{\mathfrak{C}} \wr (A_j) \ar{dd}{\varphi \gamma}
	\\
	&
	|[alias=DBE2]|
	\Omega^{n-1,\bullet,\lambda}_{\mathfrak{C}} \wr (A_j) \ar{rrrru}[swap]{\boldsymbol{V}^{k-1}}
	\\
	\Omega^{n,\bullet,\lambda'}_{\mathfrak{D}} \wr (B_{j'}) \ar{rrrrr}[name=toB2, near end]{\boldsymbol{V}^k} \ar{rd}[swap]{d_i}
	&&&
	&&
	\Sigma \wr \Omega^{n-k-1,\bullet,\lambda'}_{\mathfrak{D}} \wr (B_{j'})
	\\
	&
	|[alias=D2]| \Omega^{n-1,\bullet,\lambda'}_{\mathfrak{D}} \wr (B_{j'}) \ar{rrrru}[swap]{\boldsymbol{V}^{k-1}}
	\arrow[Leftrightarrow, from=DBE2, to=toE2, shorten <=0.1cm,shorten >=0.15cm
	,swap,near end,"\pi"
	]
	\arrow[Leftrightarrow, from=D2, to=toB2, shorten <=0.1cm,shorten >=0.15cm,swap,near end,"\pi"]
	\arrow[from=DBE2, to=D2, crossing over, near start, swap, "\varphi \gamma"]
	\end{tikzcd}
	\]
\end{proposition}

\begin{proof}
	This follows by repeating the argument in the proof of Proposition \ref{SPANPIECE PROP}.
\end{proof}

Using the composite functors
$\Omega_{\mathfrak{C}}^{n,s,\lambda} \wr (A_j)
\to \Omega_{\mathfrak{C}}^{n,s,\lambda} 
\to \Omega^{-1,0}_{\mathfrak{C}} = \Sigma_{\mathfrak{C}}$,
we can regard the 
$\Omega_{\mathfrak{C}}^{n,s,\lambda} \wr (-)$
construction as a functor
$\left(\mathsf{Cat}\downarrow \Sigma_{\mathfrak{C}}\right)^{\times l}
\to \mathsf{Cat}\downarrow \Sigma_{\mathfrak{C}}$.
In the following, 
$(\OC^k)^{\times \lambda}$
denotes the tuple given by 
$\OC^k$ for entries in $\lambda_a$ and by
$\Sigma_{\mathfrak{C}}$ for entries in $\lambda_i$,
and $(\OC^k)^{\times \lambda} \wr (A_j)$ is computed entrywise.

\begin{corollary}[{cf. \cite[Cor. 5.35]{BP21}}]
	\label{LABIDEN_COR}
	Let $-1 \leq k, -1 \leq s \leq n$.
	There are natural identifications
	\[
	\OC^k \wr \OC^{n,s,\lambda} \wr (A_j) \simeq
	\OC^{n+k+1,s+k+1,\lambda} \wr (A_j),
	\qquad
	\OC^{n,s,\lambda} \wr (\OC^k)^{\times \lambda} \wr (A_j) \simeq
	\OC^{n+k+1,s,\lambda} \wr (A_j)	
	\]
	which are unital and associative in the natural ways.
	Moreover, these induce identifications
	\[
	d_i \wr \Omega^{n,s,\lambda} \wr (A_j) \simeq d_i \wr (A_j)
	\quad
	\pi_{i,k} \wr \Omega^{n,s,\lambda} \wr (A_j) \simeq \pi_{i,k} \wr (A_j)
	\quad
	s_j \wr \Omega^{n,s,\lambda} \wr (A_j) \wr A \simeq s_j \wr  (A_j)
	\]
	\[
	\Omega^k \wr (d_i) \wr (A_j) \simeq d_{k+i+1} \wr (A_j)
	\quad
	\Omega^k \wr (s_j) \wr (A_j) \simeq s_{k+j+1} \wr (A_j)
	\]
\end{corollary}

\begin{proof}
	Much as in Proposition \ref{ASSOCIDS PROP}, this follows by noting that all squares in the following diagrams are pullback squares.
\begin{equation}\label{LSTRINGS EQ}
\begin{tikzcd}[column sep = 11pt, row sep = 9pt]
	\OC^{n+k+1,s+k+1,\lambda} \arrow{r}{\boldsymbol{V}^k} \wr (A_j)
	\ar{d}
	&
	\Sigma \wr \OC^{n,s,\lambda} \wr (A_j) \ar{r}{\boldsymbol{V}^n} \ar{d}
	&
	\Sigma^{\wr 2} \wr \amalg_j A_j \arrow{r} \ar{d}
	&
	\Sigma \wr \amalg_j A_j \ar{d}
	\\
	\OC^{n+k+1,s+k+1,\lambda} \arrow{r}{\boldsymbol{V}^k} \arrow[d]
	&
	\Sigma \wr \OC^{n,s,\lambda} \arrow{r}{\boldsymbol{V}^n} \arrow[d]
	&
	\Sigma^{\wr 2} \wr \Sigma_{\mathfrak C}^{\amalg l} \ar{r}
	&
	\Sigma \wr \Sigma_{\mathfrak C}^{\amalg l}
	\\
	\OC^k \arrow{r}{\boldsymbol{V}^k}
	&
	\Sigma \wr \Sigma_{\mathfrak C}
	\end{tikzcd}
\end{equation}	
\begin{equation}
\begin{tikzcd}[column sep = 11pt, row sep = 9pt]
	\OC^{n+k+1,s,\lambda} \arrow{r}{\boldsymbol{V}^n} \wr (A_j)
	\ar{d}
	&
	\Sigma \wr \amalg \left(\OC^{k}\right)^{\times \lambda} \wr (A_j)
	\ar{r}{\boldsymbol{V}^k} \ar{d}
	&
	\Sigma \wr \amalg_j \Sigma \wr A_j \arrow{r} \ar{d}
	&
	\Sigma^{\wr 2} \wr \amalg_j A_j \arrow{r} \ar{d}
	&
	\Sigma \wr \amalg_j A_j \ar{d}
	\\
	\OC^{n+k+1,s,\lambda} \arrow{r}{\boldsymbol{V}^n} \arrow[d]
	&
	\Sigma \wr \amalg \left(\OC^{k}\right)^{\times \lambda} \arrow{r}{\boldsymbol{V}^k} \arrow[d]
	&
	\Sigma \wr \amalg_l \Sigma \wr \Sigma_{\mathfrak{C}} \ar{r}
	&
	\Sigma^{\wr 2} \wr \amalg_l \Sigma_{\mathfrak{C}} \ar{r}
	&
	\Sigma \wr \Sigma_{\mathfrak C}^{\amalg l}
	\\
	\OC^{n,s,\lambda} \arrow{r}{\boldsymbol{V}^n}
	&
	\Sigma \wr \Sigma_{\mathfrak C}^{\amalg l}
\end{tikzcd}
\end{equation}
\end{proof}

\subsubsection*{Filtration of free extensions}
\label{EQMON_SEC}

We now return to proving Lemma \ref{OURE LEM},
describing the free extensions \eqref{OU EQ}.
As discussed after \eqref{NLAMBMON EQ},
let $\lambda_l$ denote the partition on 
\[
\langle \langle l \rangle \rangle
=
\{-\infty,-l,\cdots,-1,0,1,\cdots,+\infty\}
\]
such that $\left(\lambda_l\right)_a = \{-\infty\}$,
and define $N_{n,l}^{(\O,X,Y)}$ to be the opposite of the composite
\[
\Omega_{\mathfrak C}^{n,0,\lambda_l} \xrightarrow{(\boldsymbol{V}^0)^{\circ n+1}}
\Sigma^{\wr n} \wr \coprod_{\langle \langle l \rangle \rangle} \Sigma_{\mathfrak C} \xrightarrow{(\O,X,\dots,X,Y)}
\Sigma^{\wr n} \wr \V^{op} \xrightarrow{\otimes}
\V^{op}.
\]
The upshot of \S \ref{LCS_SEC}, in particular Corollary \ref{LABIDEN_COR}, is that
$\mbox{$N \left(N^{\circ n} \upsilon \O \amalg_{\mathsf{Set}} \upsilon X^{\amalg_{\mathsf{Set}} 2l +1}\amalg_{\mathsf{Set}} \upsilon Y \right)$}$
is the span
$\mbox{$\Sigma_{\mathfrak C}^{op} \xleftarrow{\mathsf{lr}} (\Omega_{\mathfrak C}^{n,0,\lambda_l})^{op} \xrightarrow{N_{n,l}^{(\O,X,Y)}} \V$}$
and hence, following \eqref{OU EQ1}, we conclude that
\begin{equation}\label{1STRED EQ}
\O[u] \simeq
\mathop{\colim}\limits_{(\Delta \times \Delta)^{op}}
\left(
\mathsf{Lan}_{\left(\Omega_{\mathfrak C}^{n,\lambda_l} \to \Sigma_{\mathfrak C}\right)^{op}} N_{n,l}^{(\O,X,Y)}
\right).
\end{equation}

Moreover, the simplicial operators in the $l$ direction are described by antisymmetric functions $\langle \langle l \rangle \rangle
\to \langle \langle l' \rangle \rangle
$
which are given by \eqref{SIMPLEXP EQ} on non-negative values.

In what follows,
the realization $|\mathcal{C}_{\bullet}|$
of a simplicial object
$\mathcal{C}_{\bullet} \in \mathsf{Cat}^{\Delta^{op}}$
is as in \cite[(A.1)]{BP21}.

\begin{proposition}\label{EXTENTREE PROP}
	The double simplicial realization
	$|\Omega_{\mathfrak C}^{n,\lambda_l}|$,
	denoted $\OC^e$
	and called the \emph{extension tree category},
	has objects the 
	$\{\O,X,Y\}$-labeled trees,
	and arrows the tall maps $f \colon \vect{T} \to \vect{S}$ such that
	\begin{enumerate}[label=(\roman*)]
		\item if $\vect{T}_v$ has a $X$-label, 
		then $\vect{S}_v \in \Sigma_{\mathfrak{C}}$ and
		$\vect{S}_v $ has a $X$-label;
		\item if $\vect{T}_v$ has a $Y$-label, then 
		$\vect{S}_v \in \Sigma_{\mathfrak{C}}$ and
		$\vect{S}_v $ has either a $X$-label or a $Y$-label;
		\item if $\vect{T}_v$ has a $\O$-label, then 
		$\vect{S}_v $ has only $X$ and $\O$ labels.
	\end{enumerate}
\end{proposition}

\begin{proof}
	This is a direct analogue of \cite[Prop. 5.47]{BP21}, and the proof therein carries through without notable changes, so we only sketch the key arguments.
	Firstly, it is straightforward \cite[Rem. 5.41]{BP21} that, for each fixed $l$,
	the realization $|\Omega^{n,\lambda_l}|$
	in the $n$ direction is the category 
	$\Omega^{t,\lambda_l}_{\mathfrak{C}}$
	with objects the $\langle \langle l \rangle \rangle$-labeled trees and  maps the tall label maps which are inert on colors other than $-\infty$/$\O$.
	Moreover, maps 
	$\vect{T} \to \vect{S}$
	in 
	$\Omega^e_{\mathfrak{C}}$
	canonically factor as
	$\vect{T} \to \vect{T'} \to \vect{S}$,
	where the first map is a relabel map (i.e. an underlying isomorphism of trees that just changes labels) and the second map is a label map. 
	Hence, the result follows from the observation that relabel maps in 
	$\Omega^e_{\mathfrak{C}}$
	correspond to objects of  
	$\Omega^{t,\lambda_1}$
	while label maps correspond to maps of
	$\Omega^{t,\lambda_0}$.
\end{proof}

We note that the proof of \cite[Prop. 5.42]{BP21} 
holds when replacing $\Omega_G$ with $\Omega_{\mathfrak C}$,
and we thus 
apply it to
\eqref{1STRED EQ}
(that the ``natural transformation component of differential operators are isomorphisms'' condition for the $n$ direction follows from \eqref{NMONMULTTR EQ} and \eqref{NMONIDTR EQ},
while in the $l$ direction it follows since the associated maps of tuples (cf. Proposition \ref{SPANPIECEJ PROP}) are the identity in each coordinate)
to yield
\begin{equation}\label{2NDRED EQ}
\O[u] \simeq
\mathsf{Lan}_{\left(\Omega_{\mathfrak C}^{e} \to
	\Sigma_{\mathfrak C}\right)^{op}} N^{(\O,X,Y)}.
\end{equation}

The desired filtration \eqref{FILT EQ} will now be obtained by
first replacing $\Omega_{\mathfrak C}^e$ in \eqref{2NDRED EQ} with a suitable subcategory $\widehat{\Omega}_{\mathfrak C}^{e}$,
and then producing a filtration
$\widehat{\Omega}_{\mathfrak C}^{e}[\leq k]$
of 
$\widehat{\Omega}_{\mathfrak C}^{e}$ itself.


\begin{definition}
	Let
	$\widehat{\Omega}_{\mathfrak C}^{e} \hookrightarrow \Omega_{\mathfrak C}^{e}$
	denote the full subcategory of those labeled trees whose underlying tree is alternating
	(cf. Example \ref{ALTMAP EX} and the preceding discussion),
	active nodes are labeled by $\O$, 
	and inert nodes are labeled by $X$ or $Y$.
	
	We write $\boldsymbol{V}^{in}(T)$
	(resp. $\boldsymbol{V}^{Y}(T)$)
	for the subset of inert (resp. $Y$-labeled) vertices.
	
	Letting	$|\vect{T}| = |\boldsymbol{V}^{in}(T)|$,
	$|\vect{T}|_Y = |\boldsymbol{V}^{Y}(T)|$,
	we define subcategories of 
	$\widehat{\Omega}_{\mathfrak C}^{e}$ by:
	\begin{enumerate}[label=(\roman*)]
		\item $\widehat{\Omega}_{\mathfrak C}^{e}[\leq k]$ (resp. $\widehat{\Omega}_{\mathfrak C}^{e}[k]$)
		denotes the full subcategory of those $\vect{T}$ with $|\vect{T}| \leq k$ ($|\vect{T}|=k$);
		\item $\widehat{\Omega}_{\mathfrak C}^{e}[\leq k \setminus Y]$ (resp. $\widehat{\Omega}_{\mathfrak C}^{e}[k \setminus Y]$)
		denotes the further full subcategory of those $\vect{T}$ with $|\vect{T}|_Y \neq k$.
	\end{enumerate}
	Subcategories $\OC^a[\leq k], \OC^a[k]$ of the category $\OC^a$
	of alternating trees are defined similarly.
\end{definition}

The following results follow exactly as in the cited results from 
\cite{BP21} that they adapt.

\begin{lemma}[{cf. \cite[Cor. 5.62, Lemma 5.68]{BP21}}]
	\label{LANINT LEM}
	
	$\widehat\Omega_{\mathfrak C}^e \into 
	\Omega_{\mathfrak C}^e$
	is $\Ran$-initial over $\SC$.
	More explicitly, for a functor
	$N\colon \Omega_{\mathfrak C}^e \to \mathcal{V}$
	with $\mathcal{V}$ complete,
	it is
	$\mathsf{Ran}_{\Omega_{\mathfrak C}^e \to \Sigma_{\mathfrak{C}}}
	N \simeq 
	\mathsf{Ran}_{\widehat\Omega_{\mathfrak C}^e \to \Sigma_{\mathfrak{C}}} N$.	
	
	Similarly, $\widehat\Omega_{\mathfrak C}^e[\leq k-1] \into 
	\widehat\Omega_{\mathfrak C}^e[\leq k \setminus Y]$
	is $\Ran$-initial over $\SC$.
\end{lemma}

\begin{remark}[{cf. \cite[Remark 5.66]{BP21}}]
	\label{OEFIB REM}
	The following diagram
	\begin{equation}
	\begin{tikzcd}
	\widehat\Omega_{\mathfrak C}^e[k \setminus Y] \arrow[rr, hookrightarrow] \arrow[dr]
	&&
	\widehat\Omega_{\mathfrak C}^e[k] \arrow[dl]
	\\
	&
	\Omega_{\mathfrak C}^a[k]
	\end{tikzcd}
	\end{equation}
	is a map of Grothendieck fibrations over $\Omega_{\mathfrak C}^a[k]$
	such that fibers over $\vect{T} \in \Omega_{\mathfrak C}^a[k]$ are the punctured cube and cube categories
	\begin{equation}
	(Y \to X)^{\times \boldsymbol{V}^{in}(T)} - Y^{\times \boldsymbol{V}^{in}(T)},
	\qquad
	(Y \to X)^{\times \boldsymbol{V}^{in}(T)}.
	\end{equation}
	for $\boldsymbol{V}^{in}(T)$ the set of inert vertices.
\end{remark}

We can finally describe the filtration \eqref{FILT EQ}.
\begin{definition}\label{FILTSTAGE DEF}
	Let $\O_k$ denote the left Kan extension, 
	where we abbreviate the restriction of $N^{(\O,X,Y)}$ as $\widetilde{N}$
	\begin{equation}
	\begin{tikzcd}[column sep = 40pt]
	\widehat\Omega_{\mathfrak C}^e[\leq k]^{op}
	\ar{r}{\widetilde N}[swap,name = A]{} \arrow[d, "\mathsf{lr}"']
	&
	\V
	\\
	|[alias = B]|
	\SC^{op}
	\arrow[ur, "\O_k"']
	\arrow[Rightarrow, from = A, to = B, shorten <=0.2cm, shorten >=0.2cm]
	\end{tikzcd}
	\end{equation}
\end{definition}

Since $\widehat\Omega_{\mathfrak C}^e[\leq 0] \simeq \SC$
and the nerve of $\widehat \Omega_{\mathfrak C}^e$ is the union of the nerves of the $\widehat\Omega_{\mathfrak C}^e[\leq k]$
the desired filtration \eqref{FILT EQ} follows.
We now finally prove Lemma \ref{OURE LEM}.

\begin{proof}[Proof of Lemma \ref{OURE LEM}]
	Assume first that $G = \**$ is the trivial group.
	
	The desired filtration \eqref{OUFILRE EQ} is as described by Definition \ref{FILTSTAGE DEF},
	so it remains only to check that the filtration stages
	fit into the pushout diagrams as in \eqref{OUPUSHRE EQ}.
	
	To obtain these, we consider the following diagram,
	where the left square is a pushout at the level of nerves (cf. \cite[(5.75)]{BP21}),
	so that after taking $\mathsf{Lan}$ 
	one obtains the right pushout
	(that the top right corner is $\O_{k-1}$ follows from Lemma \ref{LANINT LEM}; note that the ``ops''
	convert $\mathsf{Ran}$ to $\mathsf{Lan}$)
	\begin{equation}\label{FILTLAN EQ}
	\begin{tikzcd}
	\widehat\Omega_{\mathfrak C}^e[k \setminus Y] \arrow[r] \arrow[d]
	&
	\widehat\Omega_{\mathfrak C}^e[\leq k \setminus Y] \arrow[d]
	&
	\mathsf{Lan}_{\widehat\Omega_{\mathfrak C}^e[k \setminus Y]^{op} \to \SC^{op}} \tilde N \arrow[r] \arrow[d]
	&
	\O_{k-1} \arrow[d]
	\\
	\widehat\Omega_{\mathfrak C}^e[k] \arrow[r]
	&
	\widehat\Omega_{\mathfrak C}^e[\leq k]
	&
	\mathsf{Lan}_{\widehat \Omega_{\mathfrak C}^e[k]^{op} \to \SC^{op}} \tilde N \arrow[r]
	&
	\O_k
	\end{tikzcd}
	\end{equation}
	But now note that Remark \ref{OEFIB REM}
	allows us to iteratively compute the
	$\mathsf{Lan}$
	appearing in \eqref{FILTLAN EQ}
	by first left Kan extending to 
	$\Omega^a_{\mathfrak{C}}[k]$,
	showing that the map between the $\mathsf{Lan}$ terms 
	in \eqref{FILTLAN EQ} is
	\begin{equation}\label{FILTLANFIN EQ}
	\mathsf{Lan}_{(\OC^a[k] \to \SC)^{op}}\left(
	\bigotimes_{v \in V^{ac}(T)}\O(T_v) \otimes
	\mathop{\mathlarger{\mathlarger{\mathlarger{\square}}}}\limits_{v \in V^{in}(T)} u(T_v)
	\right)
	\end{equation}
	which matches \eqref{NKOXY EQ}, finishing the proof of the 
	$G = \**$ case
	(compare with \cite[Prop. 5.77]{BP21}).
	
	For the case of a general group $G$,
	note first that, since all our constructions 
	are compatible with changes of color
	$\varphi \colon \mathfrak{C} \to \mathfrak{D}$,
	the left Kan extension in \eqref{2NDRED EQ}
	is compatible with the $G$-action on $\mathfrak{C}$
	(see the discussion following Remark \ref{CONVER REM}
	concerning cocartesian arrows in $\mathsf{Cat} \downarrow^l \V$).
	Thus, we have the alternative formula
	\begin{equation}\label{3RDRED EQ}
	\O[u] \simeq
	\mathsf{Lan}_{\left(G^{op} \ltimes \Omega_{\mathfrak C}^{e} 
		\to
		G^{op} \ltimes \Sigma_{\mathfrak C}\right)^{op}} N^{(\O,X,Y)}.
	\end{equation}
	Moreover, since the subcategories 
	$\widehat{\Omega}_{\mathfrak{C}}^e$,
	$\widehat{\Omega}_{\mathfrak{C}}^e[\leq k]$,
	$\widehat{\Omega}_{\mathfrak{C}}^e[k]$
	are all compatible with the $G$-action, we can replace these categories with 
	$G^{op} \ltimes \widehat{\Omega}_{\mathfrak{C}}^e$,
	$G^{op} \ltimes \widehat{\Omega}_{\mathfrak{C}}^e[\leq k]$,
	$G^{op} \ltimes \widehat{\Omega}_{\mathfrak{C}}^e[k]$
	in Definition \ref{FILTSTAGE DEF}
	as well as in \eqref{FILTLAN EQ}
	(so that the right square is now in 
	$\mathsf{Sym}_{\mathfrak{C}}(\mathcal{V})^G = \V^{G \ltimes \Sigma^{op}_{\mathfrak{C}}}$).
	And, since the diagram in Remark \ref{OEFIB REM}
	remains a map of Grothendieck fibrations 
	upon applying $G^{op} \ltimes (-)$,
	one likewise has the $G^{op} \ltimes (-)$ analogue of
	\eqref{FILTLANFIN EQ},
	showing that the description in 
	\eqref{NKOXY EQ} holds for a general $G$.
\end{proof}

\begin{remark}[{cf. \cite[Prop. 5.77]{BP21}}]
	\label{FILTPUSH REM}
	
	Similarly to the formulas \eqref{FROPEXP EQ} and \eqref{FROPEXPG EQ} for
	$\mathbb{F}_\mathfrak{C} X$ and       
	$\mathbb{F}_\mathfrak{C}^G X$,
	one can use the description of left Kan extensions
	over maps of groupoids (cf. Remark \ref{CONVER REM})
	to give an explicit pointwise 
	description of the right pushouts in \eqref{FILTLAN EQ}.
	Namely, for each $\mathfrak C$-profile $\vect{C} \in \Sigma_{\mathfrak C}$, 
	one has the following alternative pushout diagrams
	\vspace{-10pt}
	\begin{align}\label{FILTPUSH EQ}
	\begin{tikzcd}[ampersand replacement=\&]
	\mathop{\coprod}\limits_{[\vect{T}] \in \Iso(\vect{C} \downarrow \Omega_{\mathfrak C}^a[k])}
	\left(
	\mathop{\bigotimes}\limits_{v \in \boldsymbol{V}^{ac}(T)} \O(\vect{T}_v) \otimes
	Q^{in}_{\vect{T}}[u]
	\right) \cdot_{\Aut_{\OC} (\vect{T})} \Aut_{\SC}(\vect{C})
	\arrow[r] \arrow[d]
	\&[15pt]
	\O_{k-1}(\vect{C}) \arrow[d]
	\\                  
	\mathop{\coprod}\limits_{[\vect{T}] \in \Iso(\vect{C} \downarrow \Omega_{\mathfrak C}^a[k])}
	\left(
	\mathop{\bigotimes}\limits_{v \in \boldsymbol{V}^{ac}(T)} \O(\vect{T}_v) \otimes
	\mathop{\bigotimes}\limits_{v \in \boldsymbol{V}^{in}(T)} Y(\vect{T}_v)
	\right) \cdot_{\Aut_{\OC}(\vect{T})} \Aut_{\SC}(\vect{C})
	\arrow[r]
	\&
	\O_k(C)
	\end{tikzcd}
	\end{align}
	\begin{align}\label{FILTPUSHG EQ}
	\begin{tikzcd}[ampersand replacement=\&,column sep = 5pt]
	\mathop{\coprod}\limits_{[\vect{T}] \in \Iso(\vect{C} \downarrow G^{op} \ltimes \Omega_{\mathfrak C}^a[k])}
	\left(
	\mathop{\bigotimes}\limits_{v \in \boldsymbol{V}^{ac}(T)} \O(\vect{T}_v) \otimes
	Q^{in}_{\vect{T}}[u]
	\right) \cdot_{\Aut_{G^{op} \ltimes \OC}(\vect{T})} \Aut_{G^{op} \ltimes \SC}(\vect{C})
	\arrow[r] \arrow[d]
	\&
	\O_{k-1}(\vect{C}) \arrow[d]
	\\                  
	\mathop{\coprod}\limits_{[\vect{T}] \in \Iso(\vect{\vect{C}} \downarrow G^{op} \ltimes \Omega_{\mathfrak C}^a[k])}
	\left(
	\mathop{\bigotimes}\limits_{v \in \boldsymbol{V}^{ac}(T)} \O(\vect{T}_v) \otimes
	\mathop{\bigotimes}\limits_{v \in \boldsymbol{V}^{in}(T)} Y(\vect{T}_v)
	\right) \cdot_{\Aut_{G^{op} \ltimes \OC}(\vect{T})} \Aut_{G^{op} \ltimes \SC}(\vect{C})
	\arrow[r]
	\&
	\O_k(\vect{C})
	\end{tikzcd}
	\end{align}
	where $Q^{in}_{\vect{T}}[u]$ denotes the source of the pushout-product map
	\begin{equation}
	\mathop{\mathlarger{\mathlarger{\mathlarger{\square}}}}_{v \in \boldsymbol{V}^{in}(T)} u(\vect{T}_v) \colon Q^{in}_{\vect{T}}[u] \to \bigotimes_{v \in \boldsymbol{V}^{in}(T)} Y(\vect{T}_v).
	\end{equation}
	Moreover, just as in Remark \ref{FROPEXPG REM},
	the contrast between
	\eqref{FILTPUSH EQ} and \eqref{FILTPUSHG EQ}
	is that \eqref{FILTPUSH EQ} has more coproduct summands while 
	\eqref{FILTPUSHG EQ}
	has larger inductions.
\end{remark}

\subsection{Injective change of color and pushouts of operads}
\label{INJCOLCH AP}

This subsection is dedicated to proving
Corollary \ref{LOCALISO_COR} below,
which will be needed 
to prove \cite[Prop. \ref{AC-J_CELL_PROP}]{BP_ACOP} in the sequel,
and deals with pushouts in 
$\mathsf{Op}_{\bullet}(\V)$
whose maps are injective on colors.

First note that, by a variation of the arguments in \eqref{OU EQ1}, \eqref{1STRED EQ}, \eqref{2NDRED EQ},
any pushout 
\[
\begin{tikzcd}
\mathcal{A} \ar{r} \ar{d} & \mathcal{O} \ar{d}
\\
\mathcal{B} \ar{r} & \mathcal{P}
\end{tikzcd}
\]
in $\mathsf{Op}_{\mathfrak{C}}(\V) $ has a description
\begin{align*}
\mathcal{P}
&
\simeq \colim_{[l] \in \Delta^{op}} 
B_l \left( \O, \mathcal{A}, \mathcal{A}, \mathcal{A}, \mathcal{B} \right)
\\
&
\simeq \colim_{[l] \in \Delta^{op},[n] \in \Delta^{op}} 
B_l \left( 
\mathbb F^{\circ n+1} \O, 
\mathbb F^{\circ n+1} \mathcal{A}, 
\mathbb F^{\circ n+1} \mathcal{A}, 
\mathbb F^{\circ n+1} \mathcal{A}, 
\mathbb F^{\circ n+1} \mathcal{B} \right)
\\
&
\simeq \colim_{[l] \in \Delta^{op},[n] \in \Delta^{op}} 
\mathsf{Lan} N \circ \left( N^{\circ n} \upsilon \O 
\amalg_{\mathsf{Set}}
\left( N^{\circ n} \upsilon \mathcal{A}\right)^{\amalg_{\mathsf{Set}} 2l +1}
\amalg_{\mathsf{Set}}
N^{\circ n} \upsilon \mathcal{B} \right)
\\
&	
\simeq
\mathop{\colim}\limits_{(\Delta \times \Delta)^{op}}
\left(
\mathsf{Lan}_{\left(\Omega_{\mathfrak C}^{n,\lambda^a_l} \to \Sigma_{\mathfrak C}\right)^{op}} N_{n,l}^{(\O,\mathcal{A},\mathcal{B})}
\right)
\\
&	
\simeq
\mathsf{Lan}_{\left(\Omega_{\mathfrak C}^{p} \to
	\Sigma_{\mathfrak C}\right)^{op}} N^{(\O,A,B)}
\stepcounter{equation}\tag{\theequation}\label{OU EQ2}
\end{align*}
where the partition $\lambda^a_l$ in the fourth line is the fully-active partition with $\left(\lambda^a_l\right)_a = \langle \langle l \rangle \rangle$
and, in analogy to Proposition \ref{EXTENTREE PROP},
the double realization
$\OC^p \simeq |\Omega_{\mathfrak C}^{n,\lambda^a_l}|$
is the category whose objects are the
$\{\mathcal{O},\mathcal{A},\mathcal{B}\}$-labeled trees,
and whose arrows are tall maps $\vect{T} \to \vect{S}$ such that

\begin{enumerate}[label=(\roman*)]
	\item if $v \in \boldsymbol{V}(T)$ is $\mathcal{A}$-labeled,
	then all vertices in $\vect{S}_{v}$ are $\mathcal{A}$-labeled;
	\item if $v \in \boldsymbol{V}(T)$ is $\mathcal{B}$-labeled,
	then all vertices in $\vect{S}_{v}$ are either $\mathcal{A}$-labeled or $\mathcal{B}$-labeled;
	\item if $v \in \boldsymbol{V}(T)$ is $\mathcal{O}$-labeled, then all vertices in $\vect{S}_{v}$ are either $\mathcal{A}$-labeled or $\mathcal{O}$-labeled.
\end{enumerate}
Moreover, one has the formula
(where $\boldsymbol{V}^{\mathcal{O}}$,
$\boldsymbol{V}^{\mathcal{A}}$,
$\boldsymbol{V}^{\mathcal{B}}$
denote
$\mathcal{O}$,$\mathcal{A}$,$\mathcal{B}$-labeled vertices)
\begin{equation}\label{NBAO EQ}
N^{(\mathcal{O},\mathcal{A},\mathcal{B})}(\vect{T}) = 
\left(\bigotimes_{v \in \boldsymbol{V}^{\mathcal{O}}(T)} \mathcal{O}(\vect{T}_v) \right)
\otimes
\left(\bigotimes_{v \in \boldsymbol{V}^{\mathcal{A}}(T)} \mathcal{A}(\vect{T}_v) \right)
\otimes
\left(\bigotimes_{v \in \boldsymbol{V}^{\mathcal{B}}(T)} \mathcal{B}(\vect{T}_v) \right).
\end{equation}

In the following, recall that if 
$\varphi \colon \mathfrak{C} \hookrightarrow \mathfrak{D}$
is injective, then
$\varphi_! = \check{\varphi}_!$,
cf. Remark \ref{OP_MAP REM}.

\begin{lemma}\label{BASICPUSH LEMMA}
	Let $\varphi \colon \mathfrak{C} \hookrightarrow \mathfrak{D}$ be an injective map of colors, 
	$\mathcal{A} \to \mathcal{B}$ a map in 
	$\mathsf{Op}_{\mathfrak{C}}(\V)$,
	and
	$\mathcal{O} \in \mathsf{Op}_{\mathfrak{D}}(\V)$.
	Then, if the leftmost diagram below in 
	$\mathsf{Op}_{\mathfrak{D}}(\V)$ is a pushout,
	so is the adjoint rightmost diagram in $\mathsf{Op}_{\mathfrak{C}}(\V)$.
	\[
	\begin{tikzcd}
	\varphi_! \mathcal{A} \ar{r} \ar{d} & \mathcal{O} \ar{d}
	&
	\mathcal{A} \ar{r} \ar{d} & \varphi^{\**} \mathcal{O} \ar{d}
	\\
	\varphi_! \mathcal{B} \ar{r} & \mathcal{P}
	&
	\mathcal{B} \ar{r} & \varphi^{\**} \mathcal{P}
	\end{tikzcd}
	\]
\end{lemma}

\begin{proof}
	
	We start by noting that the top composite in the diagram
	\[
	\begin{tikzcd}
	\Omega_{\mathfrak{C}}^{p,op} \ar{r}{\varphi} \ar{d}[swap]{\mathsf{lr}} &
	\Omega_{\mathfrak{D}}^{p,op} 
	\ar{rr}{N^{(\mathcal{O},\varphi_!\mathcal{A},\varphi_!\mathcal{B})}}
	\ar{d}[swap]{\mathsf{lr}} &&
	\mathcal{V}
	\\
	\Sigma_{\mathfrak{C}}^{op} \ar{r}{\varphi} &
	\Sigma_{\mathfrak{D}}^{op} 
	\end{tikzcd}
	\]
	is $N^{(\varphi^{\**}\mathcal{O},\mathcal{A},\mathcal{B})}$
	by the last part of Remark \ref{OP_MAP REM}, 
	so that our intended result is equivalent to showing that 
	following map is an isomorphism.
	\begin{equation}\label{LANISO EQ}
	\mathsf{Lan}_{\Omega_{\mathfrak{C}}^{p,op} \to \Sigma_{\mathfrak{C}}^{op}}
	N^{(\varphi^{\**}\mathcal{O},\mathcal{A},\mathcal{B})}
	\xrightarrow{\simeq}
	\left(
	\mathsf{Lan}_{\Omega_{\mathfrak{D}}^{p,op} \to \Sigma_{\mathfrak{D}}^{op}}
	N^{(\mathcal{O},\varphi_!\mathcal{A},\varphi_!\mathcal{A})}
	\right) \circ \varphi 
	\end{equation}
	To establish \eqref{LANISO EQ}, first let $\widehat{\Omega}^p_{\mathfrak{D}}$
	be the full subcategory of $\Omega^p_{\mathfrak{D}}$
	such that,
	if $v \in \boldsymbol{V}(T)$ is 
	$\mathcal{A}$ or $\mathcal{B}$-labeled,
	then $\vect{T}_v \in \Sigma_{\mathfrak C}$
	(i.e., all edges of $\vect{T}_v$ are colored by $\mathfrak{C}$).
	
	It follows from \eqref{NBAO EQ} that 
	$N^{(\mathcal{O}, \varphi_! \mathcal{A}, \varphi_! \mathcal{B})}(\vect{T}) =
	\emptyset$ whenever $\vect{T} \not \in \widehat{\Omega}^p_{\mathfrak{D}}$,
	and it is straightforward to check that 
	$\widehat{\Omega}^p_{\mathfrak{D}}$
	is a sieve of $\Omega^p_{\mathfrak{D}}$, 
	i.e. that for any map $\vect{T} \to \vect{S}$ 
	in $\Omega^p_{\mathfrak{D}}$ such that 
	$\vect{S} \in \widehat{\Omega}^p_{\mathfrak{D}}$ it is 
	$\vect{T} \in \widehat{\Omega}^p_{\mathfrak{D}}$.
	It then follows that 
	$N^{(\mathcal{O}, \varphi_! \mathcal{A}, \varphi_! \mathcal{B})}$
	is the left Kan extension of its restriction to 
	$\widehat{\Omega}^p_{\mathfrak{D}}$, 
	so we are free to replace
	$\Omega^p_{\mathfrak{D}}$
	with
	$\widehat{\Omega}^p_{\mathfrak{D}}$
	in the rightmost Kan extension in \eqref{LANISO EQ}.
	
	Note next that,
	for each $\vect{C} \in \Sigma_{\mathfrak{C}}$,
	the inclusion of undercategories
	$(\vect{C} \downarrow \Omega^p_{\mathfrak{C}})
	\to
	(\vect{C} \downarrow \widehat{\Omega}^p_{\mathfrak{D}})
	$
	has a left adjoint given by  
	$\vect{T} \mapsto \vect{T} - 
	\boldsymbol{E}_{\mathfrak{D} \setminus \mathfrak{C}}(\vect{T})$,
	where 
	$\boldsymbol{E}_{\mathfrak{D} \setminus \mathfrak{C}}(\vect{T})$
	is the set of edges of $\vect{T}$ whose colors are not in $\mathfrak{C}$
	(that this has a natural vertex labeling follows since all the edges being collapsed must connect $\mathcal{O}$-vertices, so there is no ambiguity as to how to label the vertices of $\vect{T} - \boldsymbol{E}_{\mathfrak{D} \setminus \mathfrak{C}}(\vect{T})$).
	But this shows that the opposite maps 
	$(\vect{C} \downarrow \Omega^p_{\mathfrak{C}})^{op}
	\to
	(\vect{C} \downarrow \widehat{\Omega}^p_{\mathfrak{D}})^{op}$
	are final and,
	since \eqref{LANISO EQ} is 
	computed via colimits over these (opposite) undercategories, the result follows.
\end{proof}

\begin{corollary}\label{FGTPUSH_COR}
	Suppose that $F,G$ on the left pushout diagram below are both injective on colors.
	\[
	\begin{tikzcd}
	\mathcal{A} \ar{r}{F} \ar{d}[swap]{G} & \mathcal{O} \ar{d}
	&
	\mathcal{A} \ar{r} \ar{d}[swap]{G} & F^{\**} \mathcal{O} \ar{d}
	\\
	\mathcal{B} \ar{r}{\bar{F}} & \mathcal{P}
	&
	\mathcal{B} \ar{r} & \bar{F}^{\**} \mathcal{P}
	\end{tikzcd}
	\]
	Then the rightmost diagram is also a pushout diagram.
\end{corollary}

\begin{proof}
	This follows by adding to both $\mathcal{A}$ and $\mathcal{O}$ a disjoint trivial operad on the object set
	$\mathfrak{C}_{\mathcal{B}} \setminus \mathfrak{C}_{\mathcal{A}}$.
	Doing so does not alter the left pushout,
	but the map $G$ now becomes a fixed color map,
	so that Lemma \ref{BASICPUSH LEMMA} can be applied.
\end{proof}

\begin{corollary}\label{LOCALISO_COR}
	Suppose that we have a pushout in $\Op_{\bullet}(\V)$ such that $F,G$ are both injective on colors.
	\[
	\begin{tikzcd}
	\mathcal{A} \arrow[d, "G"'] \arrow[r, "F"]
	&
	\O \arrow[d]
	\\
	\mathcal{B} \ar{r}{\bar{F}}
	&
	\P
	\end{tikzcd}
	\]
	If $F\colon \mathcal{A} \to \mathcal{O}$ is a local isomorphism, then so is $\bar{F} \colon \mathcal{B} \to \P$.
\end{corollary}

\begin{proof}
	The desired claim can be restated as saying that,
	if $\mathcal{A} \to F^{\**} \mathcal{O}$
	is an isomorphism,
	then so is $\mathcal{B} \to \bar{F}^{\**} \mathcal{O}$.
	But this is immediate from Corollary \ref{FGTPUSH_COR}.      
\end{proof}

\providecommand{\bysame}{\leavevmode\hbox to3em{\hrulefill}\thinspace}
\providecommand{\MR}{\relax\ifhmode\unskip\space\fi MR }
\providecommand{\MRhref}[2]{%
  \href{http://www.ams.org/mathscinet-getitem?mr=#1}{#2}
}
\providecommand{\doi}[1]{%
  doi:\href{https://dx.doi.org/#1}{#1}}
\providecommand{\arxiv}[1]{%
  arXiv:\href{https://arxiv.org/abs/#1}{#1}}
\providecommand{\href}[2]{#2}

\makeatletter\@input{labels-OneColor.tex}\makeatother

\end{document}